\documentclass[1pt]{article}
\usepackage{a4wide}

\usepackage{amssymb}
\usepackage{amsmath}
\usepackage{amsthm}
\usepackage{tikz-cd}
\usepackage{mathrsfs}
\usepackage{graphicx} 
\usepackage[font=small]{subcaption}

\usepackage{mathtools}
\usepackage[pdftex]{hyperref}
\usepackage{enumerate}

\usepackage{thmtools}
\usepackage{thm-restate}
\usepackage{verbatim}
\usepackage{romannum}

\usepackage{latexsym}
\usepackage{amscd}
\usepackage{graphics}
\usepackage{amsmath}
\usepackage{amssymb}
\usepackage{esint}
\usepackage{mathrsfs}
\usepackage{bbm}
\usepackage{color} 
\usepackage{accents}

\newtheorem{theorem}{Theorem}[section]
\newtheorem{prop}[theorem]{Proposition}
\newtheorem{definition}[theorem]{Definition}
\newtheorem{cor}[theorem]{Corollary}
\newtheorem{lemma}[theorem]{Lemma}

\newtheorem{remark}[theorem]{Remark}

\renewcommand{\ker}{\operatorname{ker}}
\newcommand{\coker}{\operatorname{coker}}
\renewcommand{\dim}{\operatorname{dim}}

\newcommand{\codim}{\operatorname{codim}}
\newcommand{\ind}{\operatorname{ind}}

\newcommand{\Crit}{\mathrm{Crit}}
\newcommand{\PD}{\mathrm{PD}}
\newcommand{\ev}{\mathrm{ev}}
\newcommand{\m}{\mathfrak{m}}
\newcommand{\aug}{\mathrm{aug}}
\renewcommand{\H}{\mathrm{H}}
\newcommand{\FC}{\mathrm{FC}}
\newcommand{\FH}{\mathrm{FH}}
\newcommand{\SC}{\mathrm{SC}}
\newcommand{\SH}{\mathrm{SH}}
\newcommand{\CZ}{\mathrm{CZ}}
\newcommand{\RFH}{\mathrm{RFH}}
\newcommand{\MC}{\mathrm{MC}}
\newcommand{\MH}{\mathrm{MH}}

\newcommand{\QH}{\mathrm{QH}}
\newcommand{\reg}{\mathrm{reg}}

\newcommand{\Cont}{\mathrm{Cont}}
\newcommand{\p}{\mathfrak{p}}

\newcommand{\R}{\mathbb{R}}
\newcommand{\Z}{\mathbb{Z}}
\newcommand{\N}{\mathbb{N}}
\newcommand{\overbar}[1]{\mkern 1.5mu\overline{\mkern-1.5mu#1\mkern-1.5mu}\mkern 1.5mu}

\newcommand*{\suchthat}[1]{\,\left|\, #1 \right.}

\DeclareFontFamily{U}{mathx}{\hyphenchar\font45}
\DeclareFontShape{U}{mathx}{m}{n}{
      <5> <6> <7> <8> <9> <10>
      <10.95> <12> <14.4> <17.28> <20.74> <24.88>
      mathx10
      }{}
\DeclareSymbolFont{mathx}{U}{mathx}{m}{n}
\DeclareFontSubstitution{U}{mathx}{m}{n}
\DeclareMathAccent{\widecheck}{0}{mathx}{"71}
\DeclareMathAccent{\wideparen}{0}{mathx}{"75}

\numberwithin{equation}{section}

\title{Rabinowitz Floer homology for prequantization bundles and Floer Gysin sequence}
\author{Joonghyun Bae, Jungsoo Kang, Sungho Kim}
\date{} 
\begin{document}
\pagenumbering{arabic}

\maketitle

\begin{abstract}
Let $Y$ be a prequantization bundle over a closed spherically monotone symplectic manifold $\Sigma$. Adapting an idea due to Diogo and Lisi, we study a split version of Rabinowitz Floer homology for $Y$ in the following two settings. First,  $\Sigma$ is a symplectic hyperplane section of  a closed symplectic manifold $X$ satisfying a certain monotonicity condition; in this case, $X \setminus \Sigma$ is a Liouville filling of $Y$.
Second, the minimal Chern number of $\Sigma$ is greater than one, which is the case where the Rabinowitz Floer homology of the symplectization $\mathbb{R} \times Y$ is defined. 
In both cases, we construct a Gysin-type exact sequence connecting the Rabinowitz Floer homology of $X\setminus\Sigma$ or  $\mathbb{R} \times Y$ and the quantum homology of $\Sigma$. As applications, we discuss the invertibility of a symplectic hyperplane section class in quantum homology,  the isotopy problem for fibered Dehn twists, the orderability problem for prequantization bundles, and the existence of translated points. We also provide computational results based on the exact sequence that we construct.

\end{abstract}

\setcounter{tocdepth}{1}
\tableofcontents
\newpage
\section{Introduction}

Let $(X,\omega)$ be a closed symplectic manifold satisfying the following condition. 
\begin{itemize}
	\item[(H)] It is spherically monotone, namely $c_1^{TX} =\tau_X \omega$ on $\pi_2(X)$ for some $\tau_X>0$. 
    Moreover, there is a closed connected symplectic submanifold $\iota:\Sigma\hookrightarrow X$ of codimension two  whose homology class in $X$ is Poincar\'e dual to $[K\omega]$ for some $0<K<\tau_X$, that is,
	$[\Sigma]=\PD([K\omega])$ in  $\H_{\dim X-2}(X;\mathbb{R})$.
\end{itemize}
Here, $c_1^{TX}$ denotes the first Chern class of the tangent bundle $TX \rightarrow X$. 
This condition implies that $[K\omega]$ is integral, i.e.~$[K\omega]$ is in the image of the map $\H^2(X;\Z)\to\H^2(X;\R)$ induced by the inclusion $\Z\hookrightarrow\R$. We denote by $[K\omega]_\Z$ the integral lift of $[K\omega]$ that is Poincar\'e dual to $[\Sigma]$ in integer coefficients, i.e.~$[\Sigma]=\PD([K\omega]_\Z)$ in $\H_{\dim X-2}(X;\mathbb{Z})$.
We also write $\omega_\Sigma:=\iota^*\omega$ and $[K\omega_\Sigma]_\Z:=\iota^*[K\omega]_\Z\in \H^2(\Sigma;\Z)$.
In this case, we can compactify the complement $X\setminus\Sigma$ by adding a circle subbundle $Y$ of the normal bundle of $\Sigma$ in $X$ and obtain a Liouville domain $W$ with boundary $\partial W=Y$.
The induced contact form  on $Y$ is a connection $1$-form $\theta$ of the principal circle bundle $Y\to\Sigma$ with Euler class $e_Y=-[K\omega_\Sigma]_\Z$. The pair $(Y,\theta)$ is called a prequantization bundle over $(\Sigma,[K\omega_\Sigma]_\Z)$.

Another viewpoint to this setting in perspective of Liouville fillings of prequantization bundles is as follows. For a closed symplectic manifold $(\Sigma,\omega_\Sigma)$, we assume that $[K\omega_\Sigma]$ admits an integral lift $[K\omega_\Sigma]_\Z\in\H^2(\Sigma;\Z)$ for some $K>0$. Let $(Y,\theta)$ be a prequantization bundle over $(\Sigma,[K\omega_\Sigma]_\Z)$. The connection $1$-form $\theta$ is a contact form on $Y$ such that Reeb orbits are exactly fiber circles of $Y\to\Sigma$. Suppose that there is a Liouville filling $W$ of $(Y,\theta)$. The symplectic cut of $W$ gives rise to a closed symplectic manifold $(X,\omega)$ such that $(\Sigma,\omega_\Sigma)$ is symplectically embedded in $(X,\omega)$ with $[\Sigma]=\PD([K\omega])$. In Section \ref{sec:Gysin_symplectization}, we discuss sufficient conditions on the filling $W$, which imply the properties in (H). We refer to Corollary \ref{cor:gysin_symplectization} and below for the setting that does not require  Liouville fillings.

  Let us assume the condition (H), and let $W$ be the Liouville domain obtained by compactifying $X\setminus\Sigma$ mentioned above. In this setting, Diogo and Lisi \cite{DL2,DL1} introduced the so called split symplectic homology of $W$, which is isomorphic to the ordinary symplectic homology of $W$, and provided a complete description of the associated boundary operator in terms of pseudo-holomorphic spheres in $\Sigma$ and $X$. 

In this paper, we extend their remarkable results to Rabinowitz Floer homology. More precisely, we investigate  boundary operators and continuation homomorphisms for $\bigvee$-shaped Hamiltonians to introduce a split version of Rabinowitz Floer homology. Then, building on this result, we construct a Gysin-type exact sequence for Rabinowitz Floer homology and study several applications. Here we use $\bigvee$-shaped Hamiltonians as in \cite{CFO} to define the Rabinowitz Floer homology of $(W,\partial W)$, although the original definition in \cite{CF09} using the Rabinowitz action functional would be also well suited for our purpose. 
As studied in \cite{CO1}, the Rabinowitz Floer homology of $(W, \partial W)$ can be understood as the symplectic homology of the trivial cobordism over $\partial W$.
Following \cite{CO1}, we write $\SH_* (\partial W)$ for the Rabinowitz Floer homology of $(W, \partial W)$ rather than $\RFH_* (\partial W)$.

To state our main results, let us denote by $m_W\in\Z$ the nonnegative generator of the kernel of the map $\Z\cong \pi_1(Y_p)\to\pi_1(W)$ induced by the inclusion of the fiber circle  $Y_p$ of $Y\to\Sigma$ over $p\in\Sigma$ into $W$. We define the ring
\begin{equation*} \label{eq:Novikov}
\Lambda : = \begin{cases}
	\Z & \quad \text{if }\; m_W=0,\\[.5ex]
	\mathbb Z [ T, T^{-1} ] & \quad \text{if }\;m_W\neq 0.
\end{cases} 
\end{equation*}
 In the latter case, when $\Lambda$ is a Laurent polynomial ring, the degree of the formal variable $T$ is given by $\frac{2(\tau_X - K)}{K} m_W\in\Z$, the Conley-Zehnder index of the $m_W$-iterate of a simple periodic Reeb orbit on $Y$. In the case that $W$ is a Weinstein domain,   the degree of $T$ is equal to 
the positive generator of $2(c_1^{TX}-K\omega)(\pi_2(X))\subset\Z$ if $\dim W\geq 4$ and to twice the minimal Chern number of $\Sigma$  if $\dim W\geq 6$, see Section \ref{sec:index}.
We denote by $\H_*(\Sigma;\Lambda)\cong \H_*(\Sigma;\Z)\otimes_\Z\Lambda$ the singular homology of $\Sigma$ with coefficients in $\Lambda$.

\begin{theorem}\label{thm:gysin}
We assume (H). Then the Rabinowitz Floer homology $\SH_*(\partial W)$  has a $\Lambda$-module structure and fits into a $\Lambda$-module long exact sequence 
\[
    \cdots \longrightarrow \H_{*+2}(\Sigma;\Lambda) \stackrel{\delta_*} {\longrightarrow} \H_* (\Sigma;\Lambda) \longrightarrow \SH_{*-\frac{\dim \Sigma}{2}+1} (\partial W) \longrightarrow  \H_{*+1} (\Sigma;\Lambda) \stackrel{\delta_*}\longrightarrow \cdots,
\]
where $\delta_*$ is the sum of two $\Lambda$-linear maps $\delta^\Sigma_*$ and $\delta^{X,\Sigma}_*$ with the following properties.
\begin{enumerate}[(a)] 
	\item The map $\delta^\Sigma_*$ is the quantum cap product with $-[K\omega_\Sigma]_\Z$. 

	\item The map $\delta^{X,\Sigma}_*$ is nonzero only if $m_W\neq 0$, $m_W\Z=K\omega(\pi_2(X))$, and  $\deg T=2$. In this case, $\delta^{X,\Sigma}_*$ is the multiplication by $c_{X,\Sigma}T^{-1}$, where $c_{X,\Sigma}\in\Z$ is divisible by $m_W$.

\end{enumerate} 
\end{theorem}

We refer to \cite{MS17} for a detailed account of the quantum product structure. The definition of the quantum cap product with $[-K\omega_\Sigma]_\Z$ using a cycle representing the Poincar\'e dual of $[-K\omega_\Sigma]_\Z$ can be found in the proof of Lemma \ref{lemma:quantum_cap}.

The coefficient $c_{X,\Sigma}$ in $\delta_*^{X,\Sigma}$ is given by counting certain pseudo-holomorphic spheres in $X$ intersecting $\Sigma$. We refer to \eqref{eq:delta_c} for the explicit description of $c_{X,\Sigma}$.  We point out that the map $\delta_*$ on $\H_* (\Sigma;\Lambda)$ depends on $X$ as well as $\Sigma$ in general. Theorem \ref{thm:gysin} is proved in Section \ref{sec:Gysin_complement}.  

\begin{remark}\label{rem:del^c=0}
	If $\delta_*^{X,\Sigma}\neq0$, then the minimal Chern number of $X$, i.e.~the positive generator of $c_1^{TX}(\pi_2(X))\subset \Z$, equals $m_X+1$, where $m_X\in\Z$ is the positive generator of $K\omega(\pi_2(X))\subset\Z$.  This follows from that $m_W=m_X$ and $\frac{\tau_X-K}{K}m_X=1$ due to Theorem \ref{thm:gysin}.(b) and the minimal Chern number of $X$ is $\frac{\tau_X m_X}{K}$. 
	Therefore, we have $\delta_*^{X,\Sigma}=0$ when the minimal Chern number of $X$ is greater than $m_X+1$.	 This is the case if the map $\pi_2(\Sigma)\to\pi_2(X)$ induced by the inclusion is surjective and the minimal Chern number of $\Sigma$ is greater than one, see Lemma \ref{nonexistence}. Note that the map $\pi_2(\Sigma)\to\pi_2(X)$ is surjective if $W$ is Weinstein with $\dim W\geq 6$ or if $W$ is subcritical Weinstein with $\dim W\geq 4$.
\end{remark}

\begin{remark}\label{rem:K_large}
	If $[\omega]$ is integral and $K\in\mathbb N$ is sufficiently large, then, for any integral lift $[\omega]_\Z\in\H^2(X;\Z)$, there exists a closed codimension two symplectic submanifold $\Sigma$ in $X$  such that $\PD([\Sigma])=K[\omega]_\Z$ and the complement $X\setminus\Sigma$ is Weinstein, see \cite{Don96,Gir17}. Although we focus on the case $0<K<\tau_X$, in the case that $K$ is sufficiently greater than $\tau_X$, the exact sequence in Theorem \ref{thm:gysin} still exists and recovers the ordinary Gysin exact sequence. We refer to Proposition \ref{prop:K_large} for a precise statement.
\end{remark}

In what follows, we present  corollaries of Theorem \ref{thm:gysin}. The proofs of the corollaries  are given in Section \ref{sec:pf_cor}. 
An immediate consequence of the exact sequence in Theorem \ref{thm:gysin} is that the following equivalence holds:
\begin{equation}\label{eq:iff}
\delta_*:\H_{*+2}(\Sigma;\Lambda)\to\H_*(\Sigma;\Lambda) \text{ is an isomorphism if and only if } \SH_*(\partial W)=0.
\end{equation}
In fact, for example in the case that $W$ is a Weinstein domain with $c_1^{TW}|_{\pi_2(W)}=0$ and $\SH_*(\partial W)=0$, our standing assumption in (H) that $(X,\omega)$ is spherically monotone with constant $\tau_X>K$ automatically holds. Thus this assumption can be dropped in the ``if'' direction of \eqref{eq:iff} as stated in the corollary below.

\begin{cor}\label{cor:invertibility}
Let $(\Sigma,\omega_\Sigma)$ be a closed connected symplectic submanifold of a closed symplectic manifold $(X,\omega)$ satisfying $\PD([\Sigma])=[K\omega]\in\H^2(X;\R)$ for some $K>0$. Let $W$ be a Liouville domain whose interior is symplectomorphic to $X\setminus\Sigma$. Assume that $W$ meets one of the following conditions.
\begin{enumerate}[(i)]
\item The first Chern class $c_1^{TW} \in \H^2(W; \Z)$ is torsion.
\item The map $\pi_1(\partial W) \to \pi_1(W)$ induced by the inclusion is surjective and $c_1^{TW}|_{\pi_2(W)}=0$.
\end{enumerate}
Assume in addition $\SH_*(\partial W)=0$. Then the map $\delta_*$ in Theorem \ref{thm:gysin} is an isomorphism, and in turn the following hold.
\begin{enumerate}[(a)]
	\item The cohomology class $[K\omega]$ is primitive, i.e.~$K\omega(\H_2(X;\Z))=\Z$.
	\item If $\delta^{X,\Sigma}_*=0$, then $[K\omega_{\Sigma}]$ is invertible in the quantum cohomology $\QH^*(\Sigma;\Z)$.
\end{enumerate}
\end{cor}

Let $\H^*_\mathrm{free}(X;\Z)$ denote the quotient of $\H^*(X;\Z)$ by the torsion submodule. 
There is a unique class in $\H^*_\mathrm{free}(X;\Z)$ corresponding to $[K\omega]\in \H^*(X;\R)$. 
We abuse notation and denote it again by $[K\omega]$, likewise for $[K\omega_\Sigma]$.
Then $[K\omega]$ being primitive is equivalent to saying that if $[K\omega]=m\mu$ for some $\mu\in \H^*_\mathrm{free}(X;\Z)$ and $m\in\Z$, then $m\in\{-1,1\}$. 

Let $\Lambda_\Sigma$ be the Novikov ring for $\Sigma$, namely
\begin{equation}\label{eq:Novikov_Sigma}
\Lambda_\Sigma : = \begin{cases}
	\Z & \quad \text{if }\; \omega_\Sigma|_{\pi_2(\Sigma)}=0,\\[.5ex]
	\mathbb Z [ T, T^{-1} ] & \quad \text{if }\;\omega_\Sigma|_{\pi_2(\Sigma)}\neq0,
\end{cases} 
\end{equation}
where the degree of $T$ is twice the minimal Chern number of $\Sigma$. 
The quantum cohomology $\QH^*(\Sigma;\Z)$ equals $\H^*_\mathrm{free}(\Sigma;\Z)\otimes_\Z\Lambda_\Sigma$  as $\Lambda_{\Sigma}$-modules.
In  Corollary \ref{cor:invertibility}.(b), we identify $[K\omega_\Sigma]$ with its image under the inclusion $\H^*_\mathrm{free}(\Sigma;\Z)\hookrightarrow\QH^*(\Sigma;\Z)$ given by $\alpha\mapsto\alpha\otimes 1$. 

The condition $\SH_*(\partial W)=0$ is equivalent to vanishing of the symplectic homology of $W$, see \cite[Theorem 13.3]{Rit} and \cite[Corollary 9.9]{CO1}. For example, this is the case if $W$ is a subcritical Weinstein domain, see  \cite{Cie}. Therefore, Corollary \ref{cor:invertibility} can be applied to subcritical Weinstein domains $W$ with $c_1^{TW}|_{\pi_2(W)}=0$.

\begin{remark}
	For a closed  K\"ahler manifold $(X,\omega,J)$ (not necessarily monotone), Biran and Cieliebak \cite[p.751]{BC01} conjectured the following. If $X\setminus\Sigma$ admits a subcritical Stein structure for a complex hypersurface $\Sigma$ with $\PD([\Sigma])=[K\omega]$, then $[K\omega]$ is primitive. Under the assumption that $c_1^{TW}=0$, a sketch of the proof of this conjecture using contact homology was given in \cite[p.661]{EGH00} and some details were carried out in \cite[Proposition 4.2]{He13}. Recently this conjecture (actually a more general statement) was proved in \cite{GSZ21} using a purely topological argument. Corollary \ref{cor:invertibility}.(a) can be thought of as a symplectic answer to the conjecture under the additional assumption $c_1^{TW}|_{\pi_2(W)}=0$.

	Biran and Jerby further asked the following question in \cite[Section 10]{BJ13}. If $X\setminus\Sigma$ is subcritical, then is the class $[K\omega_\Sigma]$ in $\QH^*(\Sigma;\Z)$  invertible? Their main result \cite[Theorem 4.2]{BJ13}, which relies on  the Seidel representation, answers the question affirmatively for a certain class of projective manifolds $X$ and the hyperplane section $\Sigma\subset X$. Results in \cite[Section 15]{BK} strongly hint that this question can be answered using a Floer Gysin exact sequence. The proof of Corollary \ref{cor:invertibility}.(b) is one way of realizing this idea. 
\end{remark}

\begin{remark}
	Another conjecture in \cite[p.751]{BC01} says that if $X\setminus\Sigma$ is subcritical Stein, then $X$ is (symplectically) uniruled. In the setting of (H), this should follow  from \cite[Theorem 9.1]{DL1} asserting that a Floer cylinder annihilating the unit of the symplectic homology of $X\setminus\Sigma$ corresponds to a pseudo-holomorphic sphere in $X$ passing through a prescribed point representing the unit. See also \cite[Theorem 1.7]{He13} for a proof using contact homology.
\end{remark}

Adapting arguments from \cite{BG07,CDvK14}, we can apply Theorem \ref{thm:gysin} to the study of the isotopy class of a right-handed fibered Dehn twist.
See also \cite{Sei97, Ulj17} for earlier results in this direction.
We denote by $\mathrm{Symp}_c(W)$ the group of symplectomorphisms on $W$ with support on the interior of $W$.  

\begin{cor}\label{cor:dehn_twist}
We assume (H) and that $W$ is a Weinstein domain. Let $\tau$ be a right-handed fibered Dehn twist on $W$ along $\partial W$. If $\tau$ is trivial  in $\pi_0(\mathrm{Symp}_c(W))$, then $[K\omega]$ is invertible in the quantum cohomology $\QH^*(X;\Z)$ of $X$.
\end{cor}

In the following two corollaries, we only consider a closed  symplectic manifold $(\Sigma,\omega_\Sigma)$ with integral $[K\omega_\Sigma]$  for some $K>0$ and forget  an ambient symplectic manifold $(X,\omega)$. Let $[K\omega_\Sigma]_\Z\in\H^2(\Sigma;\Z)$ be an integral lift of $[K\omega_\Sigma]$, and let $Y$ be a prequantization bundle over $\Sigma$ with Euler class $e_Y=-[K\omega_\Sigma]_\Z$. Note that we do not assume that a Liouville filling $W$ of $Y$ exists. However, if $(\Sigma,\omega_\Sigma)$ satisfies a certain condition (a bit stronger than semipositivity), the Rabinowitz Floer homology of the symplectization $\R\times Y$ is well-defined, and we again have the following Floer Gysin exact sequence.

\begin{cor}\label{cor:gysin_symplectization}
Let $Y$ and $(\Sigma,\omega_\Sigma)$ be as above.
	We assume that for every $A\in\pi_2(\Sigma)$ with $\omega_\Sigma(A)>0$, it holds either $c_1^{T\Sigma}(A)\geq2$ or $c_1^{T\Sigma}(A)\leq 2-\frac{1}{2}\dim \Sigma$.  
	Then the Rabinowitz Floer homology $\SH_*(Y)$ of $\R\times Y$ is defined, carries a $\Lambda_{\Sigma}$-module structure, and fits into a $\Lambda_{\Sigma}$-module long exact sequence
\[
    \cdots \longrightarrow \H_{*+2}(\Sigma;\Lambda_{\Sigma}) \stackrel{\delta^\Sigma_*} {\longrightarrow} \H_* (\Sigma;\Lambda_{\Sigma}) \longrightarrow \SH_{*-\frac{\dim \Sigma}{2}+1} (Y) \longrightarrow  \H_{*+1} (\Sigma;\Lambda_{\Sigma}) \stackrel{\delta^\Sigma_*}\longrightarrow \cdots,
\]
where $\delta^\Sigma_*$ is the quantum cap product with $e_Y$.
In the case that $\omega_\Sigma|_{\pi_2(\Sigma)}=0$ or $c_1^{T\Sigma}(A)\leq 2-\frac{1}{2}\dim\Sigma$ for all $A\in\pi_2(\Sigma)$ with $\omega_\Sigma(A)>0$, the long exact sequence  recovers the ordinary Gysin exact sequence tensored with $\Lambda_{\Sigma}$.
\end{cor}

\begin{remark}
If there is a Liouville filling $W$ of $Y$ such that the map $\pi_1(Y)\to\pi_1(W)$ induced by the inclusion is injective and $c_1^{TW}|_{\pi_2(W)}=0$, then $\SH_*(\partial W)\cong\SH_*(Y)$. In this case, the exact sequence in Corollary \ref{cor:gysin_symplectization} coincides with the one in Theorem \ref{thm:gysin} with $\delta^{X,\Sigma}_* = 0$, see Remark \ref{rem:indep_filling} for a further discussion.
\end{remark}

 We denote by $\xi$ the contact structure on $Y$ and by $\Cont_0(Y,\xi)$ the identity component of the group of contactomorphisms on $(Y,\xi)$.
\begin{cor}\label{cor:orderability}
	Let $Y$ and $(\Sigma,\omega_\Sigma)$ be as in Corollary \ref{cor:gysin_symplectization}. Suppose that the Euler class $e_Y\in\H^2(\Sigma;\Z)$ is not a primitive class. Then, the following hold.
	\begin{enumerate}[(a)]
		\item The universal cover $\widetilde{\Cont}_0(Y,\xi)$ of $\Cont_0(Y,\xi)$ is orderable.
		\item Every $\varphi\in\Cont_0(Y,\xi)$ has a translated point with respect to any contact form supporting $\xi$. 
	\end{enumerate}
\end{cor}

We should mention that \cite{ASZ16} first studied the contact topology of $Y$ using a Gysin-type exact sequence.
They also proved results in Corollary \ref{cor:orderability}.

\begin{remark}
The notion of orderability was introduced by Eliashberg and Polterovich \cite{EP00} and further studied in   \cite{Bh01,EKP06,Mil08, CN10,San11,AF12, AFM15,Wei15,BZ15,KvK16, CP16,CFP17,CCDR19,Li20,Z20,GKPS21,AK23,AA23}. 

 The notion of translated points was introduced by Sandon \cite{San12, San13}. Translated points can be understood as leafwise intersection points of a Hamiltonian lift of a contactomorphism. We also refer to \cite{AH16,She17,AM18,MN18,MU19,Ter21, Al22a, Al22b, Al22b,Al22b,UZ22,Can22a,DUZ23,Can23,AK23,Oh23a,Oh23b}  for earlier results on translated points. 
\end{remark}

\begin{remark}
 Gysin-type exact sequences  were constructed in \cite{Per08,BK} for different types of Floer homologies, and recently in \cite{AK23} for the Rabinowitz Floer homology of negative line bundles. Moreover, it was pointed out in \cite[Remark 9.8]{DL1} that their results would lead to a Gysin-type exact sequence for positive symplectic homology. 

Although settings are slightly different, the exact sequence in Corollary \ref{cor:gysin_symplectization} is isomorphic to the one in \cite[Theorem 1.1]{AK23} and presumably also to the one in \cite[Section 14]{BK}. Corollary \ref{cor:orderability} generalizes \cite[Theorem 1.1.(e)]{AK23}.
\end{remark}

Finally in Section \ref{sec:computations}, we compute Rabinowitz Floer homology for some examples using the exact sequence in Theorem \ref{thm:gysin}. 
Our examples include the complements of projective hypersurfaces of degree $1 \leq d \leq n + 1$ in $\mathbb{CP}^{n+1}$. In particular, we see that the Rabinowitz Floer homology vanishes for degree $n + 1$ hypersurfaces in $\mathbb{Q}$-coefficients.


\section{Setup and preliminaries}

\subsection{Liouville structure on $X\setminus\Sigma$}\label{sec:setting}
As mentioned in the introduction, we assume that a closed symplectic manifold $(X,\omega)$ and a codimension two closed connected symplectic submanifold $\iota:\Sigma\hookrightarrow X$ with $\omega_\Sigma=\iota^*\omega$ satisfy the following. There exist $0<K<\tau_X$ such that
\[
    c_1^{TX}(B)=\tau_X \omega(B)\qquad \forall B \in \pi_2(X),\qquad [\Sigma]=\PD([K\omega]_\Z)\in \H_{\dim X-2}(X;\mathbb{Z}),
\]
where $[K\omega]_\Z$ is an integral lift of $[K\omega]\in\H_{\dim X-2}(X;\mathbb{R})$. Note that $K\omega(B)$ agrees with the intersection number between $\Sigma$ and a representative of $B$, and we simply write 
\[
K\omega(B)=\Sigma\cdot B,\qquad B \in \pi_2(X).  
\]
 Another immediate consequence of our hypothesis is that $(\Sigma,\omega_\Sigma)$ is also spherically monotone,
\begin{equation}\label{monoto}
    c_1^{T\Sigma}(A) = (\tau_X - K) \, \omega_\Sigma(A) \qquad \forall A\in\pi_2 (\Sigma).
\end{equation} 
Indeed, the symplectic normal bundle of $\Sigma$,  
\[
\pi_E:E\longrightarrow\Sigma,
\] 
with respect to $\iota:\Sigma\hookrightarrow X$ has the first Chern class $c_1^{E}=[K\omega_\Sigma]_\Z:=\iota^*[K\omega]_\Z\in\H^2(\Sigma;\Z)$. Now \eqref{monoto} follows from $\iota^*c_1^{TX}=c_1^{T\Sigma}+c_1^{E}$.

We endow $\pi_{E}:E\to\Sigma$ with a Hermitian structure. We denote by $\rho=|\cdot|$ the norm induced by the Hermitian structure.
There is a connection 1-form $\Theta$ on $E \setminus \Sigma$, where $\Sigma$ is identified with the zero-section of $E$, such that
\begin{equation*}\label{connection}
    \Theta(\partial_t) = \frac{1}{2\pi}, \qquad d\Theta = (\pi_E)^*(-K\omega_\Sigma) \;\textrm{ on }\; E\setminus\Sigma.
\end{equation*}
Here $\partial_t$ denotes the fundamental vector field generated by the $U(1)$-action on $E$.
We equip $E$ with a non-exact symplectic form $\omega_E$ defined by 
\begin{equation*}
    \omega_E := \frac{2\rho e^{-\rho^2}}{K} d\rho \wedge \Theta +e^{-\rho^2}\pi^*_E\,\omega_\Sigma,
\end{equation*}
which equals $d(-\frac{e^{-\rho^2}}{K}\Theta)$ on $E\setminus\Sigma$.
The Liouville vector field $Z_E=-\frac{1}{2\rho}\frac{\partial}{\partial \rho}$ on $E \setminus \Sigma$ satisfying $\iota_{Z_E} \omega_E = -\frac{e^{-\rho^2}}{K}\Theta$ points toward the zero-section.
According to \cite[Lemma 2.1]{Ops13} or \cite[Lemma 2.2]{DL1}, there exist an open disk subbundle $\mathcal U\subset E$ of radius $\rho_0>0$, a $1$-form $\lambda$ on $X\setminus\Sigma$ satisfying $d\lambda=\omega|_{X\setminus\Sigma}$, and an embedding 
\begin{equation}\label{tubular}
    \varphi:\mathcal U \longrightarrow  X,
\end{equation}
which is the identity map on $\Sigma$ and satisfies $\varphi^*\lambda= -\frac{e^{-\rho^2}}{K}\Theta$ on $\mathcal U\setminus\Sigma$.

We also consider the circle bundle
\[
\pi_Y  :Y := \left\{ v \in E \, \middle | \, \rho(v) = 1 \right \} \longrightarrow \Sigma,
\]
with $\pi_Y := \pi_E |_Y$.
Then $\alpha := -\frac{1}{K}\Theta|_{Y}$ is a contact form and the associated Reeb vector field $R$ has $\frac{1}{K}$-periodic flow such that each orbit is contained in a fiber of $Y\to\Sigma$.
In fact, $R$ extends to $E$, denoted again by $R$, in such a way that $-\frac{1}{2\pi K} R$ coincides with the fundamental vector field generated by the $U(1)$-action on $E$.
As a circle bundle, we orient fibers of $Y$ by $R$.
This implies $e_Y = -[K\omega_\Sigma]_{\mathbb{Z}}$, where $e_Y$ denotes the Euler class of $Y$.

 Now, we consider the exact symplectomorphism
\begin{equation*}\label{psi_1}
	\psi_1:\big(E\setminus \Sigma,d(-\tfrac{e^{-\rho^2}}{K}\Theta)\big) \rightarrow \big((-\infty,0)\times Y,d(e^r\alpha)\big),\qquad     v \mapsto (-|v|^2, v/{|v|}),
\end{equation*}
 where $r$ denotes the coordinate of $(-\infty, 0) \subset \mathbb{R}$. 
We have $\psi_1(\mathcal U \setminus \Sigma)=(-\rho_0^2,0) \times Y$  and $d\psi_1 Z_E = \partial_r$.
Therefore, using $\psi_1$ and $\varphi$ defined in \eqref{tubular}, we attach  $([0,\infty) \times Y, d(e^r\alpha))$ to $X\setminus\Sigma$ and obtain a Liouville manifold
\begin{equation*}\label{completion}
    \widehat{W} := \big(X\setminus \Sigma\big)  \cup \big([0,\infty) \times Y\big),
\end{equation*}
where we also extend the 1-form $\lambda$ to $\widehat{W}$ by setting $\lambda=e^r\alpha$ on $[0, \infty) \times Y$. 
The Liouville flow of $\lambda$ provides an embedding of the whole symplectization $\mathbb R\times Y$  into $\widehat{W}$, which we identify with its image and write 
\[
\mathbb{R}\times Y\subset \widehat{W}.
\] 
We also denote by $W$ the closure of $X\setminus\Sigma$ in $\widehat{W}$. That is, $(W,\lambda)$ is a Liouville domain whose boundary $\partial W$ is diffeomorphic to $Y$ satisfying $\lambda|_{\partial W} = \alpha$. Conversely, it is also possible to recover $(X,\omega)$ from $(W,\lambda)$, see Section \ref{sec:Gysin_symplectization} below.

\subsection{Morse functions on $\Sigma$ and $Y$}\label{sec:Morse}

For a smooth Morse function $f_\Sigma:\Sigma \rightarrow \mathbb{R}$, 
we denote by $\Crit f_\Sigma$ the set of critical points of $f_\Sigma$ and by $\ind_{f_\Sigma} (p)$ the Morse index of $f_\Sigma$ at a critical point $p$.
Let $Z_\Sigma$ be a smooth gradient-like vector field for $f_\Sigma$ on $\Sigma$, i.e.~$\frac{1}{c}|df_\Sigma|^2 \leq Z_\Sigma(f_\Sigma) \leq c|df_\Sigma|^2$ for some Riemannian metric on $\Sigma$ and $c>0$.
Writing $\varphi^t_{Z_\Sigma}$ for the flow of $Z_\Sigma$, we define the unstable and stable manifolds at $p \in \Crit f_\Sigma$ with respect to $Z_\Sigma$ by
\begin{equation*}
\begin{split}
	    W^u_{Z_\Sigma}(p) &:= \big\{ q \in \Sigma \suchthat \displaystyle \lim_{t\to -\infty} \varphi^t_{Z_\Sigma}(q) = p \big\},\\
	    W^s_{Z_\Sigma}(p) &:= \big\{ q \in \Sigma \suchthat \displaystyle \lim_{t\to +\infty} \varphi^t_{Z_\Sigma}(q) = p \big\},
\end{split}
\end{equation*}
respectively. 
Then, we have 
$\dim W^u_{Z_\Sigma}(p) = \dim \Sigma - \ind_{f_\Sigma}(p)$ and $\dim W^s_{Z_\Sigma}(p) = \ind_{f_\Sigma}(p)$.
We assume that $(f_\Sigma,Z_\Sigma)$ is a Morse-Smale pair, namely all unstable and stable manifolds intersect transversely.

We also choose a Morse-Smale pair $(f_Y,Z_Y)$ of a smooth Morse function $f_Y:Y \rightarrow \mathbb{R}$ and a smooth gradient-like vector field $Z_Y$ for $f_Y$ such that $f_Y$ has exactly two critical points on $\pi_Y^{-1}(p)$ for each $p\in\Crit f_\Sigma$ and $Z_Y -  Z_\Sigma^\mathrm{h}$, where $ Z_\Sigma^\mathrm{h}$ denotes the horizontal lift of $Z_\Sigma$ to $Y$ with respect to a certain connection 1-form,  is tangent to the fibers of $\pi_Y:Y \rightarrow \Sigma$. 
For the construction of such $(f_Y, Z_Y)$, we refer to \cite{Oan} or Section \ref{sec:Morse_Gysin}.
We denote by $\hat{p}$ and $\check{p}$ the maximum and minimum point of $f_Y|_{\pi_Y^{-1}(p)}$ respectively. 
Hence, we have
\begin{equation}\label{crit_lift}
    \Crit f_Y = \bigcup_{ p \in \Crit f_\Sigma} \left \{ \, \hat{p}, \check{p} \, \right \},\qquad \text{ind}_{f_Y}(\tilde{p})=\text{ind}_{f_\Sigma}(p)+i(\tilde{p}) \,\textrm{ with }\, i(\hat p)=1,\;i(\check p)=0.
\end{equation}

Now, we explain our orientation convention for unstable and stable manifolds.
For each $p \in \Crit {f_\Sigma}$, we choose an arbitrary orientation for $W^s_{Z_\Sigma}(p)$. 
This induces an orientation on $W^u_{Z_\Sigma}(p)$ by requiring the isomorphism
\begin{equation}\label{eq:stable_splitting}
 T_p\Sigma \cong   T_p W^u_{Z_\Sigma}(p) \oplus T_p W^s_{Z_\Sigma}(p) 	
\end{equation}
to be orientation-preserving, where the orientation on $\Sigma$ is induced by $\omega_\Sigma$.

We orient $Y$ such that 
\begin{equation}\label{eq:ori_Y}
T_{y}Y\cong \mathbb{R} R_y\oplus  T_{{\pi_Y}(y)}\Sigma	
\end{equation}
preserves orientations where $R$ denotes the Reeb vector field on $Y$ and the isomorphism is given by mapping the horizontal subspace of $T_yY$ to $T_{\pi_Y(y)}\Sigma$ by $d \pi_Y$. 
We denote the unstable and stable manifolds with respect to $Z_Y$ by $W^u_{Z_Y}(\tilde{p})$ and $W^s_{Z_Y}(\tilde{p})$ respectively for $\tilde{p} \in \Crit f_Y$.
Since $Z_Y-Z_\Sigma^{\mathrm h}$ is tangent to the fibers, the projection $\pi_Y$ maps (un)stable manifolds in $Y$ to (un)stable manifolds in $\Sigma$:
\begin{align}
&W^u_{Z_Y}(\hat{p})\to W^u_{Z_\Sigma}(p), \quad W^s_{Z_Y}(\hat{p})\to W^s_{Z_\Sigma}(p), \label{eq:Morse_proj_hat}\\[.5ex]
&W^u_{Z_Y}(\check{p}) \to W^u_{Z_\Sigma}(p), \quad W^s_{Z_Y}(\check{p})\to W^s_{Z_\Sigma}(p). \label{eq:Morse_proj_check}
\end{align}

We orient $W^s_{Z_Y}(\hat{p})$ for every $\hat p \in \Crit f_Y$ so that the second map in \eqref{eq:Morse_proj_hat} induces an orientation-preserving isomorphism $T_{y} W^s_{Z_Y}(\hat{p}) \cong \mathbb{R} R_{y} \oplus T_{\pi_Y(y)}W^s_{Z_\Sigma}(p)$ for  $y\in W^s_{Z_Y}(\hat{p})$.
This together with the orientation on $Y$ endows $W^u_{Z_Y}(\hat{p})$ with an orientation determined by requiring the splitting  
\begin{equation}\label{eq:stable_splitting_Y}
	 T_{\hat p}Y \cong   T_{\hat p} W^u_{Z_Y}(\hat p) \oplus T_{\hat p} W^s_{Z_Y}(\hat p) 
\end{equation}
to be orientation-preserving. 
One can readily see that then the first map in \eqref{eq:Morse_proj_hat} is orientation-preserving if and only if $\dim\Sigma-\ind_{f_\Sigma}(p)$ is even. 
Next, we orient $W^s_{Z_Y}(\check{p})$ for every $\check p \in \Crit f_Y$ so that the second map in \eqref{eq:Morse_proj_check} is orientation-preserving.
This induces an orientation on $W^u_{Z_Y} (\check{p})$ by requiring \eqref{eq:stable_splitting_Y} with $\hat p$ replaced by $\check p$ to be orientation-preserving.
Then our convention \eqref{eq:stable_splitting} yields that the isomorphism $T_yW^u_{Z_Y}(\check{p})\cong \mathbb{R} R_y\oplus T_{\pi_Y(y)}W^u_{Z_\Sigma}(p)$  given by the first map in \eqref{eq:Morse_proj_check} is orientation-preserving as well.

\subsection{Reeb orbits and indices}\label{sec:index}
Let $Y_p$ denote the fiber of $Y\to\Sigma$ over $p\in\Sigma$ oriented by $R$.
We define the \textit{multiplicity} $\m (x)\in\mathbb{Z}$ of a continuous loop $x:\mathbb R/T\mathbb Z \rightarrow Y_p$ for $T>0$ by the degree of $x$ viewed as a map between oriented circles: 
\[
\m (x) :=\mathrm{deg}\big(x:\mathbb R/T\mathbb Z \rightarrow Y_p\big).
\]
Since the flow of $R$ has period $1/K$, it follows that a Reeb orbit $\gamma:\mathbb R/T\mathbb Z\to Y$ of period $T$ has $\m(\gamma) = TK > 0$ and is the $TK$-th iterate of the underlying simple Reeb orbit, which winds a fiber of $Y$ once.

We consider the map $\jmath:\Z\cong\pi_1 (Y_p) \to \pi_1 (W)$ induced by the inclusion $Y_p  \subset W$.
Let $m_W$, $m_X$, and $m_\Sigma$ be the nonnegative integers characterized by
\begin{equation*}\label{m_X}
 \ker \jmath = m_W \mathbb Z,\qquad   (K\omega)(\pi_2(X))=m_X\mathbb{Z},\qquad  (K\omega_\Sigma)(\pi_2(\Sigma))=m_\Sigma\mathbb{Z}.
\end{equation*}
One can readily see that $m_X$ divides $m_\Sigma$. Our convention is that $0$ is divisible by any integer, and $0$ divides only $0$. 
We also note that if the map $\pi_2(\Sigma)\to\pi_2(X)$ induced by the inclusion is surjective, we have $m_\Sigma=m_X$. By the definition of $m_W$, a continuous loop $x:\mathbb R/T\mathbb Z \rightarrow Y_p$ is contractible in $W$ if and only if $m_W$ divides $\m (x)$.

\begin{prop}\label{contractible}
Let $x:\mathbb{R}/T \mathbb Z \rightarrow Y_p$ be a continuous loop for some $p\in\Sigma$ and $T>0$.
\begin{enumerate}
    \item[(a)] $x$ is contractible in $Y$ if and only if $m_\Sigma$ divides $\m(x)$. In particular, $m_W$ divides $m_\Sigma$.
    \item[(b)] If $x$ is contractible in ${W}$, then $m_X$ divides $\m(x)$. In particular, $m_X$ divides $m_W$.
\end{enumerate}
\end{prop}

\begin{proof}
 We consider the homotopy long exact sequence
    \begin{equation}\label{homotopy_les}
        \cdots \longrightarrow \pi_2(Y) \longrightarrow  \pi_2(\Sigma) \stackrel{\delta}{\longrightarrow}   \pi_1(S^1) \stackrel{\mathfrak{i}}{\longrightarrow}  \pi_1(Y) \longrightarrow \cdots 
    \end{equation}
    where $\delta(A)=e_Y(A)=-K\omega_\Sigma(A)\in \mathbb{Z}\cong \pi_1(S^1)$ and $\mathfrak i$ is induced by $S^1\cong Y_p\subset Y$. 
    The homotopy class of $x$ in Y is $\mathfrak{i}(\m(x))$ where $\m(x)\in\mathbb Z\cong\pi_1(S^1)$, and thus $[x]=[\mathfrak{i}(\m(x))]=0$ in $\pi_1(Y)$ exactly when $m_\Sigma$ divides $\m(x)$. This proves (a). 
    
    To show (b), suppose that $x$ is contractible in ${W}$. We view $x$ as a contractible loop in $X\setminus\Sigma$. Then there is a  capping disk $u : D^2:=\{z\in\mathbb{C}\mid |z|\leq 1\} \rightarrow X\setminus \Sigma$ of $x$, i.e.~$u$ is continuous and $u(e^{2\pi i t}) = x(Tt)$. We choose another capping disk $v:D^2 \rightarrow \varphi(\mathcal{U})\subset X$ of $x$ in $X$ such that $\varphi^{-1}\circ v(D^2)$ is contained in a  fiber of $E\to \Sigma$. Gluing $u$ and $v$ with orientation reversed, denoted by $v^\mathrm{rev}$, along $x$, we obtain $u\# v^\mathrm{rev}:S^2\to X$ with $(u\# v^\mathrm{rev})\cdot\Sigma=\m(x)$. For the sign, recall that $\alpha=-\frac{1}{K}\Theta$. This proves that $\m(x)$ is divisible by $m_X$. 
\end{proof}

\begin{prop}\label{prop:m_W}
Consider the map $\pi_1(Y) \rightarrow \pi_1({W})$ induced by the inclusion.
\begin{enumerate}
    \item[(a)] If this map is injective, then $m_W=m_\Sigma$.
    \item[(b)] If this map is surjective, then $m_W = m_X$.
\end{enumerate}    
\end{prop}
\begin{proof}
   Statement (a) readily follows from Proposition \ref{contractible}.(a).

Now we prove (b). Due to Proposition \ref{contractible} (b), it suffices to show that $m_W$ divides $m_X$. 
   Suppose that there is a continuous map $w:S^2\to X$ with $w\cdot\Sigma=K\omega([w])=m_X$. We may assume that $w$ is smooth and intersects $\Sigma$ at finitely many points $z_1,\dots,z_n\in S^2$. Let $U$ be a sufficiently small open tubular neighborhood of $\Sigma$ in $X$. To ease the notation, up to a homeomorphism, we identify $X\setminus U$ with $W$ and $\partial(X\setminus U)$ with $Y$. We may assume that there are mutually disjoint open disk neighborhoods $D_i$ of $z_i$ in $S^2$ such that $w^{-1}(U)=\bigcup \limits_{i=1}^n D_i$ and $w(\partial D_i)\subset Y_{w(z_i)}$. We take parametrizations $l_i : \mathbb{R} / \mathbb{Z} \to S^2$ of $\partial D_i$.
     We also take paths $c_{i} : [0, 1] \to S^2\setminus\bigcup \limits_{i=1}^n D_i$ for $i = 2, \dots, n$ such that 
     \[
     c_i ((0, 1)) \cap c_j((0,1))=\emptyset \;\;\;\forall i\neq j,\qquad c_i (0) = l_1 (0)\;\;\;\forall i,\qquad c_i (1) = l_i (0) \;\;\;\forall i.
     \]
   Let $p_i := w \circ l_i (0)$. Let  $\rho_i : \pi_1 (W, p_i) \to \pi_1 (W, p_1)$ be the isomorphism given by conjugation with $w \circ c_i$.
   Then, the homotopy class 
   \[
   [w \circ l_1] \cdot \rho_2 ([w \circ l_2] ) \cdot\; \cdots\; \cdot \rho_n ([w \circ l_n])\in \pi_1 (W, p_1)
   \]
   is trivial since a contracting disk is given by $w$ restricted to $S^2\setminus(\bigcup \limits_{i=1}^n D_i\cup \bigcup \limits_{j=2}^n c_j([0,1]))$.

   Since $\pi_1(Y)\to \pi_1(W)$ is surjective, there is a homotopy in $W$ between $w \circ c_i$ and a path $q_i:[0, 1] \to Y$ relative to the endpoints for every  $i = 2, \dots, n$.
   Let $\rho'_i : \pi_1 (W, p_i) \to \pi_1 (W, p_1)$ be the isomorphism defined by conjugation with $q_i$. 
   Due to the homotopies between $w\circ c_i$ and $q_i$, we have 
   \[
   [w \circ l_1] \cdot \rho'_2 ([w \circ l_2]) \cdot \; \cdots \; \cdot \rho'_n ([w \circ l_n]) = [w \circ l_1] \cdot \rho_2 ([w \circ l_2]) \cdot \; \cdots \; \cdot \rho_n ([w \circ l_n])  =0\in \pi_1(W,p_1).
   \]
   Let $x : \mathbb R / \mathbb Z \to Y_{\pi_Y (p_1)}$ be a continuous loop with $\m (x) = m_X$.
   Since $m_X = \sum \limits_{i=1}^n\m (w\circ l_i)$, we have 
   \[
   [x] =   [w \circ l_1] \cdot \rho'_2 ([w \circ l_2]) \cdot \; \cdots \; \cdot \rho'_n ([w \circ l_n]) =0\in \pi_1 (W, p_1).
   \]
   This proves that $m_W$ divides $m_X$. 
\end{proof}

Now we define the index $\mu_\CZ(\gamma)\in\mathbb{Z}$ of a Reeb orbit $\gamma:\mathbb R/T\mathbb Z\to Y$ contractible in ${W}$.
Let $\varphi^t_R$ denote the flow of $R$. 
We define the index $\mu_\CZ (\gamma)$ by the Conley-Zehnder index of the path of linearized maps $\{t\mapsto d\varphi_{ R}^t(\gamma(0)) \}_{t\in[0,T]}$ with respect to the trivialization of $\gamma^*TW$ induced by a capping disk $\bar\gamma:D^2\to  W$ of $\gamma$.
Note that the path $\{t\mapsto d\varphi_{ R}^t(\gamma(0)) \}_{t\in[0,T]}$ is degenerate in the sense that the map $(d\varphi_R^{T}(\gamma(0)) - \mathrm{Id})$ is not invertible, and refer to \cite{RS93} for the definition and properties of the Conley-Zehnder index for a path of symplectic matrices with possibly degenerate endpoints.

As the notation indicates, $\mu_\CZ(\gamma)$ does not depend on the choice of a capping disk in $W$, see \eqref{RS-index} below. 
 To compute $\mu_\CZ (\gamma)$, we view $\gamma$ as a loop in $X$ and take another capping disk $\bar\gamma_\mathrm{fib}:D^2\to \varphi(\mathcal U)\subset X$ of $\gamma$ such that $ \varphi^{-1}\circ \bar\gamma_\mathrm{fib}(D^2)$ is contained in a  fiber of $E\to \Sigma$. 
 Gluing $\bar\gamma$ and $\bar\gamma_\mathrm{fib}$ with orientation reversed, denoted by $\bar\gamma_\mathrm{fib}^\mathrm{rev}$, along $\gamma$, we obtain $B=[\bar\gamma\#\bar\gamma_\mathrm{fib}^\mathrm{rev}]\in\pi_2(X)$ with $K\omega(B)=B\cdot\Sigma=\m(\gamma)$.  
 Then, by properties of the Conley-Zehnder index, we have
\begin{equation}\label{RS-index}
\mu_\CZ(\gamma)=\mu_\CZ(\gamma,\bar\gamma_\mathrm{fib})+2c_1^{TX}(B)=2(\tau_X-K)\omega(B)=\frac{2(\tau_X-K)}{K}\m(\gamma),
\end{equation}
where $\mu_\CZ(\gamma,\bar\gamma_\mathrm{fib})$ is the Conley-Zehnder index defined using the trivialization induced by the capping disk $\bar\gamma_\mathrm{fib}$, and a simple computation shows $\mu_\CZ(\gamma,\bar\gamma_\mathrm{fib})=-2\m(\gamma)$. The same computation was made in \cite[Section 3.1]{DL2}.

For a Reeb orbit $\gamma:\mathbb R/T\mathbb Z\to Y$ contractible inside $Y$, the index $\mu_\CZ(\gamma)$ equals the Conley-Zehnder index of $\{t\mapsto d\varphi_R^t(\gamma(0))|_{\xi}\}_{t\in[0,T]}$ with respect to the trivialization of $\gamma^*\xi$, where $\xi=\ker\alpha$, induced by a capping disk $\bar\gamma$ of $\gamma$ contained in $Y$. In this case, we have $B=[\bar\gamma\#\bar\gamma_\mathrm{fib}^\mathrm{rev}]\in \pi_2(\varphi(\mathcal U))\cong\pi_2(E)\cong\pi_2(\Sigma)$. Thus, from  
 $c_1^{TE}(B)=c_1^E(B)+c_1^{T\Sigma}(B)$ and $c_1^E(B)=K\omega_\Sigma(B)$, we  deduce
\begin{equation}\label{RS-index2}
\mu_\CZ(\gamma)=\mu_\CZ(\gamma, \bar\gamma_\mathrm{fib}) + 2c_1^{TE}(B) =  2(\tau_X-K)\omega_{\Sigma}(B) =2c_1^{T\Sigma}(B).
\end{equation}

\subsection{Almost complex structures}\label{sec:acs}
Following \cite{DL2,DL1}, we will work with specific classes of almost complex structures on various spaces, namely $\Sigma$, $E$, $\mathbb{R} \times Y$, $X$ and $\widehat{W}$, which are correlated to each other. Let $\mathcal{J}_\Sigma$ denote the space of $\omega_\Sigma$-compatible almost complex structures  on $\Sigma$. We split the tangent bundle of $E$ into
\begin{equation}\label{splitting}
T_{\tilde p}E = T^\mathrm{v}_{\tilde p}E\oplus T^\mathrm{h}_{\tilde p}E  \cong E_p\oplus T_{p}\Sigma \qquad \textrm{for}\;\;\tilde p\in E,\; p :=\pi_E(\tilde p)
\end{equation}
where the vertical subspace $T^\mathrm{v}_{\tilde p}E$ is canonically isomorphic to $E_p$ and the horizontal subspace $T^\mathrm{h}_{\tilde p}E$ defined by the connection $\Theta$ is isomorphic to $T_p\Sigma$ via $d\pi_E$. 
Given $J_\Sigma \in \mathcal{J}_\Sigma$, we define an almost complex structure on $E$ by 
\begin{equation}\label{diag_J}
J_E=\begin{pmatrix}
 i & 0\\
 0 & J_\Sigma 
\end{pmatrix}  
\end{equation}
where the decomposition is with respect to \eqref{splitting}, and $i$ denotes the complex structure of $\pi_E:E\to\Sigma$. Then  $J_E$ is $\omega_E$-compatible and $d\pi_E\circ J_E = J_\Sigma \circ d\pi_E$ holds.

Next we consider the diffeomorphism 
\[
\psi_2: E\setminus\Sigma \rightarrow \mathbb{R}\times Y,\qquad v \mapsto \Big(-\frac{1}{2\pi K}\log|v|, \, \frac{v}{|v|}\Big),
\]  
which is not a symplectomorphism, and the space 
\begin{equation*}
    \mathcal{J}_Y := \left \{ \, J_Y : = (\psi_2)_*  J_E \, \middle | \, J_E \text{ is as in \eqref{diag_J} for some } J_\Sigma \in \mathcal{J}_\Sigma
 \, \right \}.
\end{equation*}
Then an almost complex structure $J_Y \in \mathcal{J}_Y$ on $\mathbb{R}\times Y$ is $d(e^r\alpha)$-compatible and satisfies $J_Y \partial_r = R$. Furthermore, $J_Y$ is invariant under the $\mathbb{R}\times S^1$-action on $\mathbb{R}\times Y$ induced by the translation on $\mathbb R$ and the $S^1$-action on $Y$ by the Reeb flow.
We denote the induced quotient map by
\[
\pi_{\mathbb{R}\times Y}:\mathbb{R}\times Y\to \Sigma,\qquad  \pi_{\mathbb{R}\times Y}(r,y):=\pi_Y(y).
\]
Then it holds that $d\pi_{\mathbb{R}\times Y}\circ J_Y = J_\Sigma \circ d\pi_{\mathbb{R}\times Y}$.
We point out that each $J_\Sigma\in\mathcal{J}_\Sigma$ uniquely defines $J_E$ and $J_Y$, and call $J_\Sigma$ the horizontal part of $J_E$ or of $J_Y$.

Using the embedding $\varphi:\mathcal U\to X$ given in \eqref{tubular}, we define the space
\begin{equation*}
    \mathcal{J}_X := \left\{ \, J_X \; \middle | \;
    \begin{aligned}
    &J_X \text{ is an $\omega$-compatible almost complex structure on $X$, } \\ 
    &  J_X|_{\varphi(\mathcal U)} = \varphi_*(J_E|_{\mathcal U}) \text{ for some } J_E \text{ of the form } \eqref{diag_J}
    \end{aligned}  \, \right \}.
\end{equation*}
We also call $J_\Sigma \in \mathcal{J}_\Sigma$ uniquely determined by $J_X \in \mathcal{J}_X$ the horizontal part.

To define a class of almost complex structures on $\widehat{W}$ we will use, we
fix $\epsilon\in(0,\rho_0^2)$ where $\rho_0$ is the radius of $\mathcal U$.
We then choose a diffeomorphism $g:(-\infty,0)\to\mathbb{R}$ such that 
\begin{align*}
    &g(r)=r \textrm{ on }(-\infty,-\epsilon/2), \quad g(-\epsilon/4)=-\epsilon/4, \\ 
    &g'(r) > 0 \textrm{ on } [-\epsilon/2, -\epsilon/4], \quad g'(r)=-\frac{1}{4\pi Kr} \textrm{ on } [-\epsilon/4,0)
\end{align*}
and set $G:(-\infty,0)\times Y\to \mathbb{R}\times Y$ by $G(r,y):=(g(r),y)$. 
We denote by $\mathcal{J}_{\widehat{W}}$ the space of almost complex structures on $\widehat{W}$ consisting of $d\lambda$-compatible almost complex structures $J_{\widehat{W}}$ of the form 
\begin{align}\label{acs_W}
    J_{\widehat{W}} = \left\{
        \begin{array}{ll}
        J_X & \text{on } \widehat{W} \setminus ([-\epsilon/2,\infty)\times Y), \\[1ex] 
        (G\circ\psi_1)_*J_E & \text{on } [-\epsilon/2,-\epsilon/4)\times Y, \\[1ex]
        J_Y & \text{on } [-\epsilon/4,\infty)\times Y,
        \end{array}
        \right.
\end{align}
for $J_X \in \mathcal{J}_X$, $J_Y \in \mathcal{J}_Y$, and $J_E$ in \eqref{diag_J} sharing the same horizontal part $J_\Sigma \in \mathcal{J}_\Sigma$. 
We call $J_Y$ in \eqref{acs_W} the cylindrical part of $J_{\widehat W}$. The spaces $\mathcal{J}_\Sigma$, $\mathcal{J}_Y$, $\mathcal{J}_X$, and $\mathcal{J}_{\widehat W} $ are nonempty and contractible. One can readily see this by adapting standard arguments in \cite[Section 2.5]{MS17}. 
We point out that $J_X \in \mathcal{J}_X$ uniquely determines $J_\Sigma$, $J_E$, $J_Y$, and thus also $J_{\widehat{W}}$. 
To be explicit, recalling that $\mathcal{U}\setminus\Sigma \cong (-\rho_0^2,0)\times Y$ via $\psi_1$, we note that the map 
\begin{equation}\label{G_hat}
\widehat G:X\setminus\Sigma\to \widehat{W},\qquad 
\widehat G(x)=\left\{
\begin{array}{ll}
	  x, & \;\;  x\in X\setminus\varphi(\mathcal U), \\[1ex] 
	  G \circ \psi_1 \circ \varphi^{-1}(x),  & \;\;  x\in \varphi(\mathcal U\setminus\Sigma),
\end{array}
\right.
\end{equation}
is a diffeomorphism satisfying $\widehat G_*(J_X|_{X\setminus\Sigma})=J_{\widehat{W}}$. As the following proposition shows, certain $J_X$-holomorphic spheres can be viewed as $J_{\widehat{W}}$-holomorphic planes via $\widehat G$, and vice versa.

\begin{prop}(\cite[Lemma 2.6]{DL1})\label{bijection}
The map $\widehat G$ defined in \eqref{G_hat} induces a bijection between $J_{\widehat{W}}$-holomorphic planes $u_{\widehat{W}}:\mathbb{C}\to \widehat{W}$ asymptotic to periodic Reeb orbits with multiplicity $k\in\N$ and $J_X$-holomorphic spheres $u_X:\mathbb{CP}^1\to X$ with $u_X^{-1}(\Sigma)=\{\infty\}$ and $u_X \cdot \Sigma=k$, where the correspondence is given to satisfy $\widehat G\circ u_X = {u_{\widehat{W}}}$ on ${\mathbb{CP}^1\setminus\{\infty\}}=\mathbb{C}$. 
\end{prop}

\begin{proof}
    Let $u_X:\mathbb{CP}^1\to X$ be as in the statement. Since $\widehat G_*(J_X|_{X\setminus\Sigma})=J_{\widehat{W}}$, the proper map $u_{\widehat{W}}:=\widehat G\circ u_X|_{\mathbb{CP}^1\setminus\{\infty\}}:\mathbb{CP}^1\setminus\{\infty\}=\mathbb{C}\to \widehat{W}$ is $J_{\widehat{W}}$-holomorphic.  Let $\lambda$ be the primitive 1-form of $\omega$ on $X\setminus\Sigma$ mentioned before \eqref{tubular}. We compute
    \[
    \int_{\mathbb{C}}(u_{\widehat{W}})^* d((\widehat G^{-1})^*\lambda)=\int_{\mathbb{CP}^1\setminus\{\infty\}}(u_X)^*\omega= \omega([u_X])= \frac{u_X\cdot\Sigma}{K} =\frac{k}{K}.
    \]
    A straightforward computation shows that $(\widehat G^{-1})^*\lambda$ restricted to $\{r\}\times Y\subset \widehat{W}$ converges to $\alpha$ as $r\to+\infty$. 
    Therefore $u_{\widehat W}$ has finite Hofer-energy, see \cite[Lemma 3.8]{Bou}, and if $\gamma$ denotes the asymptotic Reeb orbit of $u_{\widehat W}$, then $\int_\gamma\alpha=\frac{k}{K}$ by the above computation and Stokes' theorem. Since the Reeb flow of $\alpha$ is $\frac{1}{K}$-periodic, the multiplicity $\m(\gamma)$ of $\gamma$ is $k$. 
    
    Conversely, if $u_{\widehat{W}}$ is a $J_{\widehat{W}}$-holomorphic plane in $\widehat{W}$ as in the statement, then $u_X := \widehat{G}^{-1} \circ u_{\widehat{W}}$ is a $J_X$-holomorphic map on $\mathbb{CP}^1\setminus\{\infty\}$. 
    The above computation together with the removable singularity theorem implies that $u_X$ extends to a $J_X$-holomorphic map on $\mathbb{CP}^1$ with $(u_X)^{-1}(\Sigma) =  \{ \infty  \}$ and $u_X \cdot \Sigma = k$.
\end{proof}

\begin{remark}
The moduli space of unparametrized $J_{\widehat{W}}$-holomorphic planes in $\widehat{W}$ asymptotic to a Reeb orbit $\gamma$ has dimension
\begin{equation}\label{eq:index_moduli}
|\gamma| :=\frac{1}{2}\dim W -3 + \mu_\CZ(\gamma)-\frac{1}{2}\dim\Sigma=\mu_\CZ(\gamma)-2\geq 0
\end{equation}
if it is cut out transversely, see \cite[Cor. 5.4]{Bou}. It is nonnegative due to \eqref{RS-index}. 
\end{remark}

\section{Split Floer complex for $\bigvee$-shaped Hamiltonians}

We will work with Rabinowitz Floer homology defined using $\bigvee$-shaped Hamiltonians as in \cite{CFO} rather than the original definition using the Rabinowitz action functional in \cite{CF09}. 
In this section, following \cite{DL2,DL1}, we construct a split version of the Floer complex for $\bigvee$-shaped Hamiltonians.

\subsection{$\bigvee$-shaped Hamiltonians} \label{sec: Hamiltonians}

Let $\mathcal{H}$ be the set of smooth functions $H:\widehat{W} \rightarrow \mathbb{R}$ satisfying the following properties:
\begin{enumerate}
    \item[(i)] $H(r,y)=h(e^r)$ for $(r,y) \in \mathbb R \times Y \subset \widehat{W}$  where $h:\R\to\mathbb R$ is a smooth function such that  \begin{equation*}
   {h}(1) < 0, \qquad 
       h(\rho) = \begin{cases}
         \textrm{constant}, & \rho \in(-\infty,e^{-\epsilon/8}), \\[0.5ex]
         -a \rho + b^-, & \rho \in (e^{-\epsilon/8 + \eta} , e^{-\eta}), \\[0.5ex]
         a \rho + b^+, & \rho \in (e^\eta,+\infty),
    \end{cases}
    \end{equation*}
    for some $a \in \mathbb R_{> 0}$, $b^\pm \in \mathbb R$, and sufficiently small $\eta \in (0, \epsilon/16)$.  
    Here $\epsilon>0$ is the constant fixed in \eqref{acs_W}. 
    Moreover, we require
    \[
    h''<0 \; \textrm{ on } \; (e^{-\epsilon/8},e^{-\epsilon/8+\eta}), \qquad h''>0  \; \textrm{ on } \; (e^{-\eta}, e^\eta).
    \]
    \item[(ii)] $H$ is constant on $\widehat{W} \setminus (\mathbb R \times Y)$.
    \item[(iii)] $a \in \mathbb{R}_{>0} \setminus \frac{1}{K} \mathbb{Z}$.
\end{enumerate}
The Hamiltonian vector field $X_H$ of $H \in \mathcal{H}$, defined by $\iota_{X_H}\omega=-dH$, coincides with $h'(e^r)R$ on the symplectization part $\mathbb{R}\times Y$ of $\widehat{W}$. Therefore a 1-periodic orbit of $X_H$ on this part is contained in $\{b\}\times Y$ for some $b\in\mathbb R$ and corresponds to a $h'(e^b)$-periodic orbit of $R$ if $h'(e^b)>0$, to a constant orbit if $h'(e^b)=0$, and to a $-h'(e^b)$-periodic orbit of $-R$ if $h'(e^b)<0$. Since $a \notin \frac{1}{K} \mathbb Z$ and the flow of $R$ is $\frac{1}{K}$-periodic, a $1$-periodic orbit of $X_H$ is one of the following types:
\begin{enumerate}[(I)]
    \item constant orbits on $\widehat{W} \setminus ((-\epsilon/8,+\infty) \times Y)$,
    \item nonconstant orbits on $(-\epsilon/8,-\epsilon/8+\eta) \times Y$,
    \item constant and nonconstant orbits on $(-\eta,\eta) \times Y$.
\end{enumerate}

\begin{figure}[h]
     \centering
     \includegraphics[height = 7cm]{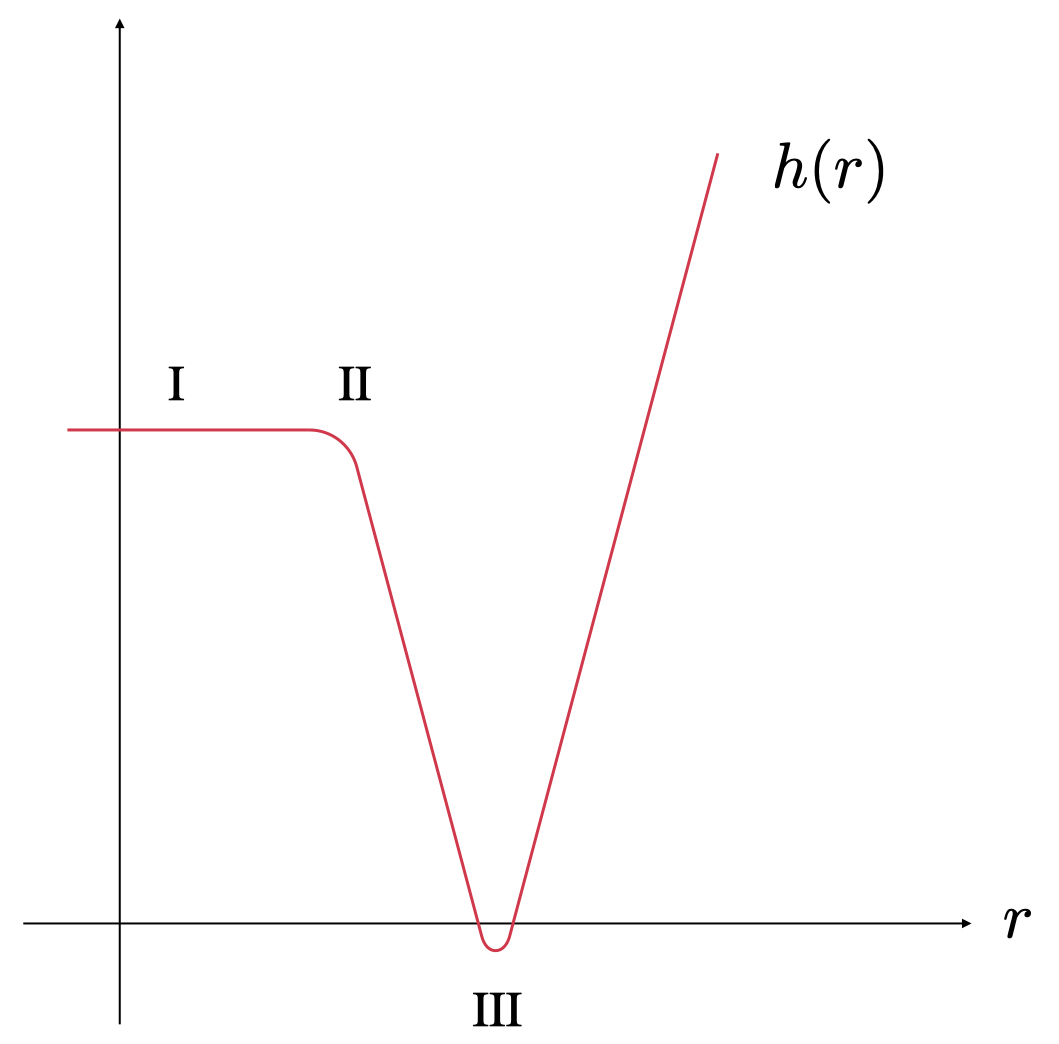}
     \caption{A Hamiltonian and types of orbits}
     \label{fig:Hamilt}
\end{figure}
In the construction of Rabinowitz Floer complex, we will discard orbits of type I and type II, and take only orbits of type III into account.  
Let $Y^H_k$ be the space of 1-periodic orbits $x$ of $X_H$ of type III with multiplicity $\m (x) = k$. It is diffeomorphic to $Y$ through 
\begin{equation}\label{identification}
Y^H_k \stackrel{\cong}{\longrightarrow} \{\beta^H_k\}\times Y ,\qquad x\longmapsto x(0),	
\end{equation}
where $\beta^H_k\in(-\eta,\eta)$ is uniquely determined by $h'(e^{\beta^H_k})=k/K$.  The index $\mu_\CZ(x)$ of $x\in Y^H_k$ contractible in $\widehat{W}$ is defined in the same way as in Section \ref{sec:index} using the linearized flow $\{t\mapsto d\varphi_{X_H}^t(x(0)) \}_{t\in[0,1]}$.

The same computation gives
\begin{equation}\label{eq:RS_XH}
\mu_\CZ(x) = \frac{2(\tau_X-K)}{K}\m(x).
\end{equation}
We point out that the multiplicity $\m(x)$ of $1$-periodic orbits $x$ of $X_H$ can be positive, zero, or negative, as opposed to periodic orbits of $R$.

Throughout this paper, we equate $S^1=\mathbb{R}/\mathbb{Z}$. Given $H \in \mathcal{H}$, the action functional $\mathcal{A}_H:C^\infty(S^1,\widehat{W})\rightarrow \mathbb{R}$ is defined by
\begin{equation}\label{action}
\mathcal{A}_H(x) = \int_0^1 x^*\lambda - \int_0^1 H(x(t))\,dt.
\end{equation}
Critical points of $\mathcal{A}_H$ are exactly 1-periodic orbits of $X_H$, and  
critical values of $\mathcal{A}_H$ are easily computed. 
If $x\in\Crit\mathcal{A}_H$ is of type I, then $\mathcal{A}_H(x)=-h(0)$. 
If $x\in\Crit\mathcal{A}_H$ is a type II or type III orbit lying on $\{b\}\times Y$, then $\mathcal{A}_H(x)=e^b h' (e^b) -h (e^b)$. We also note that type II orbits $x$ have $\mathcal{A}_H(x)<-h(0)$ since $h''<0$ on the region where $x$ are located.

\medskip
The spaces $Y^H_k$ are Morse-Bott critical manifolds of $\mathcal{A}_H$. We choose auxiliary Morse functions on $Y^H_k$ as follows. Let $(f_Y,Z_Y)$ be a Morse-Smale pair as in Section \ref{sec:Morse}. 
Through the diffeomorphism in \eqref{identification}, this gives rise to a Morse-Smale pair $(f_{Y^H_k}, Z_{Y^H_k})$ for each $k$. We denote by $\tilde p_k \in \Crit f_{Y^H_k}$ the $1$-periodic orbit corresponding to $\tilde p \in \Crit f_Y$.
We define the index of $\tilde{p}_k \in \Crit f_{Y^H_k}$ by 
\begin{equation}\label{grading}
    \mu(\tilde{p}_k):= \ind_{f_Y}(\tilde{p}) -\frac{1}{2}\dim\Sigma + \mu_\CZ(\tilde{p}_k) = \ind_{f_Y}(\tilde{p}) -\frac{1}{2}\dim\Sigma + 2\frac{(\tau_X-K)}{K}k.
\end{equation}
We refer to \cite[Proposition 2.2]{CFHW96} and \cite[Lemma 3.4]{BO1} for discussion on index in the case that a critical manifold is a disjoint union of circles, see also \eqref{eq:mu_perturb} below.

\subsection{Chains of pearls}\label{sec:chain_of_pearls}
Before studying moduli spaces of Floer cylinders for $\bigvee$-shaped Hamiltonians in $\widehat W$, in this section, we recall from \cite{DL2,DL1} the definition of moduli spaces of chains of pearls in $\Sigma$, which is reminiscent of the pearl complex for Lagrangian submanifolds studied in \cite{Oh96,BC4,BC5}.

\begin{definition}\label{moduli1}
For $N \in\mathbb{N}=\{1,2,\dots\}$, let $\mathbf{A}=(A_1,\dots,A_N) \in (\pi_2(\Sigma))^N$.
For $J_\Sigma \in \mathcal{J}_\Sigma$ and $\ell \in \mathbb{N} \cup \{ 0 \}$,
we define the moduli space 
\begin{equation*}
    \mathcal{N}_{N, \ell}(\mathbf{A};J_\Sigma)
\end{equation*}
consisting of tuples $( \mathbf{w}, \mathbf{z}  )= \big( ( w_1, \dots, w_N ), ( \Gamma_1, \dots, \Gamma_N )\big)$ where for every $i=1,\dots,N$
\begin{enumerate}[(1)]
    \item $w_i:\mathbb{CP}^1 \rightarrow \Sigma$ is a $J_\Sigma$-holomorphic map which represents $A_i \in \pi_2 (\Sigma)$, 
    \item $\Gamma_i : = ( z_{i,1}, \dots, z_{i, \ell_i} )$ is a finite sequence of pairwise distinct elements of $\mathbb{CP}^1 \setminus \left \{ 0, \infty \right \}$ for some $\ell_i \in \mathbb N \cup \left \{ 0 \right \}$ with $\ell = \sum\limits_{i=1}^{N} \ell_i$.
    If $\ell_i = 0$, we set $\Gamma_i = \emptyset$. 
\end{enumerate}
The subspace 
\begin{equation*}
    \mathcal{N}^*_{N, \ell}(\mathbf{A};J_\Sigma) \subset \mathcal{N}_{N, \ell}(\mathbf{A};J_\Sigma)
\end{equation*}
consists of elements satisfying the following additional properties:
\begin{enumerate}[(i)]
    \item each $w_i$ is either somewhere injective in the sense of \cite{MS} or constant,

    \item the image of $w_i$ is not entirely contained in that of $w_j$ for any $i \neq j$.
\end{enumerate}
\end{definition}

For a Morse-Smale pair $(f_\Sigma, Z_\Sigma)$ as in Section \ref{sec:Morse}, we consider the map
\begin{align}
    &\ev_{Z_\Sigma} : \mathcal{N}_{N, \ell} (\mathbf{A}; J_\Sigma) \times \mathbb{R}^{N - 1}_{>0} \to \Sigma^{2N} \label{eq:ev_Sigma}\\[.5ex]
    &\ev_{Z_\Sigma}((\mathbf{w}, \mathbf{z}), \mathbf t) := \big ( w_1(\infty), w_1 (0), \varphi^{t_{1}}_{Z_\Sigma}(w_2(\infty)), \dots, w_{N-1} (0), \varphi^{t_{N-1}}_{Z_\Sigma}( w_N (\infty)), w_N (0) \big ), \nonumber
\end{align}
where $\R_{>0}:=(0,\infty)$ and $\mathbf t = (t_1, \dots, t_{N-1})$. For $N = 1$, this means $\ev_{Z_\Sigma} : \mathcal{N}_{N = 1, \ell}(\mathbf{A}; J_\Sigma) \to \Sigma^2$ with $\ev_{Z_\Sigma}(w, \mathbf{z}) = (w(\infty), w(0))$.

\begin{definition}\label{def:pearl_Sigma}
 For $p, q \in \Crit f_\Sigma$, $N\in\mathbb{N}$, $\ell\in\mathbb{N}\cup\{0\}$, $\mathbf{A}\in(\pi_2(\Sigma))^N$, and  $J_\Sigma \in \mathcal{J}_\Sigma$, 
the moduli space of chains of pearls from $q$ to $p$ is defined by
\begin{equation*}
    \mathcal{N}_{N, \ell}(q, p ; \mathbf{A} ; J_\Sigma) : = \ev^{-1}_{Z_\Sigma} \big(  W^s_{Z_\Sigma}(p) \times  \Delta_\Sigma ^{N-1} \times W^u_{Z_\Sigma}(q) \big) ,
\end{equation*}
where $\Delta_\Sigma := \left \{ (x, y) \in \Sigma \times \Sigma \mid x = y \right \}$\footnote{ For $N=1$, this means $\ev_{Z_\Sigma}^{-1}( W^s_{Z_\Sigma}(p) \times W^u_{Z_\Sigma}(q))$.}, see Figure \ref{fig:pearl}.
The subspace of simple chain of pearls is defined by
\begin{equation*}\label{chain_of_pearls} 
    \mathcal{N}^*_{N, \ell}(q, p ; \mathbf{A} ; J_\Sigma) : = \mathcal{N}_{N, \ell}(q, p ; \mathbf{A} ; J_\Sigma) \cap \, \Big ( \mathcal{N}^*_{N, \ell}(\mathbf{A};J_\Sigma) \times {\mathbb R}^{N-1}_{> 0} \Big ).
\end{equation*}
We also define
\begin{equation*}
    \mathcal{N}_{N=0}(q,p) = \mathcal{N}_{N=0}^*(q,p):= W^s_{Z_\Sigma}(p) \cap W^u_{Z_\Sigma}(q).
\end{equation*}

\end{definition}

\begin{figure}[h]
     \centering
     \includegraphics[width = 12cm]{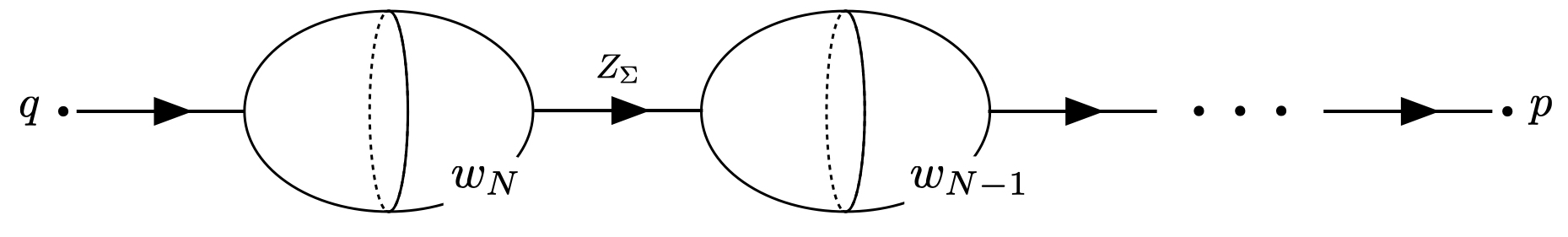}
     \caption{A chain of pearls in $\Sigma$}
     \label{fig:pearl}
\end{figure}

\begin{remark}\label{rmk:simple_pearl}
    For every  $((\mathbf{w}, \mathbf{z}),\mathbf{t})\in  \mathcal{N}_{N, \ell}(q, p ; \mathbf{A} ; J_\Sigma)$, one can obtain an underlying simple chain of pearls by replacing $w_i$ with a somewhere injective underlying map, omitting (sequences of) $w_i$ and $\Gamma_i$, and reparametrizing $w_i$ if necessary.   
    See \cite[Section 3.3]{BC4}.
\end{remark}

Next, we define chains of pearls that also contain $J_X$-holomorphic spheres in $X$.

\begin{definition}\label{def_moduli_B}
For $\ell\in\mathbb{N}$, let $\mathbf{B}=(B_1,\dots,B_\ell) \in (\pi_2(X))^\ell$  with $B_j\cdot \Sigma \neq 0$ for all $j = 1, \dots, \ell$.
For $J_X \in \mathcal{J}_X$, we define the moduli space \begin{equation*}
    \mathcal{P}_{\ell}(\mathbf{B};J_X)
\end{equation*}
consisting of $\ell$-tuples of equivalence classes $\mathbf{u}=\big( [u_1], \dots, [u_\ell] \big)$ where $u_1,\dots,u_\ell:\mathbb{CP}^1 \rightarrow X$ are $J_X$-holomorphic maps  satisfying the following conditions for every $j = 1, \dots, \ell$ :
\begin{enumerate}[(1)]
    \item $u_j$ represents $B_j \in \pi_2 (X)$ and $u^{-1}_j(\Sigma)=\left \{ \infty \right \}$,
    
    \item the image of $u_j$ is not entirely contained in the neighborhood $\varphi(\mathcal U)$ of $\Sigma$,
    
    \item the equivalence relation on $u_j$ is given by reparametrization by maps in $\mathrm{Aut}(\mathbb{CP}^1;\infty)$, where $\mathrm{Aut}(\mathbb{CP}^1;\infty)$ is the group  of biholomorphic maps on $\mathbb{CP}^1$ fixing $\infty$.
\end{enumerate}
We also consider the subspace
\begin{equation*}
    \mathcal{P}^*_{\ell}(\mathbf{B};J_X) \subset \mathcal{P}_\ell(\mathbf{B};J_X)
\end{equation*}
which consists of elements satisfying the following additional properties:
\begin{enumerate}[(i)]
    \item $u_j$ is somewhere injective in the sense of \cite{MS} for every $j = 1, \dots, \ell$,
    
    \item the image of $u_i$ is not entirely contained in that of $u_j$ for every $i \neq j$. 
\end{enumerate}
\end{definition}

We point out that if $\mathcal{P}_\ell(\mathbf{B};J_X)\neq\emptyset$, then $B_j\cdot\Sigma>0$ for every $j=1,\dots,\ell$ thanks to the positivity of intersection property between $J_X$-holomorphic maps.

For $\ell\in\N$, we define
\[
\begin{aligned}
 &{\aug}_\Sigma:\mathcal{N}_{N, \ell}(\mathbf{A};J_\Sigma) \to \Sigma^\ell    \\[.5ex]   
&    {\aug}_X: \mathcal{P}_{\ell}(\mathbf{B};J_X) \to \Sigma^\ell 
\end{aligned}
\qquad 
\begin{aligned}
	 (\mathbf{w}, \mathbf{z}) &\mapsto \big(w_1 (z_{1,1}), \dots, w_N (z_{N, \ell_N})\big), \\[.5ex]
	\mathbf{u} &\mapsto \big(u_1(\infty),\dots, u_\ell(\infty)\big)\;.
\end{aligned}
\]
Putting these and  the evaluation map in \eqref{eq:ev_Sigma} together, we define 
\begin{equation*}
	\begin{split}
    &\mathrm{EV}: \big( \mathcal{N}_{N, \ell}(\mathbf{A} ; J_\Sigma) \times \mathbb{R}^{N - 1}_{> 0} \big ) \times \mathcal{P}_{\ell}(\mathbf{B};J_X) \rightarrow \Sigma^{2N} \times \Sigma^\ell \times \Sigma^\ell \\[.5ex] 
    &\mathrm{EV} \big(((\mathbf{w},\mathbf{z}), \mathbf{t}),\mathbf{u}\big) := \big(\ev_{Z_\Sigma}((\mathbf{w},\mathbf{z}), \mathbf{t}), \aug_\Sigma(\mathbf{w},\mathbf{z}), \aug_X(\mathbf{u})\big).
\end{split}
\end{equation*}
 When considering this map, we always take  $J_\Sigma$ to be the horizontal part of $J_X$. 

\begin{definition}
For $p,q\in\Crit f_\Sigma$, $N,\ell\in\mathbb{N}$, $\mathbf{A}\in(\pi_2(\Sigma))^N$, $\mathbf{B}\in(\pi_2(X))^\ell$, and $J_X\in\mathcal{J}_X$, we define the moduli space of augmented chains of pearls from $q$ to $p$ by
\[
    \mathcal{N}_{N, \ell}(q, p; \mathbf{A}, \mathbf{B}; J_X) : =  \mathrm{EV}^{-1} \big(  W^s_{Z_\Sigma}(p) \times \Delta_\Sigma^{N-1} \times W^u_{Z_\Sigma}(q) \times \Delta_{\Sigma^\ell} \big ),
\]
where $\Delta_{\Sigma^\ell} := \left \{  (\mathbf{x}, \mathbf{y}) \in \Sigma^\ell \times \Sigma^\ell \mid \mathbf{x} = \mathbf{y} \right \}$\footnote{ For $N=1$, this means $\mathrm{EV}^{-1} \allowbreak \big(  W^s_{Z_\Sigma}(p) \times W^u_{Z_\Sigma}(q) \times \Delta_{\Sigma^\ell} \big )$.}, see Figure \ref{fig:aug_pearl}.
The subspace of simple elements is defined by
\[
    \mathcal{N}^*_{N, \ell}(q, p; \mathbf{A}, \mathbf{B}; J_X) : =  \mathcal{N}_{N, \ell}(q, p; \mathbf{A}, \mathbf{B}; J_X) \cap \big ( (\mathcal{N}^*_{N, \ell}(\mathbf{A} ; J_\Sigma) \times \mathbb{R}^{N - 1}_{>0} ) \times \mathcal{P}^*_{\ell}(\mathbf{B};J_X) \big ),   
\]
\end{definition}

\begin{figure}[h]
     \centering
     \includegraphics[width = 12cm]{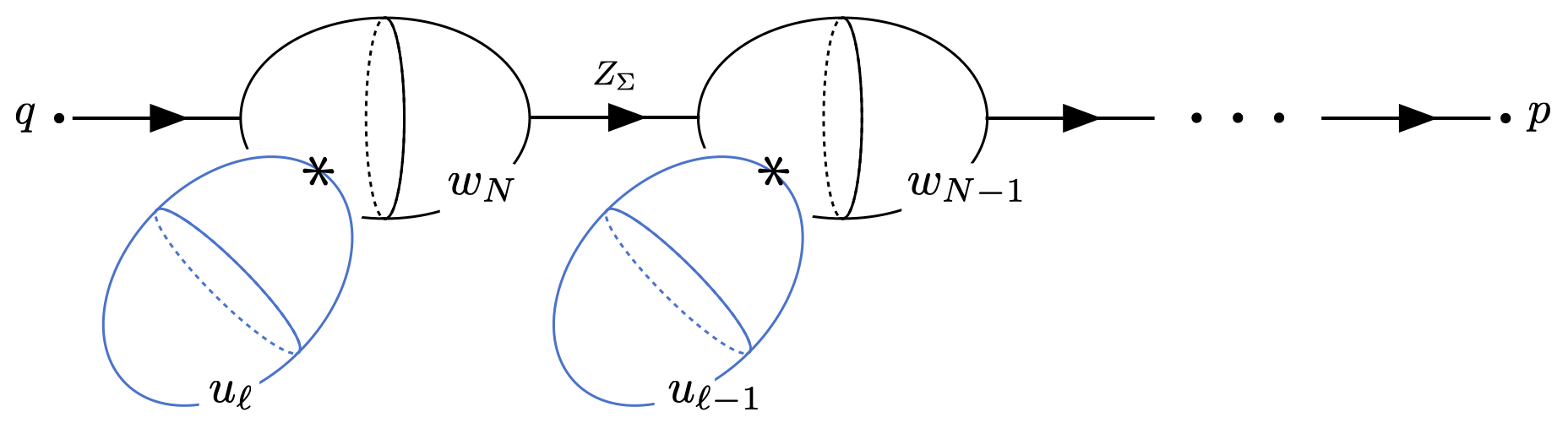}
     \caption{A chain of pearls in $\Sigma$ with augmentations in $X$}
     \label{fig:aug_pearl}
\end{figure}

The following result shows that  moduli spaces consisting of simple elements  are cut out transversely for a suitable choice of almost complex structures. 

\begin{prop}\label{transversality_chain}(\cite[Proposition 5.26]{DL2})
\begin{enumerate}[(a)]
    \item There is a residual set $\mathcal{J}^{\reg}_\Sigma \subset \mathcal{J}_\Sigma$ such that, for any $J_\Sigma \in \mathcal{J}^{\reg}_\Sigma$, $p,q \in \Crit f_\Sigma$, $N\in\mathbb{N}$, $\ell\in\mathbb{N}\cup\{0\}$, and $\mathbf{A}=(A_1,\dots,A_N) \in (\pi_2(\Sigma))^N$, the moduli space $\mathcal{N}^*_{N, \ell}(q,p; \mathbf{A}; J_\Sigma)$ is a smooth manifold of dimension 
    \begin{equation*}
        \dim \mathcal{N}^*_{N, \ell}(q,p;\mathbf{A};J_\Sigma) = \ind_{f_\Sigma} (p) - \ind_{f_\Sigma} (q) + N-1 + \sum_{i=1}^N 2c_1^{T\Sigma}(A_i) + 2\ell.
    \end{equation*}
    
    \item There is a residual set $\mathcal{J}^{\reg}_X \subset \mathcal{J}_X$ such that, for any $J_X \in \mathcal{J}^{\reg}_X$,  $p,q \in \Crit f_\Sigma$, $N,\ell\in\mathbb{N}$, $\mathbf{A} = (A_1,\dots,A_N) \in (\pi_2(\Sigma))^N $, and $\mathbf{B} = (B_1,\ldots,B_\ell) \in (\pi_2(X))^\ell$, the moduli spaces $\mathcal{P}^*_\ell(\mathbf{B};J_X)$ and $\mathcal{N}^*_{N, \ell}(q,p;\mathbf{A},\mathbf{B};J_X)$ are smooth manifolds of dimensions
    \begin{align*}
    \dim \mathcal{P}^*_\ell (\mathbf{B};J_X) & = \sum_{j=1}^\ell \big(\dim X +2 c_1^{TX}(B_j) - 2K\omega(B_j) -4 \big),\\
        \dim \mathcal{N}^{*}_{N, \ell}(q,p;\mathbf{A}, \mathbf{B};J_X) & = \ind_{f_\Sigma} (p) - \ind_{f_\Sigma} (q) + N-1 + \sum_{i=1}^N 2c_1^{T\Sigma}(A_i) \\ 
        &\quad  + \sum_{j=1}^{\ell} \big (2c_1^{TX}(B_j)-2 K\omega(B_j)\big).
    \end{align*}
    Moreover, the map $\mathcal{J}^{\reg}_X \to \mathcal{J}^{\reg}_\Sigma$ sending $J_X$ to its horizontal part $J_\Sigma$ is well-defined and surjective.
\end{enumerate}
\end{prop}

\begin{remark}\label{rem:hol_plane} 
The dimension of $\mathcal{P}^*_\ell (\mathbf{B};J_X)$ can also be computed in the following way. Due to Proposition \ref{bijection}, we can view $\mathcal{P}_\ell(\mathbf{B}; J_X)$ as the moduli space of $\ell$-tuples of unparametrized $J_{\widehat{W}}$-holomorphic planes $[u_j]$ in $\widehat{W}$, where $j=1,\dots,\ell$, asymptotic to Reeb orbits $\gamma_j$ with multiplicity $\m(\gamma_j)=B_j\cdot \Sigma=K\omega(B_j)\in m_W\Z$. The dimension of the latter moduli space can be computed as 
$
\sum\limits_{j=1}^\ell(\mu_\CZ(\gamma_i)+\dim\Sigma-2)
$
if it is cut out transversly, see \eqref{eq:index_moduli} or \cite{Bou}, and by \eqref{RS-index} this indeed agrees with the dimension computation for $\mathcal{P}^*_\ell (\mathbf{B};J_X)$ above, which follows from \cite[Lemma 6.7]{CM}.
\end{remark}

We endow  $\mathcal{N}^*_{N, \ell}(q,p;\mathbf{A};J_\Sigma)$ and $\mathcal{N}^{*}_{N, \ell}(q,p;\mathbf{A}, \mathbf{B};J_X)$ with orientations according to the following fibered sum rule, the convention used  in \cite{BO1,DL2,AK23}. 
\begin{definition}\label{def:fiber_sum}
	Let $W$ be an oriented vector space. Let $V\subset W$ be an oriented subspace. The quotient space $W/V$ is oriented so that the isomorphism $V\oplus (W/V)\to W$ given by a splitting of the  exact sequence $0\to V\to W\to W/V\to 0$ is orientation-preserving. 
    Let $f_1 : V_1 \to W$ and $f_2 : V_2 \to W$ be linear maps between oriented vector spaces such that $f_1 - f_2 : V_1 \oplus V_2 \to W$ defined by $(f_1 - f_2)(v_1, v_2) = f_1(v_1)-f_2(v_2)$ is surjective.
    We orient $V_1 \prescript{}{f_1}\oplus^{}_{f_2} V_2 := \ker (f_1 - f_2)$   to satisfy  the  isomorphism $(V_1 \oplus V_2) / (V_1 \prescript{}{f_1}\oplus^{}_{f_2} V_2)  \to W$ given by $f_1 - f_2$ being orientation-preserving if and only if $(-1)^{\dim V_2 \cdot \dim W}=1$.
\end{definition}
We explain orientations for two particular cases, namely 
\begin{equation}\label{eq:orientation_moduli_0}
\mathcal{N}^*_{N = 1, \ell = 0}(q,p;A;J_\Sigma)\quad \text{and}\quad \mathcal{N}^{*}_{N = 1, \ell = 1}(p,p;A=0, B;J_X),	
\end{equation}
which are relevant in our situation, see Proposition \ref{prop:cpt} below.
We first observe that the space $\mathcal{N}^*_{N=1 ,\ell = 0} (A ; J_\Sigma)$ carries a canonical orientation due to the fact that the associated linearized operator is a compact perturbation of a complex linear Cauchy-Riemann operator, see \cite[Theorem 3.1.6.(i)]{MS}. To orient $\mathcal{N}^*_{N=1 ,\ell = 0} (q,p;A; J_\Sigma)$, we view this space as a fiber product space as follows: 
\begin{equation}\label{eq:orientation_moduli_1}
\begin{aligned}
    \mathcal{N}^*_{N=1, \ell = 0} (q,p; A; J_\Sigma) &= W^s_{Z_\Sigma}(p) \prescript{}{\iota^s} \times^{}_{\ev_\infty} \mathcal{N}^*_{N=1, \ell = 0}(A; J_\Sigma) \prescript{}{\ev_0} \times^{}_{\iota^u} W^u_{Z_\Sigma}(q) \\
    w &\mapsto (w(\infty), w, w(0)).
\end{aligned}
\end{equation}
where $\iota^s$ and $\iota^u$ are natural inclusion maps, and $\ev_0$ and $\ev_\infty$ refer to the evaluation maps at $0$ and $\infty$, respectively. The orientations on $W^u_{Z_\Sigma}(q)$ and $W^s_{Z_\Sigma}(p)$ chosen in Section \ref{sec:Morse} and the orientation on $\mathcal{N}^*_{N=1 ,\ell = 0} (A ; J_\Sigma)$ induce an orientation on the fiber product space in \eqref{eq:orientation_moduli_1} according to the fibered sum rule in Definition \ref{def:fiber_sum}.
Note that the fibered sum rule is associative, namely the resulting orientation is independent of the order of taking fiber products.
This determines an orientation on $\mathcal{N}^*_{N=1 ,\ell = 0} (q,p;A; J_\Sigma)$ via \eqref{eq:orientation_moduli_1}.

For the latter space in \eqref{eq:orientation_moduli_0}, we use the natural diffeomorphism  
\begin{equation}\label{eq:orient_A=0}
	\mathcal{N}^*_{N = 1, \ell = 1}(A=0;J_\Sigma)  \cong \Sigma \times (\mathbb{CP}^1 \setminus \left \{ 0, \infty \right \}),
\end{equation}
which follows from the fact that $J_\Sigma$-holomorphic spheres representing $A=0$ are constant maps. The space on the right-hand side is oriented by $\omega_\Sigma$ and the complex structure on $\mathbb{CP}^1$, and thus $\mathcal{N}^*_{N = 1, \ell = 1}(A=0;J_\Sigma) $ inherits an orientation via \eqref{eq:orient_A=0}. Let $\widetilde{\mathcal{P}}^*_{\ell = 1} (B; J_X)$ be the space consisting of $J_X$-holomorphic parametrizations of elements in $\mathcal{P}^*_{\ell = 1} (B; J_X)$. The space $\widetilde{\mathcal{P}}^*_{\ell = 1} (B; J_X)$  carries a canonical orientation induced by the associated linearized operator as explained above. This induces an orientation on $\mathcal{P}^*_{\ell = 1} (B; J_X)=\widetilde{\mathcal{P}}^*_{\ell = 1} (B; J_X)/\mathrm{Aut}(\mathbb{CP}^1;\infty)$. 
        Like \eqref{eq:orientation_moduli_1}, we identify $\mathcal{N}^*_{N = 1, \ell = 1} (p, p; A = 0, B; J_X)$ with the fiber product space 
\[
    W^s_{Z_\Sigma}(p) \prescript{}{\iota^s} \times^{}_{\ev_\infty} \big ( \, \mathcal{P}^*_{\ell = 1} (B; J_X) \prescript{}{\aug_X} \times^{}_{\aug_\Sigma\,} \mathcal{N}^*_{N = 1, \ell = 1} (A = 0; J_\Sigma) \, \big ) \prescript{}{\ev_0} \times^{}_{\iota^u} W^u_{Z_\Sigma}(p)
\]
which is oriented by the fibered sum rule. This endows $\mathcal{N}^*_{N = 1, \ell = 1} (p, p; A = 0, B; J_X)$ with an orientation. 

\subsection{Split Floer cylinders}\label{sec:Floer}
In this section, we study  moduli spaces of split Floer cylinders which lead to a split version of the Floer chain complex. Although Hamiltonians used in \cite{DL2,DL1} are different from our $\bigvee$-shaped Hamiltonians in $\mathcal{H}$, analysis in those articles carries over  to our setting. 

Recalling that $\mathbb R\times Y\subset \widehat{W}$, we view $H\in\mathcal{H}$ as a Hamiltonian on $\mathbb R\times Y$, which is constant on $(-\infty,-\epsilon/8)\times Y$. 
Let $\Gamma\subset\mathbb R\times S^1$ be a finite subset, which we also allow to be empty.
For $J_Y\in\mathcal{J}_Y$, we refer to smooth solutions $\tilde{v} = (b, v):\mathbb{R}\times S^1 \setminus\Gamma \rightarrow \mathbb R\times Y$ of the Floer equation
\begin{equation}\label{Floer}
\partial_s \tilde{v} + J_Y(\tilde{v})\big(\partial_t \tilde{v} - X_H(\tilde{v})\big)=0
\end{equation}
asymptotic to critical points of $\mathcal{A}_H$ at negative and positive ends, i.e.
\begin{align*}
{\ev}_\pm(\tilde{v})=\big({\ev}_\pm(b),{\ev}_\pm({v})\big) := \lim_{s \to \pm \infty} \tilde{v}(s,\cdot) = \Big(\lim_{s \to \pm \infty} b(s,\cdot),\lim_{s \to \pm \infty} {v}(s,\cdot)\Big) \in\Crit\mathcal{A}_H,
\end{align*}
as {\it Floer cylinders}. A Floer cylinder is said to be nontrivial if $\partial_s\tilde v\neq 0$. 
Without loss of generality, we assume that $\Gamma$ consists of non-removable punctures throughout the paper. 
On the region $(-\infty,-\epsilon/8)\times Y$ where $X_H$ vanishes, equation \eqref{Floer} reduces to the Cauchy-Riemann equation, and we may apply the neck stretching technique.

A straightforward computation using $J_Y\partial_r=R$ and $d\pi_{\mathbb{R}\times Y}\circ J_Y = J_\Sigma \circ d\pi_{\mathbb{R}\times Y}$ shows that equation \eqref{Floer} for $\tilde{v}=(b,v)$ is equivalent to
\begin{equation}\label{eq:Floer_eq2}
\begin{cases}
     \partial_s b-\alpha(\partial_t v) + h'(e^b) = 0, \\[0.5ex]
     \partial_t b+\alpha(\partial_sv)=0,\\[0.5ex]
     d\pi_Y \partial_s v  + J_\Sigma(\pi_Y(v)) d\pi_Y \partial_tv= 0.
\end{cases}	
\end{equation}
The last line yields that 
$\pi_{\mathbb{R}\times Y}\circ\tilde v=\pi_Y\circ v:\mathbb{R}\times S^1\setminus\Gamma\to \Sigma$
is $J_\Sigma$-holomorphic. 

The energy of a  Floer cylinder $\tilde{v}$ is defined by 
\begin{equation}\label{energy}
    E(\tilde{v}) := \sup_{\psi} E_\psi(\tilde v),\qquad E_\psi(\tilde v):= \int_{\mathbb{R}\times S^1
    \setminus\Gamma} \tilde{v}^*(d(\psi\alpha))-\tilde{v}^*(dH)\wedge dt 
\end{equation}
 where the supremum is taken over the set
\[
    \left \{ \psi \in C^\infty(\mathbb{R},[0,+\infty)) \, \middle | \, \text{ } \psi' \geq 0 \text{  and } \psi(r)=e^r \text{ for } r \geq -\epsilon \right \},
\]
for $\epsilon>0$ chosen in Section \ref{sec:acs}.
The condition $\psi(r)=e^r$ for $r\geq-\epsilon$ implies that $E(\tilde{v})$ is equal to the sum of the energy for Floer cylinders on the region $\tilde{v}^{-1}([-\epsilon,+\infty)\times Y)$ and the Hofer energy on the complement. 

Therefore a Floer cylinder $\tilde{v}=(b,v):\mathbb{R} \times S^1 \setminus \Gamma \rightarrow \mathbb{R}\times Y$ has finite energy if and only if the following asymptotic behavior holds for every $z\in\Gamma$:
 for holomorphic polar coordinates $\epsilon_{z}:(-\infty,0] \times S^1 \rightarrow \mathbb{R}\times S^1$ near $z \in \Gamma$, the map $\tilde v\circ\epsilon_z(s,t)$ converges to the orbit cylinder $(-\infty,0)\times\gamma_z(Tt)$ for some periodic Reeb orbit $\gamma_z:\mathbb R/T\mathbb Z\to Y$ in the sense of \cite{Hof93,HWZ96} as $s\to -\infty$. In particular,
	\begin{equation}\label{asymptotic}
	    \lim_{s \rightarrow -\infty} (b\circ\epsilon_{z},v\circ\epsilon_{z})(s,t) = (-\infty, \gamma_z(Tt)).
	\end{equation}
If $\tilde v$ has finite energy, then so does $\pi_Y\circ v$, namely $\int_{\mathbb R \times S^1 \setminus \Gamma}(\pi_Y \circ v)^* \omega_\Sigma<\infty$. Hence, due to the removable singularity theorem,  $
\pi_Y\circ v $ extends to a $J_\Sigma$-holomorphic map defined on $\mathbb{CP}^1$, where $\mathbb{CP}^1\setminus\{0,\infty\}$ is equated with $\mathbb{R}\times S^1$. Since we will only deal with finite energy solutions $\tilde v$ of \eqref{Floer}, we abuse notation and write   
\[
\pi_{Y}\circ  v:\mathbb{CP}^1\longrightarrow\Sigma,
\]
where
\[
(\pi_Y\circ v)(0)=\pi_Y\circ(\ev_-( v)),\quad  (\pi_Y\circ v)(\infty)=\pi_Y\circ(\ev_+( v)),\quad(\pi_Y\circ v)(z)=\pi_Y\circ\gamma_z
\]
for $z\in\Gamma$. 
Here we equated the constant curves $\pi_Y\circ(\ev_\pm(v))$ and $\pi_Y\circ\gamma_z$ with their images. 

Let $\tilde v=(b,v):\mathbb{R}\times S^1\setminus\Gamma\to\mathbb{R}\times Y$ be a  Floer cylinder with finite energy such that $v$ converges to periodic Reeb orbits $\gamma_{z_j}$ with period $\m(\gamma_{z_j})/K > 0$ at $z_j\in\Gamma$. Then Stokes' theorem yields
\begin{equation}\label{eq:action_increase}
    0\leq E_{\psi=e^r}(\tilde v)= \mathcal{A}_H(\ev_+(\tilde v))-\mathcal{A}_H(\ev_-(\tilde v)).
\end{equation}
The first inequality follows from $E_{\psi=e^r}(\tilde v) = \int_{\mathbb R \times S^1 \setminus \Gamma} |\partial_s \tilde{v}|^2dsdt$, where the norm is induced by $d(e^r \alpha)$ and $J_Y$.
In particular, $\mathcal{A}_H(\ev_+(\tilde v))\geq\mathcal{A}_H(\ev_-(\tilde v))$. Moreover, by a property of $c_1^E$, \eqref{RS-index}, and \eqref{eq:RS_XH}, we have 
\begin{equation}\label{multiplicity}
\begin{split}
0\leq 2c_1^{T\Sigma}([\pi_Y\circ v])&=\frac{2(\tau_X-K)}{K}c_1^E([\pi_Y\circ v]) \\
&=\frac{2(\tau_X-K)}{K}\Big(\m (\ev_+({v})) - \m(\ev_-({v})) -\sum_{z_j\in\Gamma} \m(\gamma_{z_j})\Big)\\ &=\mu_\CZ(\ev_+(\tilde{v}))-\mu_\CZ(\ev_-(\tilde{v}))-\sum_{z_j\in\Gamma}\mu_\CZ(\gamma_{z_j}).	
\end{split}
\end{equation}
Here the inequality follows from $c_1^{T\Sigma}([\pi_Y\circ v])=(\tau_X-K)\omega([\pi_Y\circ v])\geq 0$. The last equality holds for contractible $\ev_+(\tilde{v})$, $\ev_-(\tilde{v})$, and $\gamma_{z_j}$. Furthermore, $\omega([\pi_Y\circ v])= 0$, i.e.~$\pi_Y\circ v$ is constant, if and only if $\m(\ev_+(v))=\m(\ev_-(v)) + \sum\limits_{z_j\in\Gamma} \m(\gamma_{z_j})$. 
In the case of $\m(\ev_+(v))=\m(\ev_-(v))$, we have $\Gamma=\emptyset$ and $\pi_Y\circ v$ being constant. If in addition both $\ev_+(\tilde v)$ and $\ev_-(\tilde v)$ are of type III, then $\mathcal{A}_H(\ev_+(\tilde v))=\mathcal{A}_H(\ev_-(\tilde v))$ and thus $\tilde v$ maps to the orbit $\ev_+(\tilde v)=\ev_-(\tilde v)$ lying over the point $\pi_Y\circ v$  by \eqref{eq:action_increase}.

\medskip

Next, we define moduli spaces which contribute to the split Floer chain complex. From now on, 
$H \in \mathcal{H}$ and $J_X \in \mathcal{J}_X$. As before, we write $J_Y$ and $J_\Sigma$ for the cylindrical and horizontal parts of $J_X$, respectively. We also fix a Morse-Smale pair $(f_Y, Z_Y)$ as in Section \ref{sec:Morse}.  Let $N\in\N$ and $\ell\in\N\cup\{0\}$. Let $\mathbf{A} := (A_1, \dots, A_N) \in (\pi_2(\Sigma))^N$ as in Section \ref{sec:chain_of_pearls}.

\begin{definition}
The moduli space
\begin{equation} \label{moduli_cylinder}
     \mathcal{M}_{N, \ell; \, k_-, k_+}(\mathbf{A};H;J_Y)
\end{equation}
consists of pairs $(\mathbf{v},\mathbf{z})=\big((\tilde{v}_1, \dots, \tilde{v}_N), ( \Gamma_1, \dots, \Gamma_N )\big)$ such that for every $i=1,\dots,N$ 
\begin{enumerate}[(1)]
    \item $\Gamma_i  = ( z_{i,1}, \dots, z_{i, \ell_i} )$ is a finite sequence of pairwise distinct elements of $\mathbb R \times S^1$ for some $\ell_i \in \mathbb N \cup \left \{ 0 \right \}$  with $\ell = \sum\limits_{i=1}^{N} \ell_i$.
    For $\ell_i = 0$, we set $\Gamma_i = \emptyset$,
    
    \item $\tilde{v}_i = (b_i, v_i) : \mathbb{R} \times S^1 \setminus \Gamma_i \rightarrow \mathbb{R} \times Y$ is a nontrivial Floer cylinder with finite energy such that $\pi_Y \circ v_i$ represents $A_i$, 
    
\end{enumerate}
    and in addition for every $i=1,\dots, N-1$  
	\[
k_+=\m(\ev_+(v_1)),\;\; k_-=\m(\ev_-(v_N)),\;\;  \m(\ev_-(v_{i}))=\m(\ev_+(v_{i+1})).
\]
We also define the subspace
\[
    \mathcal{M}^*_{N,\ell;k_-,k_+}(\mathbf{A};H;J_Y)\subset \mathcal{M}_{N,\ell;k_-,k_+}(\mathbf{A};H;J_Y)
\] 
consisting of simple elements $(\mathbf{v},\mathbf{z})$, namely
\[
\big((\pi_{ Y}\circ {v}_1,\dots,\pi_{Y}\circ {v}_N),\mathbf{z}\big)\in \mathcal{N}^*_{N, \ell}(\mathbf{A};J_\Sigma).
\]
\end{definition}

\begin{remark}
	Punctures $\mathbf{z}$ are free to move on the domain $\R\times S^1$. In order to study the moduli space in \eqref{moduli_cylinder} using a linearized operator, we may fix the position of punctures at the cost of working with a family of complex structures on $\R\times S^1$, see \cite[p.21]{DL2}.
\end{remark}

We consider the map
\begin{align}\label{ev_Y}
    &\widetilde{\ev}_{Z_Y} : \mathcal{M}_{N, \ell; k_-, k_+} (\mathbf{A}; H; J_Y) \times \R_{>0}^{N-1} \to Y^{2N} \\[1ex]
    &\widetilde{\ev}_{Z_Y}( (\mathbf{v}, \mathbf{z}), \mathbf{t} ) \nonumber \\
    &:= \big ( \ev_{+,0}(\tilde{v}_1), \ev_{-,0}(\tilde{v}_1), \varphi^{t_1}_{Z_Y}(\ev_{+,0}(\tilde{v}_2)), \ev_{-, 0}(\tilde{v}_{2}),\dots, \varphi^{t_{N-1}}_{Z_Y}(\ev_{+,0}(\tilde{v}_N)), \ev_{-,0}(\tilde{v}_N) \big ), \nonumber
\end{align}
where $\mathbf{t} : = (t_1, \dots, t_{N-1})$ and $\ev_{\pm,0}(\tilde{v}):= \lim \limits_{s \to \pm \infty} v(s,0)$ for $\tilde{v}=(b,v)$. In the case of $N = 1$, this means $\widetilde{\ev}_{Z_Y} : \mathcal{M}_{N, \ell; k_-, k_+} (\mathbf{A}; H; J_Y) \to Y^2$ with $(\tilde{v}, \mathbf{z}) \mapsto \big ( \ev_{+,0}(\tilde{v}), \ev_{-,0}(\tilde{v}) \big )$.

Let $f_{Y_{k_\pm}^H}$ be the Morse functions on $Y_{k_\pm}^H$ corresponding to $f_Y$ as in Section \ref{sec: Hamiltonians}. Let  
\[
\tilde{p}_{k_+}\in\Crit f_{Y^H_{k_+}},\qquad \tilde{q}_{k_-} \in \Crit f_{Y^H_{k_-}}\quad\textrm{for}\;\; k_\pm \in m_W \mathbb{Z}
\]
 be 1-periodic orbits of type III corresponding to $\tilde p,\tilde q\in\Crit f_Y$, respectively.

\begin{definition}
For $\ell \in \mathbb{N} \cup \left \{ 0 \right \}$, we define the moduli space 
\begin{equation*}
    \mathcal{M}_{N, \ell} (\tilde{q}_{k_-}, \tilde{p}_{k_+}; \mathbf{A}; H; J_Y) := (\widetilde{\ev}_{Z_Y})^{-1} \big ( W^s_{Z_Y}(\tilde{p}) \times \Delta_Y^{N-1} \times W^u_{Z_Y}(\tilde{q}) \big ),
\end{equation*}
where $\widetilde{\ev}_{Z_Y}$ is given in \eqref{ev_Y} and $\Delta_Y := \left \{ (x,y) \in Y \times Y \, \mid \, x = y \right \}$.\footnote{ For $N = 1$, this means $(\widetilde{\ev}_{Z_Y})^{-1} \big ( W^s_{Z_Y}(\tilde{p}) \times W^u_{Z_Y}(\tilde{q}) \big )$.}    
The subspace of simple elements is defined by
\begin{align*}
&\mathcal{M}^*_{N, \ell} (\tilde{q}_{k_-}, \tilde{p}_{k_+}; \mathbf{A}; H; J_Y) \\
&\hspace{1cm}:= \mathcal{M}_{N, \ell} (\tilde{q}_{k_-}, \tilde{p}_{k_+}; \mathbf{A}; H; J_Y)\, \cap \, \big ( \mathcal{M}^*_{N, \ell; k_-, k_+} (\mathbf{A}; H; J_Y) \times \mathbb{R}^{N-1}_{>0} \big ) .    
\end{align*}
For $\ell = 0$, we write $\mathcal{M}_{N, \ell = 0} (\tilde{q}_{k_-}, \tilde{p}_{k_+}; \mathbf{A}; H; J_Y)$ and $\mathcal{M}^*_{N, \ell = 0} (\tilde{q}_{k_-}, \tilde{p}_{k_+}; \mathbf{A}; H; J_Y)$.
\end{definition}

\begin{figure}[h]
     \centering
     \includegraphics[height = 6cm]{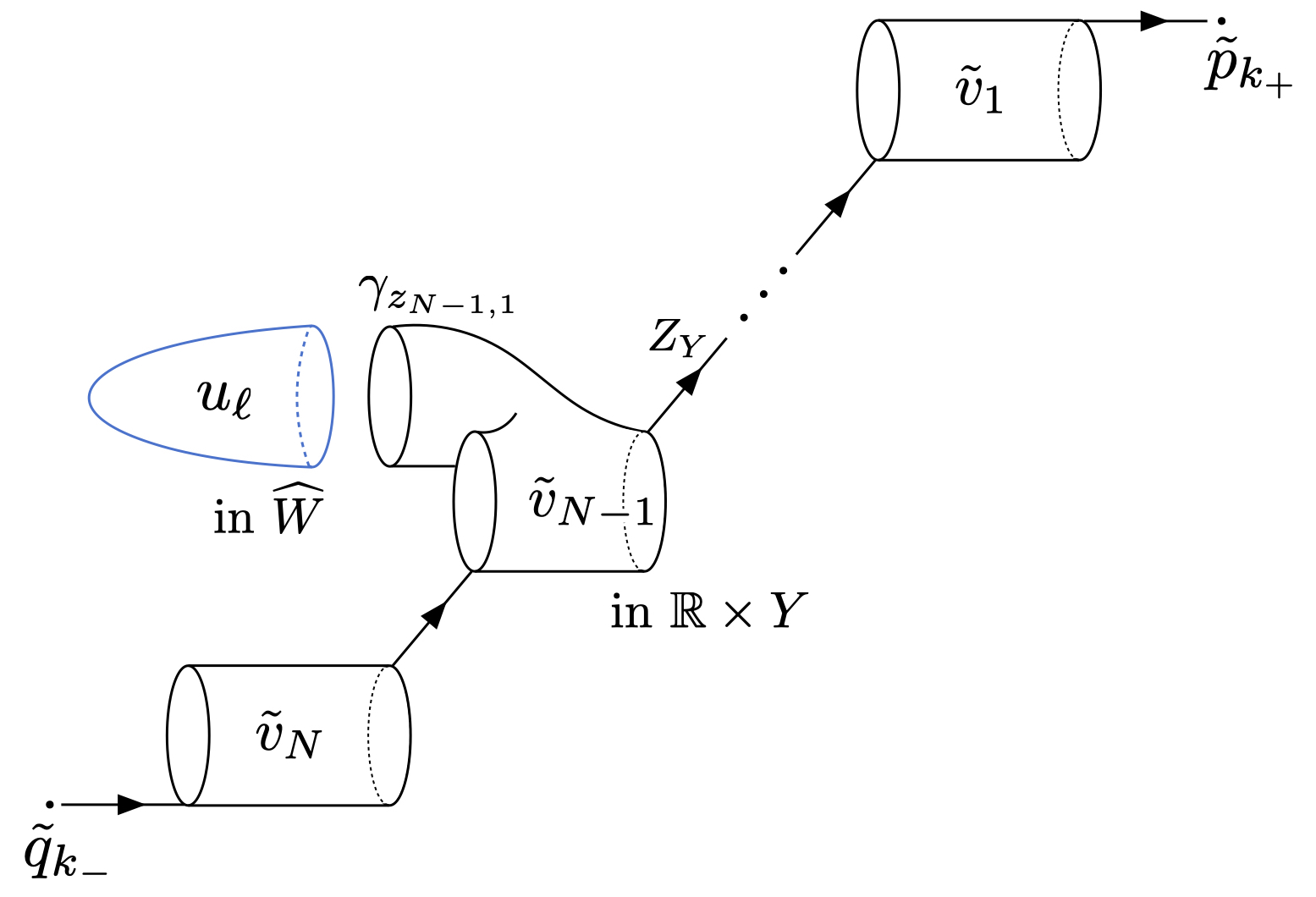}
     \caption{An element of $\mathcal{M}_{N, \ell} (\tilde{q}_{k_-}, \tilde{p}_{k_+}; \mathbf{A}, \mathbf{B}; H; J_X)$}
     \label{fig:Floer}
\end{figure}

\begin{remark}
If one wants to investigate the case $\mu({\tilde{p}_{k_+}}) - \mu({\tilde{q}_{k_-}}) \geq 2$, then it is more reasonable to consider $\R_{\geq 0}^{N-1}$ in \eqref{ev_Y}, as opposed to $\R_{>0}^{N-1}$. 
 However, we only need to study the case  $\mu({\tilde{p}_{k_+}}) - \mu({\tilde{q}_{k_-}}) = 1$ for our purpose. To this end, it suffices to deal only with the case $\R_{>0}^{N-1}$. Note that, for $((\mathbf{v},\mathbf{z}),\mathbf{t})$ with $t_i=0$ for some $i$, $\tilde{v}_i$ and $\tilde{v}_{i+1}$ can be glued, and therefore $((\mathbf{v},\mathbf{z}),\mathbf{t})$ is an interior point after Floer-Gromov compactification.
\end{remark}

Let $\mathbf{B} := ( B_1, \dots, B_\ell ) \in (\pi_2(X) )^\ell$, for $\ell \in \mathbb N$, with $B_j \cdot \Sigma \neq0$ for every $j = 1, \dots, \ell$ as in Definition \ref{def_moduli_B}.
We consider two maps
\begin{align*}
    &\widetilde{\aug}_Y : \mathcal{M}_{N, \ell; k_-, k_+ }( \mathbf{A}; H; J_Y ) \rightarrow (\Sigma\times\mathbb{N})^\ell \\
    &\widetilde{\aug}_Y ( \mathbf{v}, \mathbf{z} ) : =\big((\pi_{Y} \circ {v}_1 (z_{1,1}),\m(\gamma_{z_{1,1}})), \dots, (\pi_{Y} \circ{v} _N (z_{N, \ell_N}),\m(\gamma_{z_{N, \ell_N}})) \big ), \nonumber  \\[1.5ex]
    &\widetilde{\aug}_X:\mathcal{P}_\ell(\mathbf{B};J_X) \to (\Sigma\times\mathbb{N})^\ell	\\
    &\widetilde{\aug}_X ( \mathbf{u} ) : = \big((u_1(\infty),u_1\cdot\Sigma),\dots, (u_\ell(\infty),u_\ell\cdot\Sigma)\big), \nonumber
\end{align*}
where $\gamma_{z_{i,j}}$ are periodic Reeb orbits that $v_i$ converges to at $z_{i,j}\in\Gamma_i$, and $\mathbf{u} = ([u_1], \dots, [u_\ell])$.
Putting these maps and the one in \eqref{ev_Y} together, we define 
\begin{align*}
  &\widetilde{\mathrm{EV}}: \mathcal{M}_{N,\ell;k_-,k_+}(\mathbf{A}; H; J_Y) \times \mathbb{R}^{N - 1}_{>0}   \times \mathcal{P}_\ell(\mathbf{B};J_X) \rightarrow Y^{2N} \times (\Sigma\times\mathbb{N})^\ell \times (\Sigma\times\mathbb{N})^\ell \\[1ex]
    &\widetilde{\mathrm{EV}}\big((\mathbf{v},\mathbf{z}), \mathbf{t}, \mathbf{u}\big):= \big(\widetilde{\ev}_{Z_Y}((\mathbf{v},\mathbf{z}), \mathbf{t}), \widetilde{\aug}_Y(\mathbf{v},\mathbf{z}), \widetilde{\aug}_X(\mathbf{u}) \big ).
\end{align*}

\begin{definition}
    For $\ell \in \mathbb N$, we define the moduli space
    \[
    \mathcal{M}_{N, \ell}(\tilde{q}_{k_-},\tilde{p}_{k_+};\mathbf{A}, \mathbf{B}; H; J_X) 
       := \widetilde{\mathrm{EV}}^{-1} \big (  W^s_{Z_Y} (\tilde{p}) \times  
        \Delta_{Y}  ^{N-1} \times W^u_{Z_Y}(\tilde{q}) \times {\Delta}_{(\Sigma\times \mathbb{N})^\ell}  \big ),
    \]
    where ${\Delta}_{(\Sigma\times \mathbb{N})^\ell} : = \left \{ (\mathbf{x}, \mathbf{y}) \in (\Sigma \times \mathbb{N})^{\ell} \times (\Sigma \times \mathbb N)^{\ell} \, \middle | \, \mathbf{x} = \mathbf{y} \right \}$,\footnote{ In the case of $N=1$, this means $\widetilde{\mathrm{EV}}^{-1} \big (  W^s_{Z_Y} (\tilde{p}) \times W^u_{Z_Y}(\tilde{q}) \times {\Delta}_{(\Sigma\times \mathbb{N})^\ell}  \big )$.}
    see Figure \ref{fig:Floer}.
    The subspace of simple elements is defined by
    \begin{align*}
       & \mathcal{M}^*_{N, \ell}(\tilde{q}_{k_-},\tilde{p}_{k_+};\mathbf{A},\mathbf{B}; H; J_X) \\ 
       & \hspace{1cm} : = \mathcal{M}_{N, \ell}(\tilde{q}_{k_-},\tilde{p}_{k_+};\mathbf{A}, \mathbf{B}; H; J_X) \, \cap \, \big ( ( \mathcal{M}^*_{N, \ell; k_-, k_+}(\mathbf{A};H;J_Y) \times \mathbb{R}^{N-1}_{>0} ) \times \mathcal{P}^*_\ell(\mathbf{B};J_X) \big ).
    \end{align*}        
\end{definition}

\begin{remark}
In view of Proposition \ref{bijection} and Remark \ref{rem:hol_plane}, we think of $[u_1],\dots,[u_\ell]$ in $\ ( ((\mathbf{v},\mathbf{z}), \mathbf{t}), \mathbf{u}  ) \in \mathcal{M}_{N, \ell}(\tilde{q}_{k_-},\tilde{p}_{k_+};\mathbf{A}, \mathbf{B}; H; J_X)$ as unparametrized $J_{\widehat{W}}$-holomorphic planes in $\widehat{W}$ whose asymptotic  orbits coincide with the periodic Reeb orbits that $v_1,\dots,v_N$ converge to at $\mathbf{z}$.	
\end{remark}

There are canonical projection maps 
\begin{equation}\label{projection0}
\begin{split}
	&\Pi:\mathcal{M}_{N, \ell}(\tilde{q}_{k_-},\tilde{p}_{k_+};\mathbf{A}; H; J_Y) \longrightarrow \mathcal{N}_{N, \ell}(q,p;\mathbf{A};J_\Sigma),  \\[1 ex]
    &\Pi:\mathcal{M}_{N, \ell}(\tilde{q}_{k_-},\tilde{p}_{k_+};\mathbf{A},\mathbf{B}; H; J_X) \longrightarrow \mathcal{N}_{N, \ell}(q,p;\mathbf{A},\mathbf{B};J_X), 
\end{split}
\end{equation}
defined by $\Pi ((\mathbf{v},\mathbf{z}), \mathbf{t}) : = ((\mathbf{w},\mathbf{z}), \mathbf{t})$, and $\Pi (((\mathbf{v},\mathbf{z}), \mathbf{t}), \mathbf{u}) := (((\mathbf{w},\mathbf{z}), \mathbf{t}), \mathbf{u})$, where $\mathbf{w} := (\pi_{ Y}\circ {v}_1,\dots,\pi_{Y}\circ {v}_{N})$. 
We have the same projections with $\mathcal{M}_{N, \ell}$ and $\mathcal{N}_{N, \ell}$ replaced by $\mathcal{M}_{N, \ell}^*$ and $\mathcal{N}_{N, \ell}^*$. The following transversality result is, roughly speaking, a consequence of an automatic transversality property in the fiber direction of $\Pi$. A corresponding statement for Floer continuation cylinders will be  discussed in Proposition \ref{continuation_transversality} and below.

\begin{prop}\label{prop:5.9_DL2}(\cite[Proposition 5.9]{DL2})
	Let $\mathcal{J}_X^\reg$ be the residual subset of $\mathcal{J}_X$ given in Proposition \ref{transversality_chain}. For $J_X\in\mathcal{J}_X^\reg$ and its cylindrical part $J_Y$, the moduli spaces 
\[
	 \mathcal{M}^*_{N, \ell=0}(\tilde{q}_{k_-},\tilde{p}_{k_+};\mathbf{A}; H; J_Y) \qquad \textrm{and}\qquad \mathcal{M}^*_{N, \ell}(\tilde{q}_{k_-},\tilde{p}_{k_+};\mathbf{A},\mathbf{B}; H; J_X)
\]
 are smooth manifolds of dimension 
\begin{equation*}\label{eq:index_computation1}
 \mu (\tilde{p}_{k_+}) - \mu (\tilde{q}_{k_-})+N-1.	
\end{equation*}
\end{prop}

\begin{definition} 
	We define the moduli space of split Floer cylinders
\begin{align*}
    & \mathcal{M}_{N}(\tilde{q}_{k_-},\tilde{p}_{k_+}; H; J_X) \\
    & \hspace{1cm}  : =\bigcup_{\mathbf{A}}\Big(\mathcal{M}_{N, \ell = 0 }(\tilde{q}_{k_-},\tilde{p}_{k_+};\mathbf{A}; H; J_Y)   
    \,\cup\, 
   \bigcup_{ \ell \in \mathbb{N}  } \bigcup_{\mathbf{B}}\mathcal{M}_{N, \ell}(\tilde{q}_{k_-},\tilde{p}_{k_+};\mathbf{A},\mathbf{B}; H; J_X)\Big)
\end{align*}
and the subspace of simple elements by 
\begin{align*}
    &\mathcal{M}^*_{N}(\tilde{q}_{k_-},\tilde{p}_{k_+}; H; J_X) \\
    &\hspace{1cm}: = \bigcup_{\mathbf{A}} \Big(\mathcal{M}^*_{N, \ell = 0 }(\tilde{q}_{k_-},\tilde{p}_{k_+};\mathbf{A}; H; J_Y) \,\cup\, 
    \bigcup_{ \ell \in \mathbb{N}  } \bigcup_{\mathbf{B}}\mathcal{M}^*_{N, \ell}(\tilde{q}_{k_-},\tilde{p}_{k_+};\mathbf{A},\mathbf{B}; H; J_X)\Big)
\end{align*}
where the unions $\bigcup\limits_\mathbf{A}$ and $\bigcup\limits_\mathbf{B}$ range over $\mathbf{A}\in (\pi_2(\Sigma))^N$ and $\mathbf{B}:=(B_1,\dots,B_\ell) \in (\pi_2(X))^\ell$ with $B_j\cdot\Sigma\neq 0$ for all $j$, respectively in both cases.
\end{definition}

There is a free $\mathbb{R}^N$-action on $\mathcal{M}_{N}(\tilde{q}_{k_-},\tilde{p}_{k_+}; H; J_X)$ given by translating the domain $\mathbb{R}\times S^1$ (together with punctures $\mathbf{z}$) of each $\tilde{v}_i$ in the $\mathbb{R}$-direction. 
For $N=0$, let
\begin{equation*}
    \mathcal{M}_{N=0}(\tilde{q}_{k_-},\tilde{p}_{k_+}) := 
\begin{cases}
W^s_{Z_Y}(\tilde{p}) \cap W^u_{Z_Y}(\tilde{q}) & \text{ if }	\, k_-=k_+,\\[0.5ex]
\emptyset & \text{ if }\, k_-\neq k_+.
\end{cases}
\end{equation*}
It inherits an orientation from oriented stable and unstable manifolds, see Section \ref{sec:Morse_Gysin}. 
It also carries a natural $\mathbb{R}$-action. Putting all these together, we define 
\begin{align*}
    \overbar{\mathcal{M}}(\tilde{q}_{k_-},\tilde{p}_{k_+}; H; J_X) &:= \mathcal{M}_{N=0}(\tilde{q}_{k_-},\tilde{p}_{k_+})/\mathbb{R} \;\cup\; \bigcup_{N \in\mathbb N} \mathcal{M}_{N}(\tilde{q}_{k_-},\tilde{p}_{k_+}; H; J_X)/\mathbb{R}^N, \\
    \overbar{\mathcal{M}}^*(\tilde{q}_{k_-},\tilde{p}_{k_+}; H; J_X) &:= \mathcal{M}_{N=0}(\tilde{q}_{k_-},\tilde{p}_{k_+})/\mathbb{R} \;\cup\; \bigcup_{N \in\mathbb N} \mathcal{M}_{N}^*(\tilde{q}_{k_-},\tilde{p}_{k_+}; H; J_X)/\mathbb{R}^N .
\end{align*}
The latter space is  a smooth manifold of dimension $\mu(\tilde{p}_{k_+})-\mu(\tilde{q}_{k_-})-1$ by the Morse-Smale property of $(f_Y,Z_Y)$ and Proposition \ref{prop:5.9_DL2}.

\medskip

Next we endow $\mathcal{M}^*_{N, \ell=0}(\tilde{q}_{k_-},\tilde{p}_{k_+};\mathbf{A}; H; J_Y)$ and $\mathcal{M}^*_{N, \ell}(\tilde{q}_{k_-},\tilde{p}_{k_+};\mathbf{A},\mathbf{B}; H; J_X)$ with orientations according to the fibered sum rule introduced in Definition \ref{def:fiber_sum}.
We offer an explanation only for those contributing to the split Floer chain complex, namely 
\begin{equation}\label{eq:orientation_moduli_2}
    \mathcal{M}^*_{N = 1, \ell = 0} (\hat{q}_{k_-},\check{p}_{k_+}; A; H; J_Y) \quad \text{and}  \quad \mathcal{M}^*_{N = 1, \ell = 1} (\hat{p}_{k_-},\check{p}_{k_+};A = 0, B;H; J_X),
\end{equation}
see Proposition \ref{prop:cpt} below. 
We first orient $\mathcal{M}^*_{N=1, \ell = 0; k_-, k_+}(A; H; J_Y)$ as follows. 
For $\tilde{v} \in \mathcal{M}^*_{N=1, \ell = 0; k_-, k_+}(A; H; J_Y)$, the operator $D_{\tilde{v}}$ obtained by linearizing the Floer equation \eqref{Floer} at $\tilde{v}$ is Fredholm on suitable Sobolev spaces,
\[
D_{\tilde v}: W^{1,p,\delta}_{\mathbf{V}}(\tilde v^*T(\R\times Y))\to L^{p,\delta}(\mathrm{Hom}^{0,1}(T(\R\times S^1\setminus\Gamma),\tilde v^*T(\R\times Y))).
\]
We refer to \cite[Section 5.2]{DL2} for the precise definition of such Sobolev spaces and for properties of $D_{\tilde v}$ we recall  below. The operator $D_{\tilde{v}}$ can be expressed as 
\begin{equation} \label{decomposition_operator}
    D_{\tilde{v}} = \begin{pmatrix}
            D^\mathbb{C}_{\tilde{v}} & M \\
            0 &\dot{D}^\Sigma_{\pi_Y  \circ v}
        \end{pmatrix}
\end{equation}
with respect to the splitting $\tilde{v}^* T (\mathbb R \times Y) \cong (\mathbb R{\partial_r} \oplus \mathbb R R) \oplus (\pi_Y \circ v)^* T \Sigma$. The operator $\dot{D}^{\Sigma}_{\pi_Y \circ v}$ is the linearized operator associated with the $J_\Sigma$-holomorphic curve $\pi_Y \circ v$ and has Fredholm index $\dim \Sigma + 2 c_1^{T\Sigma}([\pi_Y \circ v])$. Moreover, $D^{\mathbb C}_{\tilde{v}}$ is a surjective Fredholm operator with Fredholm index equal to 1. The kernel of $D^{\mathbb C}_{\tilde{v}}$ is generated by the Reeb vector field $R$, which amounts to rotating $\tilde{v}$ using the flow of $R$. We orient $\ker  D^\mathbb{C}_{\tilde{v}}$ by $R$. The off-diagonal entry $M$ is a compact operator.
Hence, the Fredholm index of $D_{\tilde{v}}$ can be computed as
\begin{equation} \label{Fredholm_index}
    \ind D_{\tilde{v}} = \ind \dot{D}^\Sigma_{\pi_Y  \circ v} + \ind D^{\mathbb C}_{\tilde{v}} = \dim Y + 2 c_1^{T\Sigma}([\pi_Y \circ v]). 
\end{equation}
We have the exact sequence $0\to \ker  D^\mathbb{C}_{\tilde{v}} \to \ker D_{\tilde{v}} \to \ker \dot{D}^\Sigma_{\pi_Y  \circ v}\to 0$ given by inclusion and projection. 
This leads to a canonical isomorphism
\begin{equation}\label{eq:tangent_splitting}
    T_{\tilde{v}} \big ( \mathcal{M}^*_{N=1, \ell = 0; k_-, k_+}(A; H; J_Y) \big ) \cong \mathbb{R} R_{\tilde v} \oplus T_{\pi_Y \circ v} \big ( \mathcal{N}^*_{N=1, \ell = 0}(A; J_\Sigma) \big )
\end{equation}
where $R_{\tilde v}$ is the Reeb vector field along $\tilde v$.

As mentioned after \eqref{eq:orientation_moduli_0}, $\mathcal{N}^*_{N=1, \ell = 0}(A; J_\Sigma)$ carries a natural orientation. This defines an orientation on $\mathcal{M}^*_{N=1, \ell = 0; k_-, k_+}(A; H; J_Y)$ by declaring the isomorphism in \eqref{eq:tangent_splitting} to be orientation-preserving.

To orient $\mathcal{M}^*_{N = 1, \ell = 0} (\hat{q}_{k_-},\check{p}_{k_+}; A; H; J_Y)$, we use the identification 
\[
\begin{aligned}
    \mathcal{M}^*_{N=1, \ell = 0} (\hat{q}_{k_-},\check{p}_{k_+}; A;H; J_Y) &= W^s_{Z_Y} (\check{p}) \prescript{}{\iota^s_Y} \times^{}_{\ev_{+,0}} \mathcal{M}^*_{N=1, \ell = 0; k_-, k_+}(A;H; J_Y) \prescript{}{\ev_{-,0}} \times^{}_{\iota^u_Y} W^u_{Z_Y} (\hat{q}) \\[1ex]
    \tilde{v} &\mapsto (\ev_{+,0} (\tilde{v}), \tilde{v}, \ev_{-,0} (\tilde{v})),
\end{aligned}
\]
where $\iota^s_Y$ and $\iota^u_Y$ are natural inclusions. The fiber product space is oriented by the fibered sum rule, where the orientations on $W^s_{Z_Y} (\check{p})$ are $W^u_{Z_Y} (\hat{q})$ are chosen in Section \ref{sec:Morse}. The space $\mathcal{M}^*_{N = 1, \ell = 0} (\hat{q}_{k_-},\check{p}_{k_+}; A; H; J_Y)$ inherits an orientation from the above identification.

Next we orient the latter space in \eqref{eq:orientation_moduli_2}. Like \eqref{eq:tangent_splitting}, we have an isomorphism
\begin{equation}\label{eq:orientation_puctured}
    T_{(\tilde{v},z)} \big ( \mathcal{M}^*_{N=1, \ell = 1; k_- , k_+}(A=0; H; J_Y) \big ) \cong \mathbb{R} R_{\tilde v} \oplus T_{(\pi_Y \circ v, z)} \big ( \mathcal{N}^*_{N = 1, \ell = 1}(A=0; J_\Sigma) \big ).
\end{equation}
Using the orientation on $\mathcal{N}^*_{N = 1, \ell = 1}(A=0; J_\Sigma)$ chosen after Definition \ref{def:fiber_sum} and the above isomorphism, we endow $\mathcal{M}^*_{N=1, \ell = 1; k_- , k_+}(A=0; H; J_Y)$ with an orientation.

Then we identify $\mathcal{M}^*_{N = 1, \ell = 1} (\hat{p}_{k_-},\check{p}_{k_+};A = 0, B;H; J_X)$ with the fiber product space
\[
    W^s_{Z_Y} (\check{p}) \prescript{}{\iota^s_Y} \times^{}_{\ev_{+,0}} \big (  \mathcal{P}^*_{\ell = 1} (B; J_X) \prescript{}{\widetilde{\aug}_X} \times^{}_{\,\widetilde{\aug}_Y\,} \mathcal{M}^*_{N = 1, \ell = 1;k_-, k_+} (A = 0;H;J_Y)  \big ) \prescript{}{\ev_{-,0}} \times^{}_{\iota^u_Y} W^u_{Z_Y} (\hat{p}),
\]
which is oriented by the fibered sum rule. Note that $\mathcal{P}^*_{\ell = 1}(B;J_X)$ is oriented as discussed after Definition \ref{def:fiber_sum}.
This induces an orientation on $\mathcal{M}^*_{N = 1, \ell = 1} (\hat{p}_{k_-},\check{p}_{k_+};A = 0, B;H; J_X)$.

\medskip

We are now in a position to introduce the following remarkable result due to Diogo and Lisi. As the following proposition is one crucial ingredient of the main result of this paper, we provide an outline of their proof. To state the result, we denote by $\# \mathcal{M}$ the cardinality of $\mathcal{M}$ counted with sign for a compact oriented zero-dimensional manifold $\mathcal{M}$.

\begin{prop}\label{prop:cpt} (\cite[Theorem 9.1]{DL1}) 
Let $H\in\mathcal{H}$ and $J_X \in \mathcal{J}^{\reg}_X$. Let $\tilde{p}_{k_+}\in \Crit f_{Y^H_{k_+}}$ and $\tilde{q}_{k_-} \in \Crit f_{Y^H_{k_-}}$ be 1-periodic orbits of type III with $\mu({\tilde{p}_{k_+}}) - \mu({\tilde{q}_{k_-}}) =1$. Then all elements of $\overbar{\mathcal{M}}(\tilde{q}_{k_-},\tilde{p}_{k_+};H;J_X)$ are simple, i.e.
\[
\overbar{\mathcal{M}}(\tilde{q}_{k_-},\tilde{p}_{k_+};H;J_X) = \overbar{\mathcal{M}}^*(\tilde{q}_{k_-},\tilde{p}_{k_+};H;J_X).
\] 
Moreover if $\overbar{\mathcal{M}}(\tilde{q}_{k_-},\tilde{p}_{k_+};H;J_X)$ is not empty, then it is one of the following forms: 
\begin{enumerate}[(a)]
        \item It holds $k_-=k_+$ and therefore
        \[
        \overbar{\mathcal{M}}(\tilde{q}_{k_-},\tilde{p}_{k_+};H;J_X) = \mathcal{M}_{N=0}(\tilde{q}_{k_-},\tilde{p}_{k_+})/\mathbb{R}.
        \]
        
        \item We have 
        \[
        \overbar{\mathcal{M}}(\tilde{q}_{k_-},\tilde{p}_{k_+};H;J_X) = \bigcup_{A\in\pi_2(\Sigma)}\mathcal{M}_{N=1, \ell = 0}(\hat{q}_{k_-},\check{p}_{k_+};A;H;J_Y)/\mathbb{R}
        \] 
        where the union is taken over all nonzero $ A \in \pi_2(\Sigma)$ satisfying $\ind_{f_\Sigma}(p)-\ind_{f_\Sigma}(q)+2c_1^{T\Sigma}(A)=2$ and $K\omega_\Sigma(A)=k_+- k_-$. Moreover, for all such $A$, we have   
		\[
		\# \big(\mathcal{M}^*_{N = 1, \ell = 0} (\hat{q}_{k_-}, \check{p}_{k_+};A;H;J_Y)/\mathbb{R}\big)=-K\omega_\Sigma(A) \cdot  \#\big(  \mathcal{N}^*_{N = 1, \ell = 0} (q,p;A;J_\Sigma) / \mathbb{C}^*\big)
		\]
		where $\mathbb{C}^*$ acts on $\mathbb{CP}^1$ fixing $0$ and $\infty$.
        See the left-hand side of Figure \ref{fig:M_moduli}.
        
        \item We have $q=p$ and 
        \[
        \overbar{\mathcal{M}}(\tilde{q}_{k_-},\tilde{p}_{k_+};H;J_X) = \bigcup_{B\in\pi_2(X)} \mathcal{M}_{N=1, \ell = 1}(\hat{p}_{k_-},\check{p}_{k_+};A=0 , B;H;J_X)/\mathbb{R}
        \]
        where the union is taken over all $B \in \pi_2(X)$ satisfying $c_1^{TX}(B)-K\omega(B)=1$ and $K\omega(B)=k_+ - k_-$. Moreover, for all such $B$, we have
		\begin{align*}
            & \#\big(\mathcal{M}^*_{N = 1, \ell = 1} (\hat{p}_{k_-}, \check{p}_{k_+}; A=0,B;H;J_X) / \mathbb R \big) \\[.5ex]
            &\hspace{2cm} = K\omega(B)\cdot \#\big( \mathcal{N}^*_{N=1, \ell = 1} (p,p;A=0,B;J_X) / \mathbb{C}^*\big).
        \end{align*}
 		See the right-hand side of Figure \ref{fig:M_moduli}.
\end{enumerate}
\end{prop}
\begin{figure}[h]
     \centering
     \includegraphics[height = 5cm]{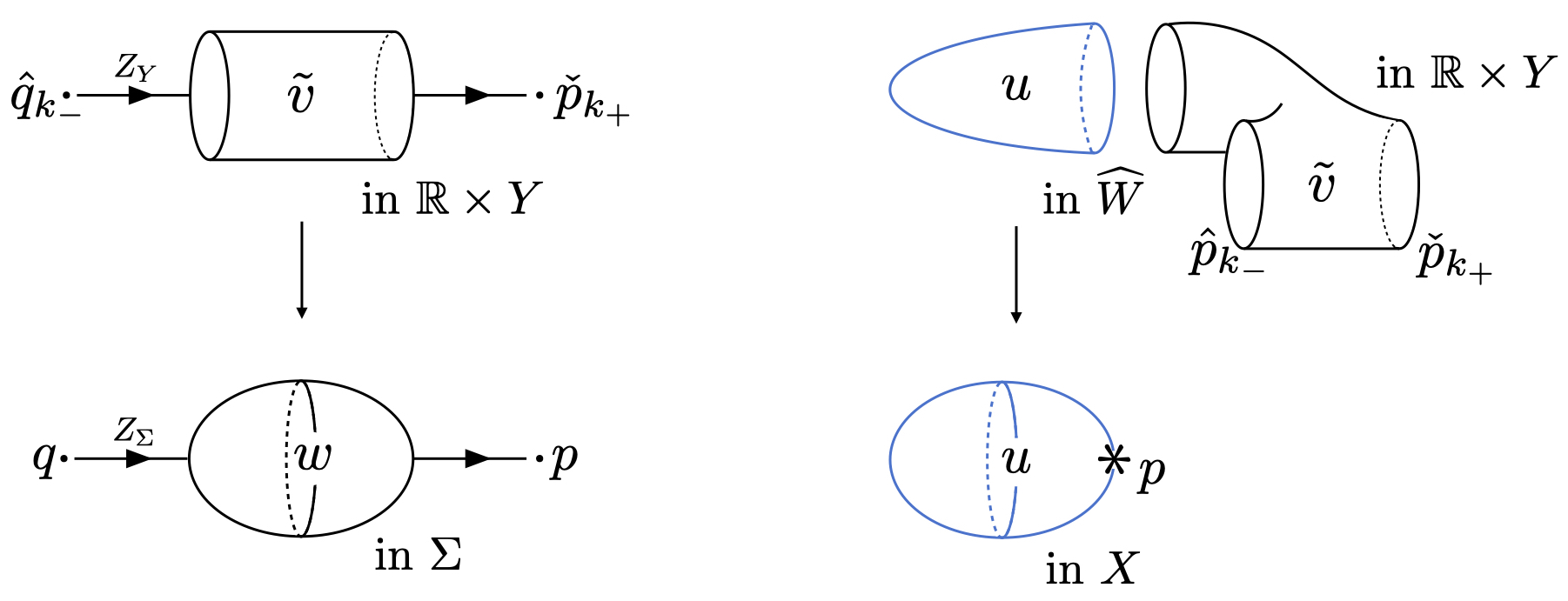}
     \caption{Split Floer cylinders and its shadows}
     \label{fig:M_moduli}
\end{figure}
\begin{proof}[Sketch of the proof]
 Each connected component of the space ${\mathcal{M}}(\tilde{q}_{k_-},\tilde{p}_{k_+};H;J_X)$ is fibered over some $\mathcal{N}_{N, \ell}(q,p;\mathbf{A},\mathbf{B};J_X)$ as pointed out in \eqref{projection0}. 
 From the hypothesis $\mu({\tilde{p}_{k_+}}) - \mu({\tilde{q}_{k_-}})=1$ and the transversality result in Proposition \ref{transversality_chain}, we deduce that it is actually fibered over $\mathcal{N}_{N, \ell}^*(q,p;\mathbf{A},\mathbf{B};J_X)$. Then we can conclude that only three cases (a), (b), and (c) can occur for dimension reasons. We refer to Proposition \ref{continuation_transversality}.(b) below for an analogous statement and its proof for Floer continuation cylinders.

Statement (a) follows from the fact that there is no nontrivial Floer cylinder with both asymptotic ends in $Y^H_{k_-=k_+}$ for energy reasons. 

Statement (b) is proved by showing that the projection 
		\[
	\widehat{\Pi}:\mathcal{M}^*_{N = 1, \ell = 0} (\hat{q}_{k_-}, \check{p}_{k_+};A;H;J_Y) / \mathbb R \rightarrow \mathcal{N}^*_{N = 1, \ell = 0} (q,p;A;J_\Sigma) / \mathbb{C}^*.
		\]
		induced by \eqref{projection0} is a surjective orientation-reversing $K\omega_\Sigma(A)$-to-$1$ map. Each $(\mathbf{v},\mathbf{z})=(\tilde v=(b,v),\emptyset)\in\mathcal{M}^*_{N = 1, \ell = 0} (\hat{q}_{k_-}, \check{p}_{k_+};A;H;J_Y)$ yields a $S^1\times S^1$-family of solutions
\[
\tilde v_{(\theta,\vartheta)}(s,t):=(b(s,t+\vartheta), \varphi_R^{\frac{\theta}{K}}\circ v(s,t+\vartheta)), \qquad (\theta,\vartheta)\in S^1\times S^1
\]
in $\mathcal{M}^*_{N=1, \ell=0; k_-, k_+ }(A; H; J_Y )$ since $J_Y$ is $S^1$-invariant and $H$ does not depend on the  $S^1$-coordinate.
Then, $\tilde v_{(\theta,\vartheta)}$ belongs to $\mathcal{M}^*_{N = 1, \ell = 0} (\hat{q}_{k_-}, \check{p}_{k_+};A;H;J_Y)$ if and only if 
\[
\ev_{-}(\tilde v)=\ev_{-}(\tilde v_{(\theta,\vartheta)})= \varphi_R^{\frac{\theta}{K}}\circ\varphi_R^{\frac{\vartheta k_-}{K}} \ev_{-}(\tilde v),\qquad \ev_{+}(\tilde v)=\ev_{+}(\tilde v_{(\theta,\vartheta)})= \varphi_R^{\frac{\theta}{K}}\circ\varphi_R^{\frac{\vartheta k_+}{K}} \ev_{+}(\tilde v).
\]
This is equivalent to $\theta+\vartheta k_-\in\mathbb{Z}$ and $\theta+\vartheta k_+\in\mathbb{Z}$, and in turn also to $(k_+-k_-)\vartheta\in\mathbb{Z}$ with $\theta$ uniquely determined by $\vartheta$. 
This holds for 
\begin{equation}\label{eq:vartheta}
\vartheta\in\big\{\tfrac{1}{k_+-k_-},\dots,\tfrac{k_+-k_-}{k_+-k_-}\big\}.	
\end{equation}

Furthermore, for any $(\mathbf{w},\mathbf{z})=(w,\emptyset)\in \mathcal{N}^*_{N = 1, \ell = 0} (q,p;A;J_\Sigma)$, there exists a unique $(\tilde v,\emptyset)\in \mathcal{M}^*_{N=1, \ell=0; k_-, k_+ }(A; H; J_Y )$ fibered over $(w,\emptyset)$ up to the $S^1\times S^1$-action. This is proved by establishing a correspondence between Floer cylinders and $J_Y$-holomorphic cylinders, see \cite[Sections 6 and 7]{DL1} for details and see also Remark \ref{rmk:hol_Floer} below.
This uniqueness property and \eqref{eq:vartheta} yield that if $\widehat{\Pi}([(\tilde{v},\emptyset)])=[(w,\emptyset)]$, then $\widehat{\Pi}^{-1}([(w,\emptyset)])$ has cardinality $k_+-k_-$, which equals $K\omega_\Sigma(A)$. From our orientation conventions, one can check that $\widehat{\Pi}$ is orientation-reversing. This completes the proof of statement (b).

Statement (c) is proved analogously. 
\end{proof}

    \begin{remark}\label{rmk:hol_Floer}
        The correspondence between Floer cylinders and $J_Y$-holomorphic cylinders established in \cite[Sections 6 and 7]{DL1} continues to hold for $\bigvee$-shaped Hamiltonians. We remark that, in contrast to the setting in \cite{DL1}, our Floer cylinder could have $\m(\ev_\pm(v))<0$, and the corresponding $J_Y$-holomorphic cylinder $(a,u)$ satisfies 
        \[
        \m(\ev_+(v))\Big(\lim_{s\to+\infty}a\Big)=+\infty,\qquad \m(\ev_-(v))\Big(\lim_{s\to-\infty}a\Big)=-\infty.
        \]
    \end{remark}

\begin{remark}
    Our orientation conventions in Section \ref{sec:Morse} and $\eqref{eq:orientation_moduli_2}$ differ from those in \cite{DL1}, and this is why   Proposition \ref{prop:cpt}.(b) and the corresponding result in \cite[Theorem 9.1]{DL1} have different signs. 
\end{remark}

\begin{remark} \label{rmk: indep_of_point}
The signed count $\#\big( \mathcal{N}^*_{N=1, \ell = 1} (p,p;A=0,B;J_X) / \mathbb{C}^* \big)$ in Proposition \ref{prop:cpt}.(c) and \eqref{eq:delta_c} is independent of $p \in \Crit f_\Sigma$, which can be shown by a cobordism argument, see \cite[Section 3.2]{CM18}.
\end{remark}

\subsection{Split Floer chain complex}

Let $a,b \in \mathbb{R}$ with $a<b$. We take $H \in \mathcal{H}$, which is of the form  $H(r, y)=h(e^r)$ on $\mathbb{R} \times Y\subset \widehat{W}$ by definition, such that $h(0) > -a$.
We define a finite set 
\begin{equation*}\label{contractible_slope}
    \m(a,b\, ;H) : = \left\{  k \in m_W \mathbb{Z} \, \middle | \, \exists\, \beta^H_k\in\mathbb{R} \,:\,h'(e^{\beta^H_k}) =\tfrac{k}{K} \textrm{ and } a < e^{\beta^H_k}h'(e^{\beta^H_k}) - h(e^{\beta^H_k}) < b  \right \}.
\end{equation*}
A $1$-periodic orbit $x\in\Crit\mathcal{A}_H$ of $X_H$ with $\m(x) \in \m(a,b\, ; H)$ is contractible in $\widehat{W}$ by the definition of $m_W$ and of type III by the action computation after \eqref{action}. Thus, for $k\in\m(a,b\, ;H)$, the constant $\beta^H_k$ is uniquely determined and appears in $(-\eta,\eta)$.

Now, we define the split Floer complex of $H$ and $J_X \in \mathcal{J}^{\reg}_X$.
The chain module is generated by critical points of $f_{Y^H_k}$ with $k\in\m(a,b\,;H)$, i.e.
\begin{equation*}\label{chain_group}
    \FC_*^{(a,b)}(H) := \bigoplus_{k \in \m(a,b\, ; H)} \, \bigoplus_{p \in \Crit f_\Sigma} \mathbb{Z} \langle \, \hat{p}_k, \check{p}_k \, \rangle,
\end{equation*}
with grading given by $\mu(\tilde{p}_k)$ defined in \eqref{grading}.
The boundary operator is defined to be the $\mathbb Z$-linear extension of
\begin{equation}\label{differential}
\partial: \FC_*^{(a,b)}(H)\to \FC_{*-1}^{(a,b)}(H)  ,\qquad    \tilde{p}_{k_+} \mapsto \sum_{\tilde{q}_{k_-} }\, \sum_{[\tilde{v}]} \,\epsilon([\tilde{v}])\tilde{q}_{k_-}
\end{equation}
where the first and second sums run over all generators $\tilde{q}_{k_-}$ satisfying $\mu(\tilde{p}_{k_+})-\mu(\tilde{q}_{k_-})=1$ and $[\tilde{v}] \in \overbar{\mathcal{M}}(\tilde{q}_{k_-},\tilde{p}_{k_+}; H; J_X)$, respectively. The orientation sign $\epsilon([\tilde{v}]) \in \{1, -1\}$ of $[\tilde{v}]$ is given by $\epsilon([\tilde{v}]) = 1$ if the orientation on the moduli space at $\tilde{v}$ coincides with the orientation induced by the $\mathbb{R}$-action, and $\epsilon([\tilde{v}]) = -1$ otherwise. 
This split chain complex can be identified with the corresponding ordinary Floer chain complex with a sufficiently stretched almost complex structure thanks to the neck-streching argument, see \cite{BO2,DL1} and also Remark \ref{rem:conti_two_levels}. It follows that the homology 
\begin{equation*}
    \FH_i^{(a,b)}(H):=\H_i \big(\FC_*^{(a,b)}(H), \partial \big)
\end{equation*}
is well-defined, isomorphic to the ordinary Floer homology of $H$, and independent of the choice of $J_X\in \mathcal{J}^{\reg}_X$ up to canonical isomorphisms.

\section{Split Rabinowitz Floer homology} \label{sec:split_Rabinowitz}

\subsection{Split Floer continuation cylinders}
We endow $\mathcal{H}$ with the standard strict partial order $\prec$, namely $H\prec L$ if $H < L$ pointwise. A smooth family of smooth maps $H_s:\widehat{W} \rightarrow \mathbb{R}$  parametrized by $s \in \mathbb R$ is said to be a {\it monotone homotopy for $H \prec L$} if the following conditions are fulfilled:
\begin{enumerate}[(i)]
    \item $H_s$ satisfies the conditions (i) and (ii) in the definition of $\mathcal{H}$ (see Section \ref{sec: Hamiltonians}).
    \item $H_s \equiv L$ for $s \leq -1$ and $H_s \equiv H$ for $s \geq 1$.
    \item $\frac{\partial}{\partial s} H_s \leq 0$.
\end{enumerate} 
As we will study split Floer cylinders for $H_s$, we again view $H,L\in\mathcal{H}$ and a monotone homotopy $H_s$ for $H \prec L$ as functions defined on $\mathbb{R}\times Y$ using the embedding $\mathbb{R}\times Y \subset \widehat{W}$. In particular, $H_s(r,y)=h_s(e^r)$ for some smooth family of smooth maps $h_s:\mathbb{R}\to\mathbb{R}$, and $X_{H_s}$ vanishes on $(-\infty,-\epsilon/8)\times Y$ for $\epsilon$ as in Section \ref{sec:acs}.

Let $H_s$ be a monotone homotopy for $H \prec L$, and let $\Gamma = \left \{ z_1, \dots, z_\ell \right \}$ be a possibly empty finite subset of $\mathbb{R} \times S^1$. 
For $J_Y\in\mathcal{J}_Y$, we refer to smooth solutions $\tilde{v}=(b,v):\mathbb{R}\times S^1 \setminus\Gamma \rightarrow \mathbb R\times Y$ of 
\begin{equation}\label{Floer_conti}
\partial_s \tilde{v} + J_Y(\tilde{v})(\partial_t \tilde{v} - X_{H_s}(\tilde{v}))=0
\end{equation}
which converge to a critical point of $\mathcal{A}_L$ at the negative end and to a critical point of $\mathcal{A}_H$ at the positive end, i.e. 
\begin{align*}
    &{\ev}_-(\tilde{v})=({\ev}_-(b),{\ev}_-(v)) := \lim_{s \to - \infty} \tilde{v}(s,\cdot) = \big(\lim_{s \to - \infty} b(s,\cdot), \lim_{s \to - \infty} v(s,\cdot) \big) \in \Crit \mathcal{A}_{L}, \\
    &{\ev}_+(\tilde{v})=({\ev}_+(b),{\ev}_+(v)) := \lim_{s \to + \infty} \tilde{v}(s,\cdot) = \big(\lim_{s \to + \infty} b(s,\cdot), \lim_{s \to + \infty} v(s,\cdot) \big) \in \Crit \mathcal{A}_H,  
\end{align*}
as {\it Floer continuation cylinders}. Note that a usual maximum principle applies to Floer continuation cylinders.
There is no loss of generality in assuming that $\Gamma$ consists of non-removable punctures.
We define the energy $E(\tilde{v})$ of a Floer continuation cylinder $\tilde{v}$ by \eqref{energy} with $H$ replaced by $H_s$.
A direct computation gives 
\begin{equation*}\label{eq:action_increase_cont}
    0\leq E_{\psi=e^r}(\tilde v)= \mathcal{A}_H(\ev_+(\tilde v))-\mathcal{A}_L(\ev_-(\tilde v)) + \int_{\mathbb{R}\times S^1 \setminus\Gamma } \frac{\partial H_s}{\partial s} (\tilde{v})ds \wedge dt ,
\end{equation*}
which implies $\mathcal{A}_H(\ev_+(\tilde v)) \geq \mathcal{A}_L(\ev_-(\tilde v))$.
A Floer continuation cylinder $\tilde{v}$ has finite energy if and only if $\tilde{v}$ converges at every $z\in\Gamma$ to the orbit cylinder of some periodic Reeb orbit as in \eqref{asymptotic}. 
Moreover, \eqref{multiplicity} holds also for $\tilde v$.
Thus every  nontrivial  Floer continuation cylinder $\tilde{v}=(b,v)$ has $\m (\ev_+(v))\geq \m (\ev_-(v))$. However, in contrast to nontrivial Floer cylinders, nontrivial Floer continuation cylinders $\tilde v=(b,v)$ can have $\m (\ev_+(v))= \m (\ev_-(v))$.

We observe that system of equations \eqref{eq:Floer_eq2} holds again with $h$ replaced by $h_s$. 
In particular, $\tilde v=(b,v)$ with finite energy projects to a $J_\Sigma$-holomorphic map $\pi_Y \circ v :\mathbb{R}\times S^1\setminus \Gamma \rightarrow \Sigma$ with finite energy. 
As before, due to the removable singularity theorem, this extends to a $J_\Sigma$-holomorphic map defined on $\mathbb{CP}^1$, denoted again by $\pi_Y \circ v :\mathbb{CP}^1\rightarrow \Sigma$. 

\medskip

Now we define moduli spaces which will be used in the construction of a split version of continuation homomorphisms.  
Let $H_s$ be a monotone homotopy for $H \prec L$. Let $J_X \in \mathcal{J}_X$, and let $J_Y$ and $J_\Sigma$ be the cylindrical and horizontal parts of $J_X$, respectively.
Let 
\[
\mathbf{A} := (A^H_1, \dots, A^H_{N_H}, A, A^L_1, \dots, A^L_{N_L}) \in (\pi_2(\Sigma))^N,\;\; N = N_H + N_L + 1,\;\; N_H, N_L \in \mathbb N \cup  \{ 0  \},
\]
and let 
\[
\ell = \ell_H + \ell_c + \ell_L, \quad \ell_H,\, \ell_c,\, \ell_L \in \mathbb{N} \cup \left \{ 0 \right \}.
\]
\begin{definition} 
    We define the moduli space 
    \begin{equation*} \label{continuation}
        \mathcal{M}_{N,\ell;k_-,k_+}(\mathbf{A}; H_s; J_Y)
    \end{equation*}
    consisting of pairs 
    \begin{align*}
        (\mathbf{v},\mathbf{z}) = \big ( (\tilde{v}^{H}_1, \dots, \tilde{v}^{H}_{N_{H}},\, \tilde{v}, \, \tilde{v}^L_1, \dots, \tilde{v}^L_{N_L}), (\Gamma^{H}_1, \dots, \Gamma^{H}_{N_{H}}, \,  \Gamma, \, \Gamma^L_1, \dots, \Gamma^L_{N_L} ) \big )
    \end{align*}
    such that for every $i=1,\dots,N_H$ and $j=1,\dots,N_L$
    \begin{enumerate}[(1)]
         \item $\Gamma^H_i : = ( z^{H}_{i,1}, \dots, z^{H}_{i, \ell_{H,i}} )$ is a finite sequence of pairwise distinct elements of $\mathbb R \times S^1$ for some $\ell_{H,i} \in \mathbb N \cup \left \{ 0 \right \}$  with $\ell_H = \sum\limits_{i=1}^{N_H} \ell_{H, i}$.
         For $\ell_{H,i} = 0$, we set $\Gamma^H_i = \emptyset$,
        \item $\tilde{v}^H_i=(b^H_i,v^H_i): \mathbb{R}\times S^1 \setminus \Gamma^H_i \rightarrow \mathbb{R} \times Y$ is a  nontrivial Floer cylinder with finite energy with respect to H such that $\pi_Y \circ v^H_i$ represents $A^H_i$,
        \item $\Gamma : = (z_{1}, \dots, z_{\ell_\mathrm{c}} )$ is a finite sequence of pairwise distinct elements of $\mathbb R \times S^1 $.
            For $\ell_{\mathrm{c}} = 0$, we set $\Gamma = \emptyset$.
        \item  $\tilde{v}=(b,v): \mathbb{R}\times S^1 \setminus \Gamma \rightarrow \mathbb{R}\times Y$ is a  Floer continuation cylinder with finite energy with respect to $H_s$ such that $\pi_Y \circ v$ represents $A$,
        \item $\Gamma^{L}_j : = ( z^{L}_{j,1}, \dots, z^{L}_{j, \ell_{L,j}}  )$ is a finite sequence of pairwise distinct elements of $\mathbb R \times S^1$ for some $\ell_{L,j} \in \mathbb N \cup \left \{ 0 \right \}$  with $\ell_{L} = \sum\limits_{j=1}^{N_L} \ell_{L, j}$.
        For $\ell_{L,j} = 0$, we set $\Gamma^{L}_j = \emptyset$,
        \item  $\tilde{v}^{L}_j=(b^{L}_j,v^{L}_j): \mathbb{R}\times S^1 \setminus \Gamma^{L}_j \rightarrow \mathbb{R} \times Y$ is a nontrivial  Floer cylinder with finite energy with respect to $L$ satisfying $\pi_Y \circ v^{L}_j$ represents $A^{L}_j$, 
    \end{enumerate}
    and in addition for every $i=1,\dots,N_{H}-1$ and $j=1,\dots,N_L-1$
        \[
        \begin{split}
            &\m(\ev_+(\tilde{v}^{H}_1)) = k_+,\quad \m(\ev_-(\tilde{v}^{H}_i)) = \m(\ev_+(\tilde{v}^{H}_{i+1})),\quad \m(\ev_-(\tilde{v}^{H}_{N_{H}}))=\m({\ev}_+(\tilde{v})), 	\\[0.5ex]
            & \m({\ev}_-(\tilde{v})) = \m({\ev}_+(\tilde{v}^L_1)), \quad \m(\ev_-(\tilde{v}^L_j)) =\m(\ev_+(\tilde{v}^L_{j+1})), \quad \m(\ev_-(\tilde{v}^L_{N_L})) = k_-.
        \end{split}
        \]
        
    We also define the subspace of simple elements by
    \[
        \mathcal{M}^*_{N,\ell;k_-,k_+}(\mathbf{A}; H_s; J_Y) \subset \mathcal{M}_{N,\ell;k_-,k_+}(\mathbf{A}; H_s; J_Y)
    \] 
    consisting of elements $(\mathbf{v},\mathbf{z})$ such that 
    \[
    \big((\pi_{ Y}\circ {v}^{H}_1,\dots,\pi_{Y}\circ {v}^{H}_{N_{H}},\pi_{Y}\circ {v},\pi_{ Y}\circ {v}^L_1,\dots,\pi_{Y}\circ {v}^L_{N_L}),\mathbf{z}\big) \in \mathcal{N}^*_{N, \ell}(\mathbf{A};J_\Sigma).
    \]      
\end{definition}

\medskip

Next, we consider the map
\begin{align}
    &\widetilde{\ev}^c_{Z_Y} : \mathcal{M}_{N, \ell; k_-, k_+} (\mathbf{A}; H_s; J_Y) \times \mathbb{R}^{N-1}_{>0} \longrightarrow Y^{2N_H} \times Y^2 \times Y^{2N_L} \label{eq:ev^c}\\[0.5ex]
    &\widetilde{\ev}^c_{Z_Y} \big ( (\mathbf{v}, \mathbf{z}), \mathbf{t}) := \big ( \ev_{+,0}(\tilde{v}^H_1), \ev_{-,0}(\tilde{v}^H_1), \varphi^{t^H_1}_{Z_Y}(\ev_{+,0}(\tilde{v}^H_2)), \ev_{-,0}(\tilde{v}^H_2), \dots, \nonumber \\ 
    &\hspace{4cm} \varphi^{t_c}_{Z_Y}(\ev_{+,0}(\tilde{v})), \ev_{-,0}(\tilde{v}),  \dots, \varphi^{t^L_{N_L}}_{Z_Y} (\ev_{+,0}(\tilde{v}^L_{N_L})), \ev_{-,0}(\tilde{v}^L_{N_L}) \big ), \nonumber 
\end{align}
where $\mathbf{t} := (t^H_1, \dots,t^H_{N_H-1},  t_c, t^L_1, \dots, t^L_{N_{L}})$. For $N = 1$, this means $\widetilde{\ev}^c_{Z_Y} : \mathcal{M}_{N, \ell; k_-, k_+} (\mathbf{A}; H_s; \allowbreak J_Y) \to Y^2$ with $(\tilde{v}, \mathbf{z}) \mapsto (\ev_{+,0}(\tilde{v}), \ev_{-,0}(\tilde{v}))$.

 Let $\tilde{p}^H_{k_+} \in \Crit f_{Y^H_{k_+}}$ and $\tilde{q}^L_{k_-} \in \Crit f_{Y^L_{k_-}}$ be 1-periodic orbits of type III for $k_\pm\in m_W\Z$.

\begin{definition}\label{def:conti_moduli1}
    For $\ell \in \mathbb{N} \cup \left \{ 0 \right \}$, we define 
    \begin{equation*}
        \mathcal{M}_{N, \ell} (\tilde{q}^L_{k_-}, \tilde{p}^H_{k_+}; \mathbf{A}; H_s; J_Y) := (\widetilde{\ev}^c_{Z_Y})^{-1} \big ( W^s_{Z_Y}(\tilde{p}) \times \Delta_Y^{N-1} \times W^u_{Z_Y}(\tilde{q}) \big ),
    \end{equation*}
    where $\widetilde{\ev}^c_{Z_Y}$ is given in \eqref{eq:ev^c}.\footnote{ For $N = 1$, or equivalently $N_H = N_L = 0$, this means $(\widetilde{\ev}^c_{Z_Y})^{-1} \big ( W^s_{Z_Y}(\tilde{p}) \times W^u_{Z_Y}(\tilde{q}) \big )$.}
    The subspace of simple elements is defined by
    \begin{align*}
        &\mathcal{M}^*_{N, \ell} (\tilde{q}^L_{k_-}, \tilde{p}^H_{k_+}; \mathbf{A}; H_s; J_Y) \\
        &\hspace{1cm}:=\mathcal{M}_{N, \ell} (\tilde{q}^L_{k_-}, \tilde{p}^H_{k_+}; \mathbf{A}; H_s; J_Y) \, \cap \, \big ( \mathcal{M}^*_{N, \ell; k_-, k_+} (\mathbf{A}; H_s; J_Y) \times \mathbb{R}^{N-1}_{>0} \big ).
    \end{align*}
    For $\ell = 0$, we write $\mathcal{M}_{N, \ell = 0} (\tilde{q}^L_{k_-}, \tilde{p}^H_{k_+}; \mathbf{A}; H_s; J_Y)$ and $\mathcal{M}^*_{N, \ell = 0} (\tilde{q}^L_{k_-}, \tilde{p}^H_{k_+}; \mathbf{A}; H_s; J_Y)$.
\end{definition}

\begin{figure}[h]
     \centering
     \includegraphics[height = 6cm]{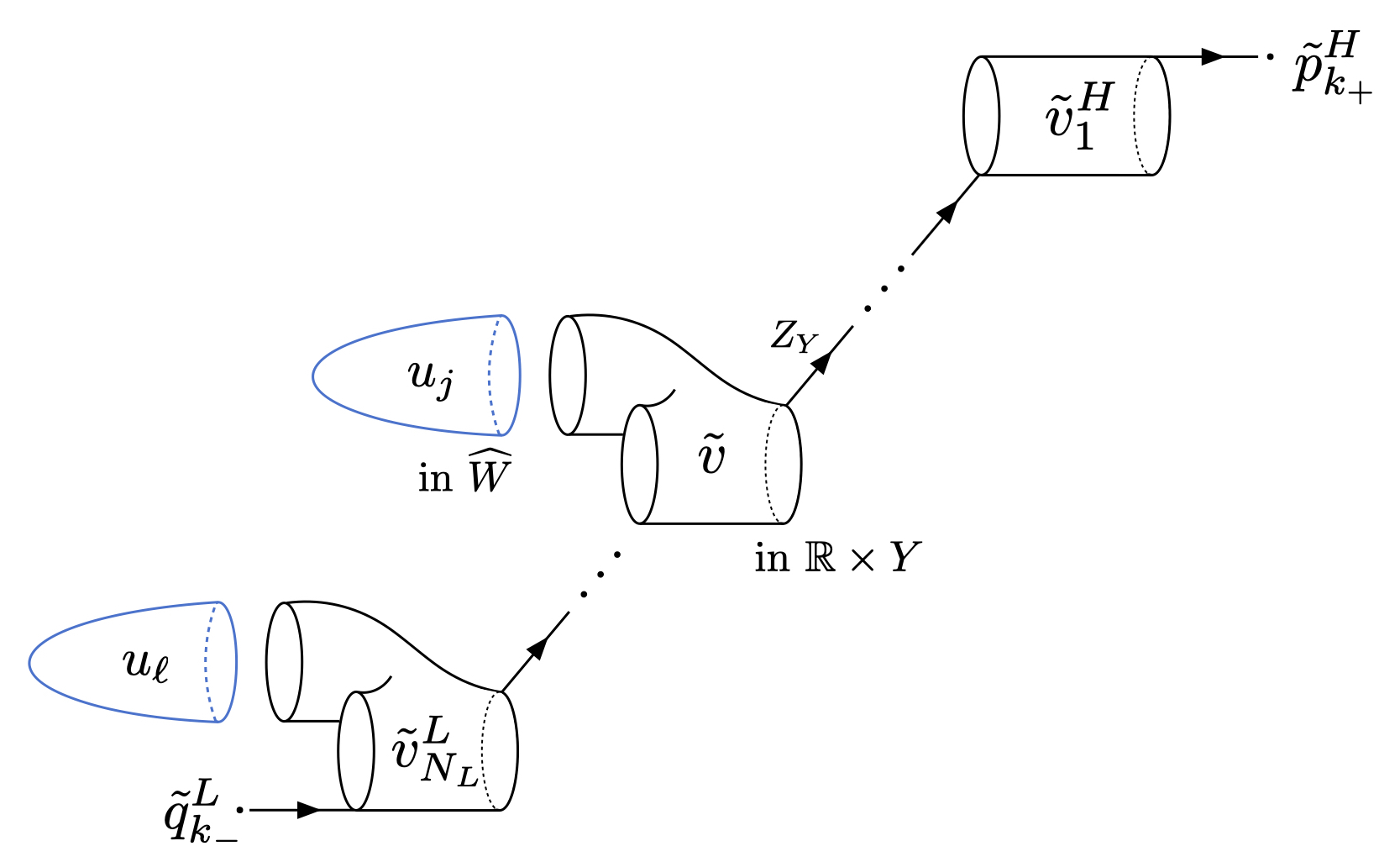}
     \caption{An element of $\mathcal{M}_{N, \ell} (\tilde{q}^L_{k_-}, \tilde{p}^H_{k_+}; \mathbf{A}, \mathbf{B}; H_s; J_X)$}
     \label{fig:cont}
\end{figure}

If $\ell=\ell_H+\ell_c+\ell_L \geq 1$, we consider
\[
\begin{split}
    &\widetilde{\aug}^c_Y : \mathcal{M}_{N,\ell;k_-,k_+}(\mathbf{A}; H_s; J_Y) \rightarrow (\Sigma\times\mathbb{N})^{\ell_{H}} \times (\Sigma\times\mathbb{N})^{\ell_\mathrm{c}} \times (\Sigma\times\mathbb{N})^{\ell_L} \\[0.5ex]
       & \widetilde{\aug}^c_Y (\mathbf{v},\mathbf{z}) := 
        \big( (\pi_{Y}\circ{v}^{H}_{1}(z^{H}_{1,1}),\m(\gamma_{z^{H}_{1,1}})), \dots, (\pi_Y\circ v(z_1),\m(\gamma_{z_{1}})),\dots \\ & \hspace{6cm} \dots, (\pi_{Y}\circ{v}^L_{N_L}(z^L_{N_L,\ell_{L,N_L}}),\m(\gamma_{z^L_{N_L,\ell_{L,N_L}}})) \big) ,	
\end{split}
\]
where $\gamma_{z_i}$ are Reeb orbits that $v$ converges to at $z_{i}$ and similarly for $\gamma_{z^H_{i,j}}$ and $\gamma_{z^L_{k,l}}$. For $\mathbf{B} := ( B_1, \dots, B_\ell ) \in ( \pi_2(X) )^\ell$ satisfying $B_j \cdot \Sigma >0$ for every $j = 1, \dots, \ell$, we define
\begin{align*}
    &\widetilde{\mathrm{EV}}^c :  \big ( \mathcal{M}_{N, \ell; k_-,k_+}(\mathbf{A}; H_s; J_Y) \times \mathbb{R}^{N - 1}_{>0} \big ) \times \mathcal{P}_\ell(\mathbf{B};J_X) \rightarrow Y^{2N} \times (\Sigma \times \mathbb N)^{\ell} \times (\Sigma \times \mathbb N)^{\ell} \\[.5ex]
    &\widetilde{\mathrm{EV}}^c \big(((\mathbf{v},\mathbf{z}), \mathbf{t}),\mathbf{u}\big) : = \big ( \widetilde{\ev}^c_{Z_Y} ((\mathbf{v},\mathbf{z}), \mathbf{t}), \, \widetilde{\aug}^c_Y (\mathbf{v},\mathbf{z}), \, \widetilde{\aug}_X(\mathbf{u}) \big ).
\end{align*}

\begin{definition}\label{def:conti_moduli2}
    For $\ell \geq 1$, we define 
    \[
       \mathcal{M}_{N, \ell}(\tilde{q}^{L}_{k_-},\tilde{p}^H_{k_+};\mathbf{A},\mathbf{B}; H_s; J_X) 
       := (\widetilde{\mathrm{EV}}^c)^{-1} \big( \, W^s_{Z_Y}(\tilde{p}) \times \Delta_{Y}^{N - 1} \times W^u_{Z_Y}(\tilde{q}) \times \Delta_{(\Sigma\times\N)^\ell} \, \big ),
   \]
    see Figure \ref{fig:cont}.\footnote{ If $N=1$, we ignore $\Delta_Y^{N-1}$ as usual.}
      The subspace of simple elements is defined by
    \begin{align*}
        & \mathcal{M}^*_{N, \ell}(\tilde{q}^{L}_{k_-},\tilde{p}^H_{k_+};\mathbf{A},\mathbf{B}; H_s; J_X) \\
        &\qquad :=\mathcal{M}_{N, \ell}(\tilde{q}^{L}_{k_-},\tilde{p}^H_{k_+};\mathbf{A},\mathbf{B}; H_s; J_X)   \cap   \big ( (\mathcal{M}^*_{N, \ell; k_-,k_+}(\mathbf{A}; H_s; J_Y) \times \mathbb{R}^{N - 1}_{>0} ) \times \mathcal{P}^*_\ell(\mathbf{B};J_X) \big ).
\end{align*}
\end{definition}

Analogous to \eqref{projection0}, we have the projection maps 
\begin{equation}\label{projection}
\begin{split}
	&\Pi:\mathcal{M}_{N, \ell}(\tilde{q}^{L}_{k_-},\tilde{p}^H_{k_+}; \mathbf{A}; H_s; J_Y)\longrightarrow \mathcal{N}_{N, \ell}(q,p;\mathbf{A};J_\Sigma),  \\[1 ex]
    &\Pi:\mathcal{M}_{N, \ell}(\tilde{q}^{L}_{k_-},\tilde{p}^H_{k_+};\mathbf{A},\mathbf{B}; H_s; J_X) \longrightarrow \mathcal{N}_{N, \ell}(q,p;\mathbf{A},\mathbf{B};J_X), 
\end{split}
\end{equation}
defined by $\Pi( (\mathbf{v},\mathbf{z}), \mathbf{t}  ) := ( (\mathbf{w},\mathbf{z}), \mathbf{t}  )$ and $\Pi (((\mathbf{v},\mathbf{z}), \mathbf{t}), \mathbf{u}) := (((\mathbf{w},\mathbf{z}), \mathbf{t}), \mathbf{u})$ where 
\[
\mathbf{w} := (\pi_{ Y}\circ {v}^{H}_1,\dots,\pi_{Y}\circ {v}^{H}_{N_{H}},\pi_{Y}\circ {v},\pi_{ Y}\circ {v}^L_1,\dots,\pi_{Y}\circ {v}^L_{N_L}).
\]
These maps remain valid with $\mathcal{M}_{N, \ell}$ and $\mathcal{N}_{N, \ell}$ replaced by $\mathcal{M}_{N, \ell}^*$ and $\mathcal{N}_{N, \ell}^*$.

\begin{definition}
	Taking unions over all $\mathbf{A}\in (\pi_2(\Sigma))^N$ and $\mathbf{B}:=(B_1,\dots,B_\ell) \in (\pi_2(X))^\ell$ with $B_j\cdot\Sigma\neq 0$ for all $j$, we define 
\begin{align*}
    &\mathcal{M}_{N}(\tilde{q}^{L}_{k_-},\tilde{p}^H_{k_+}; H_s; J_X) \\
    &\hspace{1cm}:= \bigcup_{\mathbf{A}}\Big(\mathcal{M}_{N, \ell = 0}(\tilde{q}^{L}_{k_-},\tilde{p}^H_{k_+};\mathbf{A}; H_s; J_Y) \cup  \bigcup_{\ell\in\mathbb{N}} \bigcup_{\mathbf{B}}\,\mathcal{M}_{N, \ell}(\tilde{q}^{L}_{k_-},\tilde{p}^H_{k_+};\mathbf{A},\mathbf{B}; H_s; J_X)\Big),
\end{align*}
and the subspace of simple elements by
\begin{align*}
    &\mathcal{M}_{N}^*(\tilde{q}^{L}_{k_-},\tilde{p}^H_{k_+}; H_s; J_X) \\
    &\hspace{1cm}:= \bigcup_{\mathbf{A}}\Big(\mathcal{M}_{N, \ell = 0}^*(\tilde{q}^{L}_{k_-},\tilde{p}^H_{k_+};\mathbf{A}; H_s; J_Y) \cup  \bigcup_{\ell\in\mathbb{N}} \bigcup_{\mathbf{B}}\,\mathcal{M}_{N, \ell}^*(\tilde{q}^{L}_{k_-},\tilde{p}^H_{k_+};\mathbf{A},\mathbf{B}; H_s; J_X)\Big).
\end{align*}
\end{definition}
We orient the space $\mathcal{M}_{N}^*(\tilde{q}^{L}_{k_-},\tilde{p}^H_{k_+}; H_s; J_X)$ in the same way as moduli spaces of split Floer cylinders, see above Proposition \ref{prop:cpt}. 
There is a natural free $\mathbb{R}^{N-1}$-action on $\mathcal{M}_{N}(\tilde{q}^L_{k_-},\tilde{p}^H_{k_+}; H_s; J_X)$  given by translating the domain $\R\times S^1$ of $\tilde{v}^{H}_i$ and $\tilde{v}^L_j$ (together with punctures $\mathbf{z}$) in the $\mathbb{R}$-direction.
We write
\begin{align*}
    \overbar{\mathcal{M}}(\tilde{q}^{L}_{k_-},\tilde{p}^H_{k_+}; H_s; J_X) &: = \bigcup_{N \in\mathbb N} \mathcal{M}_{N}(\tilde{q}^{L}_{k_-},\tilde{p}^H_{k_+}; H_s; J_X)/\mathbb{R}^{N-1}, \\[1.5ex]
    \overbar{\mathcal{M}}^*(\tilde{q}^{L}_{k_-},\tilde{p}^H_{k_+}; H_s; J_X) &: = \bigcup_{N \in\mathbb N} \mathcal{M}^*_{N}(\tilde{q}^{L}_{k_-},\tilde{p}^H_{k_+}; H_s; J_X)/\mathbb{R}^{N-1}.
\end{align*}

We are now in a position to state the main result of Section \ref{sec:split_Rabinowitz}, which implies that the continuation homomorphism induced by $H_s$ between the split Floer chain complexes $\FC_*(H)$ and $\FC_*(L)$ is essentially the identity map, see \eqref{continuation_identity} below.  

\begin{prop}\label{continuation_transversality}
Let $H_s$ be a monotone homotopy for $H \prec L$ as above.
Then the following assertions hold for every $J_X \in \mathcal{J}^{\reg}_X$, $\tilde{p}^H_{k_+} \in \Crit f_{Y^H_{k_+}}$, and $\tilde{q}^L_{k_-} \in \Crit f_{Y^L_{k_-}}$, where $k_\pm\in m_W \Z$. 
\begin{enumerate}[(a)]
	\item The moduli space $\mathcal{M}^{*}_{N}(\tilde{q}^{L}_{k_-},\tilde{p}^H_{k_+}; H_s; J_X)$ is cut out transversely and has dimension 
\begin{equation*}
    \dim \mathcal{M}^{*}_{N}(\tilde{q}^{L}_{k_-},\tilde{p}^H_{k_+}; H_s; J_X) = \mu( \tilde{p}^H_{k_+} ) - \mu ( \tilde{q}^{L}_{k_-} ) + N-1.
\end{equation*}

	\item Suppose that $\mu ( \tilde{q}^{L}_{k_-} )=\mu( \tilde{p}^H_{k_+} )$.
	Then $\mathcal{M}_{N}(\tilde{q}^{L}_{k_-},\tilde{p}^H_{k_+}; H_s; J_X)$ is not empty if and only if
		\[
		N=1,\;\; q=p,\;\; k_-=k_+,\;\; (\tilde{q}^{L}_{k_-},\tilde{p}^H_{k_+})\in\big\{(\hat{p}^L_{k_+},\hat{p}^{H}_{k_+}),(\check{p}^L_{k_+},\check{p}^{H}_{k_+})\big\}.
		\]
	Moreover, $\overbar{\mathcal{M}}(\hat{p}^{L}_{k_+},\hat{p}^H_{k_+}; H_s;J_X)$ consists of a single element $[\mathbf{v}]=[\tilde v]$ with $\ell=0$, and this projects to a constant curve, i.e.~$\pi_{\mathbb{R}\times Y}\circ \tilde{v}=p$. In particular,
	\[
	\overbar{\mathcal{M}}(\hat{p}^L_{k_+},\hat{p}^H_{k_+}; H_s;J_X)=\overbar{\mathcal{M}}^*(\hat{p}^L_{k_+},\hat{p}^H_{k_+}; H_s;J_X).
	\]
	Furthermore, every element of $\overbar{\mathcal{M}}(\hat{p}^L_{k_+},\hat{p}^H_{k_+}; H_s;J_X)$ is positively oriented. 
The corresponding statement holds for the case $(\tilde{q}^{L}_{k_-},\tilde{p}^H_{k_+})=(\check{p}^L_{k_+},\check{p}^{H}_{k_+})$, see Figure \ref{fig:C_Moduli}.
	
\end{enumerate} 
\end{prop}

\begin{figure}[h]
     \centering
     \includegraphics[height = 4cm]{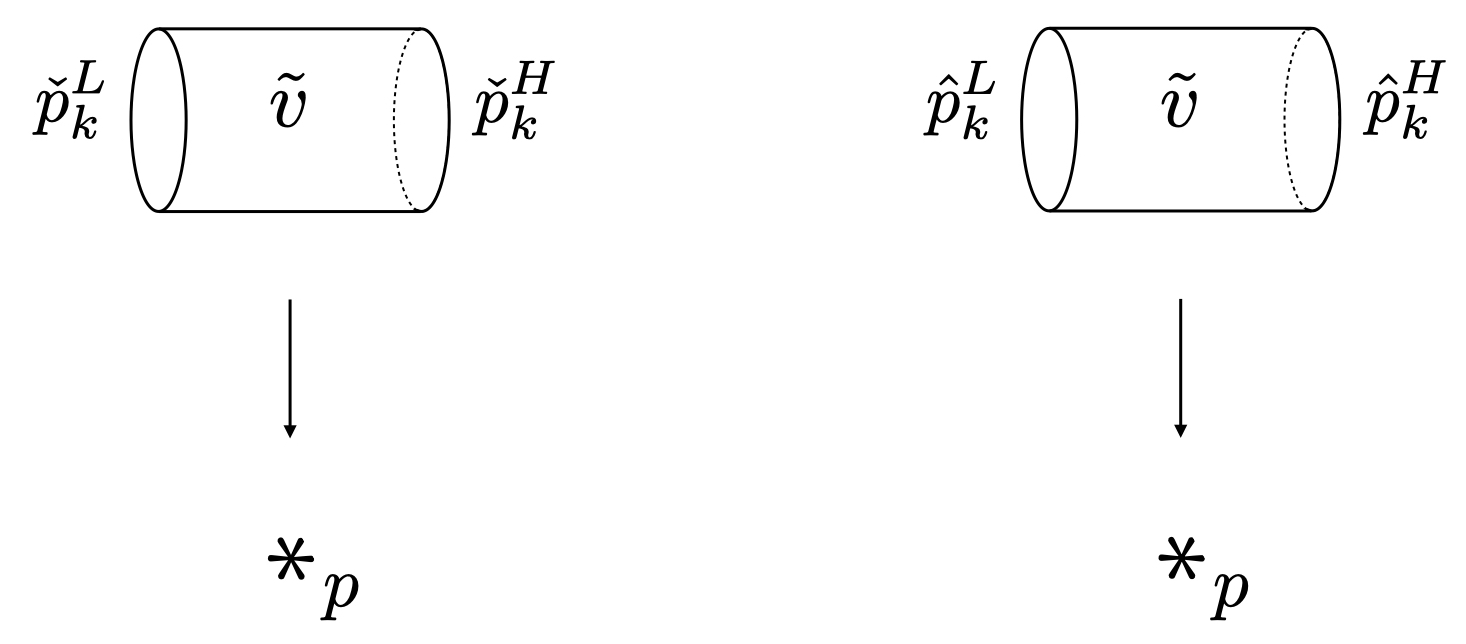}
     \caption{Split Floer continuation cylinders and its shadows}
     \label{fig:C_Moduli}
\end{figure}

\begin{remark}\label{rem:conti_two_levels}
	In the moduli spaces considered in Proposition \ref{prop:cpt} or Proposition \ref{continuation_transversality}, we do not take Floer-holomorphic buildings with more than two levels into account since ordinary (continuation) Floer cylinders of index one (zero) do not converge to such limit buildings through the neck-stretching procedure under our standing hypothesis. This can be seen by similar computations made in the proof of Proposition \ref{continuation_transversality}.(b).
\end{remark}

\subsection{Proof of Proposition \ref{continuation_transversality}}\label{sec:continuation_unique}

Recalling that $Y_p$ is the fiber of $Y\to\Sigma$ over $p\in\Sigma$,
we denote by $J_p$ the complex structure on $\mathbb{R} \times Y_p\cong \R\times S^1$ for $p \in \Sigma$ defined by $J_p\partial_r=R$. This coincides with the restriction of any $J_Y\in\mathcal{J}_Y$ to $\mathbb{R} \times Y_p$. 
Before proving Proposition \ref{continuation_transversality}, we need the following proposition. We refer to \cite[Theorem 2.2]{AGH} for a related result.

\begin{prop}\label{useful_prop}
Let $H_s$ be a monotone homotopy for $H \prec L$. Let $x^H$ and $x^L$ be $1$-periodic orbits of type III of $X_H$ and $X_L$ respectively such that $\m (x^H) = \m (x^L)$ and $\pi_{\mathbb{R}\times Y} \circ x^H =\pi_{\mathbb{R}\times Y} \circ x^L  = p \in \Sigma$.
Then the space
\[
    \mathcal{M} : = \left \{ \, \tilde{v}\in C^\infty(\mathbb{R}\times S^1 , \mathbb{R}\times Y_p) \, \middle | \, \begin{array}{l}
         \partial_s \tilde{v} + J_p(\tilde{v})(\partial_t \tilde{v} - X_{H_s}(\tilde{v})) = 0,  \\[0.5ex]
         \lim\limits_{s \to - \infty} \tilde{v} (s, \cdot) = x^L(\cdot+\tau^L) \text{ for some } \tau^L \in S^1, \\[0.5ex]
         \lim\limits_{s \to +\infty} \tilde{v} (s, \cdot) = x^H(\cdot+\tau^H) \text{ for some } \tau^H \in S^1
    \end{array} \, \right \}
\]
is a nonempty connected smooth $1$-dimensional manifold. More precisely, there exists $b \in C^\infty(\mathbb R)$ such that, for $b_0(s,t)=b(s)$ and $v(s,t)=\varphi_R^{\m(x^H)t/K}(y)$ with any $y\in Y_p$, we have $\tilde{v}=(b_0,v) \in \mathcal{M}$. Conversely, any $\tilde{v}=(b_0,v) \in \mathcal{M}$ is of such form. 
\end{prop}

\begin{proof}
The first part of the proof is essentially the same as the proof of \cite[Lemma 5.24]{DL2}. To explain this, we briefly describe a necessary functional setting. We choose a smooth function 
\[
\beta:\mathbb{R}\to[0,1],\qquad \beta'\geq0,\;\;\beta(0)=0,\;\;\beta(1)=1.
\]
We denote by $\mathbf{e}_1=(1,0)$ and $\mathbf{e}_2=(0,1)$ the standard basis of $\mathbb{R}^2$ and by $J_0$ the standard complex structure given by  $J_0\mathbf{e}_1=\mathbf{e}_2$. 
For $p>2$ and $\delta>0$, we consider the Banach space $W^{k,p,\delta}_{\mathbf{e}_2}(\mathbb{R}\times S^1, \mathbb{R}^2)$   
which consists of maps $\sigma : \mathbb{R} \times S^1 \to \mathbb{R}^2$ such that 
\[
(s,t)\mapsto e^{  \delta\beta(s) s -  \delta (1-\beta(s))  s  } \big(\sigma(s,t) - c\beta(s) \mathbf{e}_2 - d(1-\beta(s))\mathbf{e}_2 \big)  
\] 
belongs $W^{k,p}(\mathbb{R}\times S^1, \mathbb{R}^2)$ for some $
c,d\in\mathbb R$. Similarly, a map $\sigma : \mathbb{R} \times S^1 \to \mathbb{R}^2$ is in 
 $L^{p,\delta}(\mathbb{R}\times S^1, \mathbb{R}^2)$ if and only if $(s,t)\mapsto e^{  \delta\beta(s) s -  \delta (1-\beta(s))  s  } \sigma(s,t)$ is of class $L^p$. 
For $\tilde{v} = (b,v) \in \mathcal{M}$, we choose a unitary trivialization 
\[
{\tilde{v}}^*T(\mathbb R \times Y_p) \cong (\mathbb{R}\times S^1)\times \mathbb{R}^2
\] 
sending ${\partial_r}$ and  $R$ to $\mathbf{e}_1$ and $\mathbf{e}_2$, respectively. Linearizing the Floer equation at $\tilde{v}$, we obtain an operator $D_{\tilde v}$ such that $\ker D_{\tilde v}=T_{\tilde v}\mathcal{M}$. We do not give a precise formula of $D_{\tilde v}$. Instead, in the above trivialization, $D_{\tilde v}$ is of the form
\begin{align}\label{linearized_operator}
    D: W^{1,p,\delta}_{\mathbf{e}_2}(\mathbb{R}\times S^1, \mathbb{R}^2) \rightarrow L^{p,\delta}(\mathbb{R}\times S^1, \mathbb{R}^2),  \nonumber \\
    D\sigma = \partial_s \sigma + J_0 \partial_t \sigma + \bigg (\begin{matrix} h_s^{''}(e^{b})e^{b} & 0 \\ 0 & 0 \end{matrix} \bigg ) \sigma .
\end{align}
The same operator with $h_s$ replaced by $s$-independent function $h$ is studied in \cite[Lemma 5.24]{DL2}. 
Since $h_s$ is independent of $s$ outside $[-1,1]$, the asymptotic operators of $D$ have the form considered in \cite[Lemma 5.24]{DL2}, and the proof therein holds verbatim. 
For sufficiently small $\delta$, the operator $D$ is Fredholm of index 1 and surjective due to an automatic transversality result in \cite[Prop.2.2]{W}.  Since $\ker D$ is 1-dimensional and contains $\mathbf{e}_2$ by \eqref{linearized_operator}, it is generated by $\mathbf{e}_2$. This reflects the fact that $\mathbf{e}_2$ corresponds to $R$ in the above trivialization and $R(\tilde v)\in\ker D_{\tilde v}=T_{\tilde v}\mathcal{M}$ since $\varphi^{t}_R\circ\tilde v\in\mathcal{M}$ for all $t \in \mathbb{R}$. This proves that each connected component of $\mathcal{M}$ is a smooth 1-dimensional manifold diffeomorphic to $S^1$. 

\medskip

Next we observe that if $\tilde v\in\mathcal{M}$, then $\tilde{v}(s,t+t_0)\in \mathcal{M}$ for any $t_0 \in  S^1$.
Thus $\partial_t \tilde{v}\in \ker D_{\tilde{v}}$. As we have shown above, $\ker D_{\tilde v}$ is generated by $R(\tilde v)$, and hence 
\begin{equation}\label{eq:tilde_v//R}
\partial_t \tilde{v} = (\partial_t b, \partial_t v) = aR(\tilde v)	
\end{equation}
for some  $a \in \mathbb{R}$.
This implies that $b$ is independent of $t$. Moreover 
since we have
\[
 \lim_{s \rightarrow  \infty}  \partial_t \tilde{v} =X_H(x^H(\cdot+\tau^H))  = \frac{\m(x^H)}{K} R(x^H(\cdot+\tau^H))
\]
for some $\tau^H\in S^1$, we obtain $a=\m(x^H)/K$.
Recall from \eqref{eq:Floer_eq2} that $\tilde v=(b, v)$ solves
\begin{equation}\label{eq:Floer_fiber}
\begin{cases}
     \partial_s b-\alpha|_{Y_p}(\partial_t v) + h'_s(e^b) = 0, \\[0.5ex]
     \partial_t b+\alpha|_{Y_p}(\partial_sv)=0.
\end{cases}	
\end{equation}
The second line together with $\partial_t b =0$ implies $\partial_s v =0$. Therefore $v$ is independent of $s$, and by \eqref{eq:tilde_v//R}
\[
v(s,t)=v(t) = \varphi_R^{\m(x^H)t/K}(v(0)).
\]
We have shown the converse part of the proposition. 

\medskip

Now the first line of \eqref{eq:Floer_fiber} simplifies to an ODE
\begin{equation}\label{eq:ode1}
    \frac{d}{ds} b - \frac{\m(x^H)}{K} + h'_s(e^b)=0,
\end{equation}
and the asymptotic conditions on $\tilde v$ implies 
\begin{equation}\label{eq:ode2}
    \lim_{s \to - \infty} b(s) = \beta^L_{k}, \qquad \lim_{s \to \infty} b(s) = \beta^H_{k}
\end{equation}
where $\beta^L_{k} \in (-\eta,\eta)$ is the $\mathbb{R}$-coordinate of the type III orbit $x^L$ in $\mathbb{R}\times Y_p$, and likewise for $\beta^H_{k}\in(-\eta,\eta)$. 
Once we prove that there exists a unique solution $b\in C^\infty(\R)$ of \eqref{eq:ode1} and \eqref{eq:ode2}, then the above computations imply that the unique solution $b$ together with $v(s,t)=\varphi_R^{\m(x^H)t/K}(y)$ for any $y\in Y_p$ belongs to $\mathcal M$. 

To prove the existence and the uniqueness of such solution $b$, we recall that $H_s$ satisfies the conditions (i) and (ii) of $\mathcal{H}$ in Section \ref{sec: Hamiltonians}, namely
\[
H_s(r,y)=h_s(e^r),\qquad h_s(r)= \begin{cases}
         -a_s r + b^-_s, & r \in (e^{-\epsilon/8 + \eta} , e^{-\eta}) ,\\[0.5ex]
         a_s r + b^+_s, & r \in (e^\eta,+\infty),
    \end{cases}
    \qquad  h''_s|_{(e^{-\eta}, e^\eta)}>0
\]
for some $a_s, b^-_s, b^+_s \in \mathbb R$. 
Moreover, it holds  
\begin{equation}\label{eq:homotopy_m}
-a_s < \frac{\m(x^L)}{K}=\frac{\m(x^H)}{K} <a_s \qquad \forall s\in\R.
\end{equation}
We first claim 
\begin{equation}\label{eq:claim1}
-\eta<b(s)<\eta \qquad \forall s\in\R.	
\end{equation}
Assume on the contrary that there is $s_0 \in \mathbb R$ with $b(s_0) \geq \eta$. Let $s_1 := \inf \left \{ s \in \mathbb R  \mid  b(s) = \eta  \right \}$. 
We note that $s_1$ is finite due to the first condition in \eqref{eq:ode2} and $\frac{d}{ds}b|_{s=s_1}\geq0$ holds.
On the other hand, \eqref{eq:homotopy_m} yields
\[
\frac{d}{ds}b\,\bigg|_{s=s_1}=\frac{\m(x^L)}{K} - h_{s_1}'(e^{b(s_1)}) = \frac{\m(x^L)}{K}-a_s < 0.
\]
This contradiction shows $b(s)<\eta$ for every $s\in\R$. Analogously, $b(s)>-\eta$ for every $s\in\R$. 

Our next claim is 
\begin{equation}\label{eq:claim2}
b(s)=\beta^L_{k} \qquad \forall s\leq-1.
\end{equation}
By the first condition in \eqref{eq:ode2}, $\mathbf{b}:=\inf\{ b(s)\mid s\leq -1\}$ is finite and $\mathbf{b}\leq \beta_k^L$. We argue by contradiction and assume  $\mathbf{b} < \beta^L_{k}$. Then there is $s_2\in(-\infty,-1]$ such that $b(s_2)=\mathbf{b}$. 
Since $b$ takes its values in $(-\eta,\eta)$ by \eqref{eq:claim1} on which $h_s''>0$, 
\[
    \frac{d}{ds}b\,\bigg|_{s=s_2}=\frac{\m(x^L)}{K} - h_{s_2}'(e^{b(s_2)}) > \frac{\m(x^L)}{K} - h_{s_2}'(e^{\beta^L_{k}})=0.
\]
The last equality is due to $H_s=L$ for $s\leq-1$. However, this inequality implies $b(s_2-\epsilon)<b(s_2)=\mathbf{b}$ for small $\epsilon>0$. This contradicts the definition of $\mathbf{b}$. Thus $\mathbf{b}=\beta_k^L$ for $s\leq -1$. Analogously, we can show that $\sup\{ b(s)\mid s\leq -1\}=\beta_k^L$ for $s\leq -1$, and \eqref{eq:claim2} follows.

Let $b\in C^\infty(\R)$ be a solution of \eqref{eq:ode1} satisfying \eqref{eq:claim2}, which uniquely exists due to the existence and uniqueness theorem of ODE. It remains to verify that $b$ meets the second condition in  \eqref{eq:ode2}. Let $s_3\geq1$ be arbitrary. In the case of $b(s_3)<\beta^H_{k}$, we have
\[
\begin{split}
\frac{d}{ds}b\,\bigg|_{s=s_3}&=\frac{\m(x^L)}{K} - h_{s_3}'(e^{b(s_3)}) = \frac{\m(x^H)}{K} - h_{s_3}'(e^{\beta^H_{k}}) +\int^{e^{\beta^H_{k}}}_{e^{b(s_3)}}h''_{s_3}(\rho)d\rho \\[.5ex]
&=\int^{e^{\beta^H_{k}}}_{e^{b(s_3)}}h''_{s_3}(\rho)d\rho\geq (e^{\beta^H_{k}}-e^{b(s_3)})\min_{[e^{b(s_3)},e^{\beta^H_{k}}]} h_{s_3}''	>0
\end{split}
\]
The last equality is due to $H_s=H$ for $s\geq 1$.
A similar estimate shows that 
\[
\begin{split}
\frac{d}{ds}b\,\bigg|_{s=s_3}&\leq-(e^{b(s_3)}-e^{\beta^H_{k}})\min_{[e^{\beta^H_{k}},e^{b(s_3)}]} h_{s_3}''<0 \;\;\textrm{ if } \;\; b(s_3)>\beta^H_{k},\\[.5ex]
 \frac{d}{ds}b\,\bigg|_{s=s_3}&=0 \;\;\textrm{ if }\;\; b(s_3)=\beta^H_{k}.
\end{split}
\]
These imply that $b$ converges to $\beta^H_{k}$ as $s$ goes to $\infty$.
This completes the proof.
\end{proof}

\subsubsection*{Proof of Proposition \ref{continuation_transversality}.(a)}
The proof goes along the same lines as the proof of \cite[Proposition 5.9]{DL2}. For completeness, we give an outline of the proof and add some details when additional care is needed.

For $J_X \in \mathcal{J}^{\reg}_X$, as usual $J_Y$ and $J_\Sigma$ denote the cylindrical and horizontal parts, respectively. 
Let $\tilde{v}=(b, v)$ be a  Floer cylinder or a  Floer continuation cylinder, which appears in $\mathbf{v}$ of an  element of $\mathcal{M}^*_{N, \ell;k_-,k_+}(\mathbf{A};H_s;J_Y)$.
The operator obtained by linearizing 
\eqref{Floer} or \eqref{Floer_conti} defined on suitable Sobolev spaces is of the form \eqref{decomposition_operator}.
Following the discussion after \eqref{decomposition_operator}, we also obtain the index computation \eqref{Fredholm_index}.
Since the $J_\Sigma$-holomorphic sphere $\pi_Y\circ v$ is either constant or somewhere injective, the associated Cauchy-Riemann operator $\dot{D}^\Sigma_{\pi_Y  \circ v}$ is surjective. Note that $J_\Sigma\in\mathcal{J}^{\reg}_\Sigma$ by  Proposition \ref{transversality_chain}.
The operator $D^{\mathbb{C}}_{\tilde{v}}$ is the same as the operator $D$ in \eqref{linearized_operator} studied in the proof of Proposition \ref{useful_prop}, which is also surjective.
Thus, $D_{\tilde{v}}$ is surjective by \eqref{decomposition_operator}.
We note that in the case $\tilde v$ is a  Floer continuation cylinder, the exact form of the operator $D^\mathbb{C}_{\tilde{v}}=D$ in \eqref{linearized_operator} is slightly different from that in \cite[Proposition 5.9]{DL2} due to the dependence on $s$ in the homotopy $H_s$. This does not cause any problem since the proof of \cite[Proposition 5.9]{DL2} relies on the asymptotic formula of the operator and $H_s$ depends on $s$ only in the finite interval $[-1,1]$. 
This proves that $\mathcal{M}^*_{N, \ell;k_-,k_+}(\mathbf{A};H_s;J_Y)$  is  a smooth manifold of dimension
\[
\begin{split}
\dim \mathcal{M}^*_{N, \ell;k_-,k_+}(\mathbf{A};H_s;J_Y) &= \sum_{i=1}^{N_H} \ind D_{\tilde{v}^H_i} + \ind D_{\tilde{v}} + \sum_{j=1}^{N_L}\ind D_{\tilde{v}^L_j} +2\ell\\
&= N \dim Y + \sum_{i=1}^{N_H} 2c_1^{T\Sigma}(A^H_i) + 2c_1^{T\Sigma}(A) + \sum_{j=1}^{N_L} 2c_1^{T\Sigma}(A^L_j) + 2\ell.
\end{split}
\]

To prove that $\mathcal{M}^*_{N, \ell=0}(\tilde{q}^L_{k_-},\tilde{p}^H_{k_+};\mathbf{A};H_s;J_Y)$ given in Definition \ref{def:conti_moduli1} is a smooth manifold, we want to show that the following transversality holds:
\begin{equation}\label{eq:trasv_ev_c}
\widetilde{\ev}^c_{Z_Y} \pitchfork \big(W^s_{Z_Y}(\tilde{p}) \times \Delta_{Y}^{N - 1} \times W^u_{Z_Y} (\tilde{q}) \big).
\end{equation}
We consider the following commutative diagram
    \begin{equation*}
        \begin{tikzcd}
             \mathcal{M}^*_{N, \ell=0;k_-,k_+}(\mathbf{A};H_s;J_Y) \times {\mathbb R}^{N - 1}_{> 0} \arrow{r}{\widetilde{\ev}^c_{Z_Y}} \arrow[swap]{d}{\Pi\times \mathrm{id}} &Y^{2N} \arrow{d}{\pi_Y\times \cdots \times \pi_Y} \\
\mathcal{N}^*_{N,\ell=0}(\mathbf{A};J_\Sigma) \times {\mathbb R}^{N-1}_{> 0} \arrow{r}{\ev_{Z_\Sigma}}& \Sigma^{2N}
        \end{tikzcd}
    \end{equation*}
where $\Pi$ is defined in \eqref{projection} and $\mathrm{id}$ is the identity map on $\mathbb R^{N - 1}_{> 0}$. In the proof of Proposition \ref{transversality_chain}, the horizontal part of \eqref{eq:trasv_ev_c}, i.e.~$\ev_{Z_\Sigma}\pitchfork (W^s_{Z_\Sigma}(p)\times \Delta^{N-1}_\Sigma\times W^u_{Z_\Sigma}(q))$, is proved. 
Hence, it suffices to establish the vertical part of \eqref{eq:trasv_ev_c}, namely  
\[
   (R,0,\dots,0), \dots, (0,\dots,0,R)  \in \mathrm{im\,}(d_{((\mathbf{v},\emptyset),\mathbf{t})}\widetilde{\ev}^c_{Z_Y}) + T_{\widetilde{\ev}^c_{Z_Y}{((\mathbf{v},\emptyset),\mathbf{t})}}\big( W^s_{Z_Y}(\tilde{p}) \times \Delta_{Y}^{N - 1} \times W^u_{Z_Y}(\tilde{q})\big)
\]
for any $((\mathbf{v},\emptyset), \mathbf{t}) \in   \mathcal{M}^*_{N, \ell=0}(\tilde{q}^L_{k_-},\tilde{p}^H_{k_+};\mathbf{A};H_s;J_Y)$. As shown in the proof of \cite[Proposition 5.41]{DL2}, this holds provided that $\mathbf{v}$ contains a nontrivial Floer cylinder $\tilde{v}_0$. 
A key input here is $\m(\ev_-(\tilde{v}_0))\neq \m(\ev_+(\tilde{v}_0))$. The same proof carries over to the case when $\mathbf{v}$ contains a Floer continuation cylinder $\tilde v$ with $\m(\ev_-(\tilde v))\neq\m(\ev_+(\tilde v))$.

The remaining case is that $\mathbf{v}$ consists of a single Floer continuation cylinder, i.e.~$\mathbf{v}=\tilde v$, with 
 $\m(\ev_-(\tilde v))=\m(\ev_+(\tilde v))$.
In this case $[ \pi_{\mathbb{R}\times Y}\circ\tilde v ] = 0$ by \eqref{multiplicity}, and  \eqref{eq:trasv_ev_c} reads
\[
\big(\widetilde{\ev}^c_{Z_Y}: \mathcal{M}^*_{N = 1,\ell = 0; k, k}(A=0;H_s;J_Y) \rightarrow Y^2\big) \pitchfork \big(W^s_{Z_Y}(\tilde{p}) \times W^u_{Z_Y}(\tilde{q})\big),
\] 
where $k=\m(\ev_-(\tilde v))=\m(\ev_+(\tilde v))$. 
By Proposition \ref{useful_prop}, the image of $\widetilde{\ev}^c_{Z_Y}$ is the diagonal $\Delta_Y$, and it is indeed transverse to $W^s_{Z_Y}(\tilde{p}) \times W^u_{Z_Y}(\tilde{q})$ since $(f_Y, Z_Y)$ is a Morse-Smale pair.
This completes the proof of \eqref{eq:trasv_ev_c}. 

\medskip

Next, we deal with the case $\ell\geq1$. According to Definition \ref{def:conti_moduli2}, we want to prove the following transversality result: 
\[
\widetilde{\mathrm{EV}}^c\pitchfork \big(W^s_{Z_Y}(\tilde{p}) \times \Delta_{Y}^{N - 1} \times W^u_{Z_Y}(\tilde{q}) \times {\Delta}_{(\Sigma\times \mathbb{N})^{\ell}}\big).
\]
Due to \eqref{eq:trasv_ev_c}, which also holds for $\ell \neq 0$, this is equivalent to showing 
\[
\big(\widetilde{\aug}_Y^c \times \widetilde{\aug}_X:  \mathcal{M}^*_{N,\ell;k_-,k_+}(\mathbf{A};H_s;J_Y) \times \mathcal{P}^*_\ell(\mathbf{B};J_X) \rightarrow (\Sigma\times \mathbb{N})^{\ell} \times (\Sigma\times \mathbb{N})^{\ell}\big) \pitchfork {\Delta}_{(\Sigma\times \mathbb{N})^{\ell}}
\]
which follows from the fact that $\widetilde{\aug}_X$ is a submersion, see \cite[Proposition 5.29]{DL2}.
This proves that the moduli space $\mathcal{M}^*_{N,\ell}(\tilde{q}^L_{k_-}, \tilde{p}^H_{k_+};\mathbf{A}, \mathbf{B};H_s;J_X)$ is a smooth manifold, and its dimension is computed as follows:
\[
	\begin{split}
    &\hspace{-1cm}\dim  \mathcal{M}^{*}_{N, \ell}(\tilde{q}^L_{k_-},\tilde{p}^H_{k_+};\mathbf{A}, \mathbf{B}; H_s;J_X) \\[.5ex]
    &= \dim  \mathcal{M}^*_{N,\ell;k_-,k_+}(\mathbf{A}; H_s; J_Y) + N - 1 + \dim\mathcal{P}^*_{\ell}(\mathbf{B};J_X)\\[.5ex]
    &\quad - \codim \big ( W^s_{Z_Y}(\tilde{p}) \times \Delta_{Y}^{N-1} \times W^u_{Z_Y}(\tilde{q}) \times {\Delta}_{(\Sigma\times \N)^\ell} \big ) \\
    & = \ N \dim Y + \sum_{i=1}^{N_H} 2c_1^{T\Sigma}(A^H_i) + 2c_1^{T\Sigma}(A) + \sum_{j=1}^{N_L} 2c_1^{T\Sigma}(A^L_j)  
    + 2\ell + N - 1 \\
    &\quad +  \sum_{j=1}^\ell \big(\dim X +  2c_1^{TX}(B_j)-2 (B_j \cdot \Sigma)  - 4 \big ) \\ 
    &\quad - \big (\dim Y  -\ind_{f_Y}(\tilde{p}) + (N-1) \dim Y +  \ind_{f_Y}(\tilde{q}) + \ell\dim\Sigma \big ) \\[1.5 ex]
    &= \sum_{i=1}^{N_H} 2c_1^{T\Sigma}(A^H_i) + 2c_1^{T\Sigma}(A) + \sum_{j=1}^{N_L} 2c_1^{T\Sigma}(A^L_j) +\sum_{j=1}^\ell \big(2c_1^{TX}(B_j)-2(B_j \cdot \Sigma)\big )\\
    &\quad +  N-1+\ind_{f_Y}(\tilde{p})-\ind_{f_Y}(\tilde{q}) \\[.5ex]
    &= \mu( \tilde{p}^H_{k_+} ) - \mu( \tilde{q}^L_{k_-} ) + N-1 .
\end{split}
\]
The last line follows from \eqref{RS-index}, \eqref{grading}, and \eqref{multiplicity}. A similar computation also yields
\[
\dim \mathcal{M}^*_{N, \ell=0}(\tilde{q}^L_{k_-},\tilde{p}^H_{k_+};\mathbf{A};H_s;J_Y)=N-1+ \mu( \tilde{p}^H_{k_+} ) - \mu( \tilde{q}^L_{k_-}).
\]
This completes the proof of Proposition \ref{continuation_transversality}.(a).
\qed

\subsubsection*{Proof of Proposition \ref{continuation_transversality}.(b)}

We first prove that, under the assumption $\mu( \tilde{p}^H_{k_+} ) =\mu( \tilde{q}^L_{k_-})$,
\begin{equation}\label{proving_simple}
	\mathcal{M}_{N, \ell}(\tilde{q}^L_{k_-},\tilde{p}^H_{k_+}; \mathbf{A}, \mathbf{B}; H_s;J_X)=\mathcal{M}_{N, \ell}^*(\tilde{q}^L_{k_-},\tilde{p}^H_{k_+}; \mathbf{A}, \mathbf{B}; H_s;J_X) 
\end{equation}
for any $N\in\mathbb{N}$ and $\ell\in\mathbb{N}\cup\{0\}$.{\footnote{For $\ell=0$, this means $\mathcal{M}_{N, \ell=0}(\tilde{q}^L_{k_-},\tilde{p}^H_{k_+}; \mathbf{A}; H_s;J_X)=\mathcal{M}_{N, \ell=0}^*(\tilde{q}^L_{k_-},\tilde{p}^H_{k_+}; \mathbf{A}; H_s;J_X)$}
For any $\big ( ((\mathbf{v},\mathbf{z}), \mathbf{t}), \mathbf{u} \big ) \in \mathcal{M}_{N, \ell}(\tilde{q}^L_{k_-},\tilde{p}^H_{k_+}; \mathbf{A}, \mathbf{B};\allowbreak H_s;J_X)$, we consider its shadow $\Pi ((\mathbf{v},\mathbf{z}),\mathbf{t}) \in \mathcal{N}_{N, \ell}(q,p;\mathbf{A};J_\Sigma)$, see \eqref{projection}. We prove \eqref{proving_simple} by showing that $\Pi ((\mathbf{v},\mathbf{z}),\mathbf{t})$ is simple, $\ell\leq 1$, and $\mathbf{u}=[u]$ is somewhere injective when $\ell=1$. 
As mentioned in Remark \ref{rmk:simple_pearl}, we take an underlying simple chain of pearls of $\Pi  ((\mathbf{v},\mathbf{z}), \mathbf{t})$ and denote the space that it belongs to by
\[
\mathcal{N}^*_{N^*, \ell^*}(q,p;\mathbf{A^*};J_\Sigma).
\]
It is obvious that $1\leq N^*\leq N$ and $\ell^*\leq\ell$. To ease the notation, we simply write 
\[
\mathbf{v}=(\tilde v_1,\dots,\tilde v_N),\quad \mathbf{A}=(A_1,\dots,A_N),\quad\mathbf{A}^*=(A_{n_1}^*,\dots,A_{n_{N^*}}^*),
\] 
where $I^*:=\{n_1,\dots,n_{N^*}\}$ is a subset of $I:=\{1,\dots,N\}$, and do not indicate which Hamiltonian $A_i$, $A_{n_i}^*$, and  $\tilde v_i$ are associated with. 
We  denote by $N^*_0$ the number of zero entries in $\mathbf{A}^*$. 
We also denote by $\gamma_j$ for $j = 1 ,\dots, \ell$ the periodic Reeb orbits in $Y$ that Floer (continuation) cylinders in $\mathbf{v}$ converge to at punctures.
Using \eqref{multiplicity} we compute
\begin{equation}\label{eq:index_0}
\begin{split}
	0 &= \mu( \tilde{p}^H_{k_+} ) - \mu( \tilde{q}^L_{k_-} )  \\
    &= \ind_{f_Y}(\tilde{p})-\ind_{f_Y}(\tilde{q})+\sum_{i=1}^{N} 2c_1^{T\Sigma}(A_i) +\sum_{j=1}^\ell \mu_\CZ(\gamma_j) \\
    &= \ind_{f_\Sigma}(p) - \ind_{f_\Sigma}(q) + \sum_{i=1}^{N} 2c_1^{T\Sigma}(A_i)     + 2\ell + i(\tilde{p})-i(\tilde{q}) + \sum_{j=1}^{\ell} | \gamma_j |  \\
    &=\dim \mathcal{N}^*_{N^*, \ell^*}(q,p;\mathbf{A}^{*};J_\Sigma) + \sum_{i\in I\setminus I^*} 2c_1^{T\Sigma}(A_i)+\sum_{i\in I^*}2c_1^{T\Sigma}(A_i-A_i^*)  \\
    &\quad  +1-N^* + 2(\ell-\ell^*) + i(\tilde{p})-i(\tilde{q}) + \sum_{j=1}^{\ell} | \gamma_j |  \\
    &\geq  N^*-2N^*_0+2\ell^*  + \sum_{i\in I\setminus I^*} 2c_1^{T\Sigma}(A_i)+\sum_{i\in I^*}2c_1^{T\Sigma}(A_i-A_i^*)  \\
    &\quad + 2(\ell-\ell^*) + (i(\tilde{p})-i(\tilde{q})+1) + \sum_{j=1}^{\ell} | \gamma_j |.
\end{split}
\end{equation}
For the definitions of $i(\tilde p)$ and $|\gamma_j|$, see \eqref{crit_lift} and \eqref{eq:index_moduli}.
The last two lines follow from
\[
\begin{split}
	    \dim \mathcal{N}^*_{N^*, \ell^*}(q,p;\mathbf{A}^{*};J_\Sigma) &= \ind_{f_\Sigma} (p) - \ind_{f_\Sigma} (q) + N^*-1 + \sum_{i\in I^*} 2c_1^{T\Sigma}(A_i^*) + 2\ell^*\\
        &\geq 2(N^*-N^*_0)+2\ell^*,    
\end{split}
\]
which is due to Proposition \ref{transversality_chain}.(a) and the existence of a free $\mathbb{C}^*$-action on each nonconstant $J_\Sigma$-holomorphic map in $\mathcal{N}^*_{N^*, \ell^*}(q,p;\mathbf{A}^{*};J_\Sigma)$. 
Rearranging \eqref{eq:index_0}, we obtain
\begin{equation}\label{eq:proving_simple}
	\begin{split}
	1&\geq (N^*-N^*_0)+(\ell^*-N^*_0+1) +\sum_{i\in I\setminus I^*} 2c_1^{T\Sigma}(A_i)+\sum_{i\in I^*}2c_1^{T\Sigma}(A_i-A_i^*) \\
	& \quad + \ell^* + 2(\ell-\ell^*) + (i(\tilde{p})-i(\tilde{q})+1) + \sum_{j=1}^{\ell} | \gamma_j |.
\end{split}
\end{equation}
We claim that all the terms on the right-hand side are nonnegative.  Indeed, $|\gamma_j|\geq 0$ by \eqref{eq:index_moduli}, and 
\begin{equation}\label{eq:ell*}
\ell^* \geq N^*_{0} - 1	
\end{equation}
 holds since every nontrivial Floer cylinder $\tilde v_i$ which projects to a point in $\Sigma$ must have at least one puncture since otherwise $\m(\ev_+(v_i))=\m(\ev_-(v_i))$ and $\partial_s\tilde v_i=0$ by \eqref{multiplicity} and the discussion afterwards. Here the contribution $-1$ is due to the fact that, in contrast to a nontrivial Floer cylinder, a nontrivial Floer continuation cylinder may project to a point in $\Sigma$ and have no punctures. 
  Moreover, since $A_i$ is represented by a $J_\Sigma$-holomorphic map and $c_1^{T\Sigma}=(\tau_X-K)\omega_\Sigma$ on $\pi_2(\Sigma)$, we have $2c_1^{T\Sigma}(A_i)\geq 0$ with equality if and only if $A_i=0$. The same holds for $A_i-A_i^*$ when $i \in I^*$ since either $A_i=A_i^*=0$ or $A_i$ is a positive multiple of $A_i^*$. 
This proves the claim. Therefore equation \eqref{eq:proving_simple} yields
\begin{equation}\label{eq:concl1}
\ell=\ell^*\leq 1,\qquad\;  A_i=0\quad \forall i\in I\setminus I^*,\qquad\; A_i=A_i^* \quad \forall i\in I^*.	
\end{equation}

 We next show $I= I^*$. Assume $I\neq I^*$ for contradiction. Then \eqref{eq:concl1} implies that $\tilde v_i$ for $i\in I\setminus I^*$ does not have a puncture, i.e.~$\tilde v$ is defined on the whole $\mathbb{R}\times S^1$, and $c_1^{T\Sigma}( [\pi_Y\circ v_i ] )=0$. 
 As discussed after \eqref{multiplicity}, this implies that $\partial_s\tilde v_i=0$ for a Floer cylinder $\tilde v_i$, which is not allowed. Thus $\tilde v_i$ for $i\in I\setminus I^*$ cannot be a Floer cylinder but is a Floer continuation cylinder without a puncture. Therefore we actually have $\ell=\ell^*\geq N^*_0$ rather than \eqref{eq:ell*}. Now equation \eqref{eq:proving_simple} reads
\[
\begin{split}
	0\geq (N^*-N^*_0)+(\ell-N^*_0)  + \ell+(i(\tilde{p})-i(\tilde{q})+1) + \sum_{j=1}^{\ell} | \gamma_j | ,
\end{split}
\]
and all  five terms on the right-hand side are nonnegative. We deduce that 
$\ell=0$, $N^*_0=0$, and $N^*=0$. 
This contradicts $N^*\geq1$, and hence $I=I^*$. 

Hence, by \eqref{eq:concl1}, $\mathbf{A}=\mathbf{A}^*$. Moreover, if $\ell=1$ and $\mathbf{u}=[u]$ is not somewhere injective,  the asymptotic Reeb orbit $\gamma$ of $u$ at the puncture is multiply covered. Then \eqref{RS-index} and \eqref{eq:index_moduli} imply $|\gamma|\geq2$, which contradicts \eqref{eq:proving_simple}. Thus $u$ is somewhere injective. This completes the proof of  \eqref{proving_simple}.

\medskip

To prove the rest of Proposition \ref{continuation_transversality}.(b), we continue to assume $\mu( \tilde{p}^H_{k_+} ) =\mu( \tilde{q}^L_{k_-} )$. Let $\big ( ((\mathbf{v},\mathbf{z}), \mathbf{t}), \mathbf{u} \big ) \in \mathcal{M}_{N, \ell}^*(\tilde{q}^L_{k_-},\tilde{p}^H_{k_+}; \mathbf{A}, \mathbf{B}; H_s;J_X)$. We use the same notation as above and denote by $N_0$  the number of zero entries in $\mathbf{A}$. Now equation \eqref{eq:proving_simple} simplifies to 
\begin{equation}\label{index inequality}
    1 \geq (N-N_0)  + (\ell - N_0 + 1) + \ell + (i(\tilde{p})-i(\tilde{q})+1) + \sum_{j=1}^\ell | \gamma_j |.
\end{equation}

Since $\Pi\big (((\mathbf{v}, \mathbf{z}),\mathbf{t}),\mathbf{u}\big )$ is simple, we also have
\begin{equation}\label{index equality}
	\begin{split}
  1 & = (i(\tilde{p})-i(\tilde{q})+1)+(\ind_{f_\Sigma}(p)-\ind_{f_\Sigma}(q)) + \sum_{i=1}^{N} 2c_1^{T\Sigma}(A_i) + 2\ell +\sum_{j=1}^\ell | \gamma_j |\\
    &=(i(\tilde{p})-i(\tilde{q})+1)+\dim \mathcal{N}^*_{N,\ell}(q,p;\mathbf{A}, \mathbf{B};J_X)-N+1,
	\end{split}
\end{equation}
see \eqref{eq:index_0}. We investigate all possible cases of \eqref{index inequality} separately. There are six cases, and  we show that only the fourth case can happen.
\begin{enumerate}[(1)]
    \item  $N-N_0=\mathbf{1}$, $\ell-N_0+1=0$, $\ell=0$, $i(\tilde{p})-i(\tilde{q})+1=0$, $\sum\limits_{j=1}^{\ell} | \gamma_j | = 0$: 
    
    Since $N_0=1$ and $N=2$, by \eqref{index equality} we have $\dim\mathcal{N}^*_{N=2, \ell = 0}(q,p;\mathbf{A};J_\Sigma)=2$. 
    Let us denote $\Pi ((\mathbf{v}, \emptyset), t ) = ((\mathbf{w},\emptyset),t) \in \mathcal{N}^*_{N=2, \ell = 0}(q,p;\mathbf{A};J_\Sigma)$, $\mathbf{v}=(\tilde v_1,\tilde v_2)$, $\mathbf{w} := (w_1, w_2)$, and $t \in \mathbb{R}_{> 0}$ as usual.
    Without loss of generality, we may assume that $w_1$ is a nonconstant $J_\Sigma$-holomorphic sphere and $w_2$ is a constant one. 
    Since $((\mathbf{w},\emptyset), t)$ is simple, $w_1(-\infty)$ is different from the image of $w_2$ and thus  $((w_1,\varphi_{Z_\Sigma}^{t'}\circ w_2), t - t' ) \in \mathcal{N}^*_{N=2, \ell = 0}(q,p;\mathbf{A};J_\Sigma)$ for any $0 < t' < t$. 
    Together with the $\mathbb{C}^*$-action on the domain of $w_1$, we get $\dim\mathcal{N}^*_{N=2, \ell = 0}(q,p;\mathbf{A};J_\Sigma)\geq 3$ which leads to a contradiction.

    \item  $N-N_0=0$, $\ell-N_0+1=\mathbf{1}$, $\ell=0$, $i(\tilde{p})-i(\tilde{q})+1=0$, $\sum\limits_{j=1}^{\ell} | \gamma_j | = 0$:
    
    This contradicts $N \geq 1$ which is due to the presence of a Floer continuation cylinder.

    \item  $N-N_0=0$, $\ell-N_0+1=0$, $\ell=\mathbf{1}$, $i(\tilde{p})-i(\tilde{q})+1=0$, $\sum\limits_{j=1}^{\ell} | \gamma_j | = 0$:
    
    By \eqref{index equality}, $\ind_{f_\Sigma} (p) - \ind_{f_\Sigma} (q) = -1$.
    Since $N=N_0 = 2$, we have 
    \[
    \Pi\big(((\mathbf{v}, z), t),u\big) = \big(((\mathbf{w},z), t),u \big) \in \mathcal{N}^*_{N=2, \ell = 1} \big(q,p;\mathbf{A}=(0,0),\mathbf{B}=B;J_X\big).
    \] 
    Thus there is a gradient flow line of $Z_\Sigma$ from $q$ to $p$. This derives a contradiction $\ind_{f_\Sigma}(p)\geq\ind_{f_\Sigma}(q)$.

    \item  $N-N_0=0$, $\ell-N_0+1=0$, $\ell=0$, $i(\tilde{p})-i(\tilde{q})+1=\mathbf{1}$, $\sum\limits_{j=1}^{\ell} | \gamma_j | = 0$:
    
     We have $N_0=N=1$, i.e.~$\mathbf{A}=A=0$, and $(\tilde{p},\tilde{q})=(\hat{p},\hat{q})$ or $(\tilde{p},\tilde{q})=(\check{p},\check{q})$. By \eqref{index equality}, $\ind_{f_\Sigma}(p)=\ind_{f_\Sigma}(q)$, and this together with $A=0$ yields $p=q$. Moreover, \eqref{multiplicity} implies  $k_-=k_+$. Therefore, due to Proposition \ref{useful_prop}, $\mathcal{M}_{N=1, \ell=0}^*(\hat{p}^L_{k_+},\hat{p}^H_{k_+}; A=0; H_s;J_Y)$ consists of a single element $(\mathbf{v}=\tilde v,\mathbf{z}=\emptyset)$ with $\pi_{\mathbb{R}\times Y}\circ\tilde v=p$. The same holds with $(\hat{p}^L_{k_+},\hat{p}^H_{k_+})$ replaced by $(\check{p}^L_{k_+},\check{p}^H_{k_+})$.

    \item  $N-N_0=0$, $\ell-N_0+1=0$, $\ell=0$, $i(\tilde{p})-i(\tilde{q})+1=0$, $\sum\limits_{j=1}^{\ell} | \gamma_j | = \mathbf{1}$:
    
    This is absurd as $\ell=0$.
    
    \item  $N-N_0=0$, $\ell-N_0+1=0$, $\ell=0$, $i(\tilde{p})-i(\tilde{q})+1=0$, $\sum\limits_{j=1}^{\ell} | \gamma_j | = 0$:
    
    We have $N_0=N=1$, i.e.~$\mathbf{A}=A=0$, $(\tilde p,\tilde q)=(\check p,\hat q)$, and $\ind_{f_\Sigma}(p)-\ind_{f_\Sigma}(q)=1$ by \eqref{index equality}. Thus, $\ind_{f_Y}(\check p)= \ind_{f_Y}(\hat q)$ and $k_-=k_+$ by \eqref{multiplicity}. Applying Proposition \ref{useful_prop}, we observe $\ev_+(v)=\ev_-(v)$ for $\tilde{v}=(b,v)\in \mathcal{M}_{N=1, \ell=0}^*(\hat{q}^L_{k_+},\check{p}^H_{k_+}; A=0; H_s;J_Y) $. This yields a nontrivial gradient flow line of $Z_Y$ connecting $\check p$ and $\hat q$.
    This contradicts $\ind_{f_Y}(\check p)= \ind_{f_Y}(\hat q)$. 
\end{enumerate}
We have shown that only case (4) occurs, and in this case Floer continuation cylinders are exactly as described in the statement of Proposition \ref{continuation_transversality}.(b). It remains to show that such Floer continuation cylinders are positively oriented. This is verified in the following lemma. Hence the proof of Proposition \ref{continuation_transversality}.(b) is complete. \qed

\begin{lemma}\label{lem:conti_sign}
	Every element in $\mathcal{M}_{N=1,\ell=0}^*(\hat{p}^{L}_{k},\hat{p}^H_{k};A=0; H_s;J_Y)$ is positively oriented. The same statement holds with $\hat{p}^{L}_{k}$ and $\hat{p}^H_{k}$ replaced by $\check{p}^{L}_{k}$ and $\check{p}^H_{k}$, respectively.
\end{lemma}
\begin{proof}
	We only prove the first assertion, and the second one follows analogously. As in \eqref{eq:orient_A=0} and \eqref{eq:orientation_puctured}, we observe 
\begin{equation}\label{eq:continuation_orientation}
\begin{split}
	T_{\tilde v}(\mathcal{M}^*_{N=1, \ell = 0; k,k} (A = 0; H_s; J_Y)) &\cong \mathbb{R} R_{\tilde v} \oplus T_{\pi_Y \circ v}(\mathcal{N}^*_{N=1, \ell = 0}(A=0; J_\Sigma))\\
	 &\cong \mathbb{R} R_{\ev_{+,0}(\tilde v)} \oplus T_p\Sigma\\
	 &\cong T_{\tilde p}Y,
\end{split}
\end{equation}
where $\pi_Y\circ v=p$ and $\ev_{+,0}(\tilde v)=\tilde p$. All the isomorphisms in \eqref{eq:continuation_orientation} preserve orientations, see our orientation conventions \eqref{eq:ori_Y} and \eqref{eq:orientation_puctured}. 
We denote by $\mu_\xi\in T_{\tilde v}(\mathcal{M}^*_{N=1, \ell = 0; k, k} (A = 0; H_s; J_Y))$ the tangent vector corresponding to $\xi\in T_{\tilde{p}}Y$.

As usual, $\mathcal{M}^*_{N = 1, \ell = 0}(\hat{p}^L_{k}, \hat{p}^H_{k}; A = 0; H_s; J_Y)$ inherits an orientation from the natural isomorphism to the fiber product
\begin{equation}\label{eq:fiber_product_conti}
	W^s_{Z_Y} (\hat{p}) \prescript{}{\iota^s_Y} \times^{}_{\ev_{+, 0}} \mathcal{M}^*_{N=1, \ell = 0; k, k}(A = 0; H_s; J_Y) \prescript{}{\ev_{-, 0}} \times^{}_{\iota^u_Y} W^u_{Z_Y} (\hat{p}),
\end{equation}
whose elements have the form $(\hat{p},\tilde{v},\hat{p})$ with $\hat{p}=\ev_{\pm,0}(\tilde v)$. Let $(\hat{p},\tilde{v},\hat{p})$ be an element in this fiber product. We first consider 
\[
d_{\hat p}\iota_Y^s-d_{\tilde v}\ev_{+,0}:T_{(\hat{p},\tilde{v})}\big(W^s_{Z_Y} (\hat{p}) \times \mathcal{M}^*_{N=1, \ell = 0; k, k}(A = 0; H_s; J_Y)\big)\longrightarrow T_{\hat p} Y.
\]
Note that 
\[
T_{(\hat{p},\tilde{v})}\big(W^s_{Z_Y} (\hat{p}) \prescript{}{\iota^s_Y} \times^{}_{\ev_{+, 0}} \mathcal{M}^*_{N=1, \ell = 0; k, k}(A = 0; H_s; J_Y)\big) = \ker (d_{\hat p}\iota_Y^s-d_{\tilde v}\ev_{+,0}).
\]
We orient this space by requiring the isomorphism
\begin{equation}\label{eq:1st_fiber_product}
	 \ker (d_{\hat p}\iota_Y^s-d_{\tilde v}\ev_{+,0})\cong T_{\hat p} W^s_{Z_Y} (\hat{p}),\qquad (\xi,\mu_\xi)  \mapsto \xi
\end{equation}
to be orientation-preserving and check this obeys the fibered sum rule in Definition \ref{def:fiber_sum}. Then the isomorphism 
\[
\frac{T_{(\hat{p},\tilde{v})}\big(W^s_{Z_Y} (\hat{p}) \times \mathcal{M}^*_{N=1, \ell = 0; k, k}(A = 0; H_s; J_Y)\big)}{\ker (d_{\hat p}\iota_Y^s-d_{\tilde v}\ev_{+,0})}\cong T_{\hat p} Y
\]
induced by $d_{\hat p}\iota_Y^s-d_{\tilde v}\ev_{+,0}$ is orientation-reversing due to the negative sign in front of $d_{\tilde v}\ev_{+,0}$ and the fact that the dimension of $Y$ is odd. This indeed obeys the fibered sum rule since both $\mathcal{M}^*_{N=1, \ell = 0; k, k}(A = 0; H_s; J_Y)$ and $Y$ have odd dimension. 

Next we consider 
\begin{equation}\label{eq:conti_ori1}
d_{\tilde v}\ev_{-,0}|_{\ker (d_{\hat p}\iota_Y^s-d_{\tilde v}\ev_{+,0})}- d\iota^u_Y:\ker (d_{\hat p}\iota_Y^s-d_{\tilde v}\ev_{+,0}) \oplus  T_{\hat p} W^u_{Z_Y} (\hat{p})\longrightarrow T_{\hat p} Y.
\end{equation}
Due to the transversality result in Proposition \ref{continuation_transversality}.(a), this map is an isomorphism and its trivial kernel equals the tangent space of the fiber product in \eqref{eq:fiber_product_conti} at $(\hat p,\tilde v, \hat p)$. The isomorphism in \eqref{eq:conti_ori1} is orientation-preserving exactly when
\begin{equation}\label{eq:positive_condition}
(-1)^{\dim W^s_{Z_Y}(\hat p)\cdot \dim W^u_{Z_Y}(\hat p)+\dim  W^u_{Z_Y}(\hat p)}=1,
\end{equation}
see our orientation conventions in \eqref{eq:stable_splitting_Y} and \eqref{eq:1st_fiber_product}. 
Let us write $\Theta\in\{-1,1\}$ for the orientation of the trivial kernel of the map in \eqref{eq:conti_ori1} given by the fibered sum rule, i.e.~$\Theta=1$ if and only if the kernel is positively oriented. In this case, the fibered sum rule translates into that the isomorphism in \eqref{eq:conti_ori1} is orientation preserving if and only if $(-1)^{\dim W^u_{Z_Y}(\hat p)\cdot \dim Y}=\Theta$. Thus \eqref{eq:positive_condition} is equivalent to $(-1)^{\dim W^u_{Z_Y}(\hat p)\cdot \dim Y}=\Theta$, and in turn $\Theta=1$. This proves that every element in the space in \eqref{eq:fiber_product_conti} is positively oriented.
\end{proof}

\subsection{Split Rabinowitz Floer homology}

For $a,b \in \mathbb{R}$ with $a<b$, we take $H \in \mathcal{H}$, which is of the form  $H(r, y)=h(e^r)$ for $(r,y) \in \mathbb{R} \times Y\subset \widehat{W}$ by definition, satisfying $h(0) > -a$. For $J_X \in \mathcal{J}^{\reg}_X$ and a monotone homotopy $H_s$ for $H \prec L$, we define a continuation homomorphism
\begin{equation}\label{continuation_map}
\sigma : \FC^{(a,b)}_* (H) \to \FC^{(a,b)}_* (L),\qquad  \tilde{p}^H_{k_+} \mapsto \sum_{\tilde{q}^L_{k_-} } \sum_{ [\tilde{v}] } \epsilon( [\tilde{v}] ) \, \tilde{q}_{k_-}^L
\end{equation}
where the first sum runs over all $\tilde{q}^L_{k_-}$ with $k_-\in\m(a,b;L)$ and $\mu(\tilde{p}^H_{k_+})-\mu(\tilde{q}^L_{k_-})= 0$, and the second one is over all $[ \tilde{v} ] \in \overbar{\mathcal{M}}(\tilde{q}^L_{k_-},\tilde{p}^H_{k_+}; H_s; J_X)$. 
Here, the orientation sign $\epsilon( [\tilde{v}] ) \in \left \{ 1, -1 \right \}$ is determined as $\epsilon([\tilde{v}]) = 1$ if $[\tilde{v}]$ is positively oriented and $\epsilon( [\tilde{v}] ) = -1$ otherwise.
The map $\sigma$ corresponds to a continuation homomorphism in the ordinary Floer homology via neck-stretching, see also Remark \ref{rem:conti_two_levels}. Hence $\sigma$ is a chain map.
 Proposition \ref{continuation_transversality}.(b) yields 
\begin{equation}\label{continuation_identity}
\sigma(\hat{p}^H_k) = \hat{p}^L_k ,\qquad \sigma(\check{p}^H_k) = \check{p}^L_k,
\end{equation}
for every $p \in \Crit f_{\Sigma}$ and $k \in \m(a,b;H)$.  
As shown in Proposition \ref{prop:cpt}, the coefficient of $\tilde{q}_{k_-}^H$ in $\partial\tilde{p}_{k_+}^H$ (and likewise for $L$), where $\partial$ denotes the boundary operator in \eqref{differential}, is independent of $H$ (and of $L$). This together with \eqref{continuation_identity} provides an alternative proof of the fact that $\sigma$ is a chain map in this case. 

The modules $\FH_*^{(a,b)}(H)$ and the homomorphisms induced by \eqref{continuation_map} at the homology level form a direct system over $(\mathcal{H},\prec)$.  
We define the Rabinowitz Floer homology of $\partial W$ for finite action-window $(a,b)\subset\mathbb{R}$ by
\[
\SH_{j}^{(a,b)} (\partial W) : = \varinjlim_{H\in\mathcal{H}} \FH_{j}^{(a,b)}(H),\qquad j\in\mathbb{Z}
\]
Using this and natural action filtration chain homomorphisms
\begin{equation*}\label{action_filter_map}
\FC_*^{(a,b)}(H)  \twoheadrightarrow  \FC_*^{(a',b)}(H),\qquad    \FC_*^{(a,b)}(H) \hookrightarrow \FC_*^{(a,b')}(H)
\end{equation*}
for $a<a'<b<b'$ which induce a bidirect system on $\SH^{(a,b)}(\partial W)$, we  complete the action-window and define the Rabinowitz Floer homology of $\partial W$ by
\begin{equation*}\label{eq:def_RFH}
\SH_{j} (\partial W) : = \varinjlim_{b\uparrow+\infty}\varprojlim_{a\downarrow-\infty}  \SH_{j}^{(a,b)} (\partial W),\qquad j\in\mathbb{Z}.	
\end{equation*}
This split version of Rabinowitz Floer homology is isomorphic to the ordinary one defined in \cite{CFO} via neck-stretching.

For index reasons, see \eqref{grading}, there are at most finitely many generators in each degree. Therefore, for a given degree $j$, the two limits on action-window stabilize for sufficiently negative $a$ and positive $b$. 
Moreover, the limit on $\mathcal H$ also stabilizes eventually at each  degree due to \eqref{continuation_identity}. We conclude that, for each $j\in\mathbb{Z}$, there exist $a,b\in\mathbb{R}$ and $H\in\mathcal H$ such that all three direct systems become trivial after this stage, and hence
\begin{equation*}\label{eq:SH=FH}
\SH_j(\partial W)=\FH_j^{(a,b)}(H).	
\end{equation*}

In fact, the identities in \eqref{continuation_identity} imply that $(\FC_*^{(a,b)}(H), \sigma)$ is a direct system already at the chain level. 
We define
\begin{equation}\label{eq:chain_completion}
\mathrm{SC}_j^{(a,b)}(\partial W):=\varinjlim_{H \in \mathcal{H}} \mathrm{FC}_j^{(a,b)}(H),\qquad \mathrm{SC}_j(\partial W) := \varinjlim_{b\uparrow+\infty}\varprojlim_{a\downarrow-\infty}   \mathrm{SC}_j^{(a,b)}(\partial W)	
\end{equation}
and denote the induced boundary operators again by $\partial$ for both cases. In view of \eqref{continuation_identity}, we can write 
\begin{equation}\label{eq:rabinowitz_chain_cplx}
    \SC_*(\partial W) = \bigoplus_{k\in m_W\mathbb{Z}}\,\bigoplus_{p\in\Crit f_\Sigma}\mathbb{Z}\langle \,\hat \p_k,\check \p_k\rangle 
\end{equation}
where $\hat \p_k$ and $\check \p_k$ are the limits of $\hat p_k^H$ and $\check p_k^H$ respectively for $p \in \Crit f_\Sigma$ and large $H\in \mathcal{H}$. 
For the same reason as above, all three limits in \eqref{eq:chain_completion} stabilize for each degree $j$. 
Therefore, even the inverse limit commutes with taking homology, and we have
\[
\SH_j^{(a,b)}(\partial W)\cong \H_j(\mathrm{SC}_*^{(a,b)}(\partial W),\partial),\qquad \SH_j(\partial W)\cong \H_j( \mathrm{SC}_*(\partial W),\partial).
\]

\section{Floer Gysin exact sequence}\label{sec:Gysin}

Before embarking on the construction of a Gysin-type exact sequence for $\SH_{*} (\partial W)$, we first review the ordinary Gysin exact sequence in the context of Morse homology, see also \cite{Fu1,Oan,GG20,AK23}.

\subsection{Ordinary Gysin exact sequence} \label{sec:Morse_Gysin}

Let $(f_\Sigma, Z_\Sigma)$ be a Morse-Smale pair.
For each $p \in \Crit f_\Sigma$, we choose an arbitrary orientation on $W^s_{Z_\Sigma} (p)$ as in Section \ref{sec:Morse}.
In order to orient $W^s_{Z_\Sigma} (p) \cap W^u_{Z_\Sigma} (q)$ for $p, q \in \Crit f_\Sigma$, we consider the exact sequence
\begin{equation}\label{eq:ses}
	0\to T_r \big( W^s_{Z_\Sigma}(p) \cap W^u_{Z_\Sigma}(q) \big) \to T_r W^s_{Z_\Sigma}(p) \to T_r \Sigma/ T_r W^u_{Z_\Sigma}(q)\to0
\end{equation}
for $r \in W^s_{Z_\Sigma}(p) \cap W^u_{Z_\Sigma}(q)$, where two nontrivial maps are induced by canonical inclusion and projection. 
Translating the oriented  space $T_qW^s_{Z_\Sigma}(q)$, we obtain an oriented vector space $N_r$ in $T_r \Sigma$ such that $T_r \Sigma = T_r W^u_{Z_\Sigma} (q) \oplus N_r$.
The canonical isomorphism $N_r\cong  T_r \Sigma/ T_r W^u_{Z_\Sigma}(q)$ induces an orientation on the latter space. 
This together with the  orientation on $W^s_{Z_\Sigma}(p)$ gives rise to an orientation on $W^s_{Z_\Sigma}(p) \cap W^u_{Z_\Sigma}(q)$ according to \eqref{eq:ses}. 

Suppose that $\ind_{f_\Sigma} (p) - \ind_{f_\Sigma} (q) = 1$. Recalling that $W^s_{Z_\Sigma} (p) \cap W^u_{Z_\Sigma} (q)$ carries a  natural $\mathbb R$-action induced by the flow of $Z_\Sigma$, we define the sign 
\[
\epsilon([r])\in\{-1,1\},\qquad [r]\in \big(W^s_{Z_\Sigma}(p)\cap W^u_{Z_\Sigma}(q)\big)/\mathbb{R}
\] 
by $\epsilon([r])=1$ if the orientation on $T_{r}(W^s_{Z_\Sigma}(p)\cap W^u_{Z_\Sigma}(q))$ agrees with the one given by $Z_\Sigma$ and by $\epsilon([r])=-1$ otherwise. 
We refer to \cite{Ush14} for a detailed discussion on orientations.

We define the Morse chain module of $f_\Sigma$ by
\begin{equation}\label{eq: Morse_chain}
    \MC_*(f_\Sigma)=\bigoplus_{p \in \Crit f_\Sigma} \mathbb{Z} \langle p \rangle
\end{equation}
with grading given by $\ind_{f_\Sigma}(p)$. 
The boundary operator associated with $Z_\Sigma$ is defined by the $\mathbb{Z}$-linear extension of
\begin{equation}\label{eq: Morse_chain_bdry}
  \partial_{Z_\Sigma}:\MC_*(f_\Sigma)\to \MC_{*-1}(f_\Sigma),\qquad   \partial_{Z_\Sigma}(p)=\sum_{q}\sum_{[r]} \epsilon([r]) q,
\end{equation}
where the sums range over all $q \in \Crit f_\Sigma$ with $\ind_{f_\Sigma}(q)=\ind_{f_\Sigma}(p)-1$ and $[r]\in (W^s_{Z_\Sigma}(p)\cap W^u_{Z_\Sigma}(q))/\mathbb{R}$. The homology of this chain complex, denoted by $\MH_*(f_\Sigma)=\MH_* (f_\Sigma, Z_\Sigma)$, is isomorphic to the singular homology of $\Sigma$ with $\mathbb{Z}$-coefficients.

\medskip

Next, we recall the circle bundle $\pi_Y:Y\to\Sigma$ oriented by the Reeb vector field $R$. We choose a closed oriented smooth submanifold $\mathfrak{N} \subset \Sigma$ of codimension $2$ such that 
\begin{enumerate}[(i)]
    \item the homology class of $\mathfrak{N}$ is Poincar\'e dual to the Euler class $e_Y \in \H^2(\Sigma ; \mathbb Z)$ of $Y \to \Sigma$,    
    \item $\mathfrak{N}$ intersects $W^s_{Z_\Sigma} (p) \cap W^u_{Z_\Sigma} (q)$ transversely when $\ind_{f_\Sigma} (p) - \ind_{f_\Sigma} (q) \leq 2$.
\end{enumerate}
Condition (ii) implies that $\mathfrak{N}$ and $(W^s_{Z_\Sigma} (p) \cap W^u_{Z_\Sigma} (q))$ are disjoint if $\ind_{f_\Sigma} (p) - \ind_{f_\Sigma} (q) \leq 1$. We choose a small  neighborhood  $U_{\mathfrak{N}}$ of $\mathfrak{N}$ in $\Sigma$ such that 
\[
U_{\mathfrak{N}} \cap \big(W^s_{Z_{\Sigma}} (p) \cap W^u_{Z_{\Sigma}} (q)\big)=\emptyset \qquad\textrm{provided }\; \ind_{f_\Sigma} (p) - \ind_{f_\Sigma} (q) \leq 1.
\]
The bundle $\pi_Y:Y\to\Sigma$ restricted to $\Sigma \setminus U_{\mathfrak{N}}$ is trivial, and  we fix a trivialization 
\begin{equation}\label{eq:trivial}
Y|_{\Sigma \setminus U_{\mathfrak{N}}} \cong (\Sigma \setminus U_{\mathfrak{N}}) \times S^1.
\end{equation}
 We  choose a connection 1-form $\theta$ on $Y$ which corresponds to $dt$ on $Y|_{\Sigma \setminus U_{\mathfrak{N}}}$ via the isomorphism in \eqref{eq:trivial} where $t$ is the coordinate on $S^1$. We also choose a perfect Morse function $h_{S^1} : S^1 \to \mathbb R$  and  a smooth gradient-like vector field   $Z_{S^1}$ for $h_{S^1}$. We extend these to $Y|_{\Sigma \setminus U_{\mathfrak{N}}}$ in a trivial way using the trivialization \eqref{eq:trivial} and use the same notations. 
For each $p \in \Crit f_\Sigma$, we take a small neighborhood $U_p \subset \Sigma \setminus U_{\mathfrak{N}}$ and  a smooth cutoff function $\beta : \Sigma \to [0,1]$  which has support in the union of $U_p$ over $p\in\Crit f_\Sigma$  and is identically $1$ even nearer to $p \in \Crit f_\Sigma$.
We define the pair $(f_Y, Z_Y)$ by 
\begin{align*}\label{eq:MS_pair_Y}
        f_Y := \pi_Y^* f_\Sigma + \epsilon ( \beta \circ \pi_Y ) h_{S^1}, \qquad 
        Z_Y := Z^\mathrm{h}_\Sigma + \epsilon ( \beta \circ \pi_Y ) Z_{S^1},
\end{align*}
for sufficiently small $\epsilon > 0$. 
Here $Z^\mathrm{h}_\Sigma$ denotes the horizontal lift of $Z_\Sigma$ to $Y$ with respect to $\theta$. We note that 
\begin{equation}\label{eq:trivial1}
Z_Y=(Z_\Sigma,\epsilon\beta Z_{S^1}) \qquad  \textrm{on} \quad Y|_{\Sigma \setminus U_{\mathfrak{N}}} 	
\end{equation}
in the trivialization \eqref{eq:trivial} due to our choice of $\theta$ and that $Z_Y$ is a gradient-like vector field for  the Morse function $f_Y$ satisfying \eqref{crit_lift}. Adapting arguments in  \cite[Section 3.3]{Oan}, we can perturb $\theta$ inside $Y|_{U_{\mathfrak{N}}}$ so that 
  $(f_Y, Z_Y)$ satisfies the Morse-Smale condition.

 For $p, q \in \Crit f_\Sigma$ with $\ind_{f_\Sigma} (p) - \ind_{f_\Sigma} (q) = 2$, the space $ ( W^s_{Z_\Sigma} (p) \cap W^u_{Z_\Sigma} (q)  ) \cap \mathfrak{N}$ consists of finitely many points. In this case, 
for $r \in  ( W^s_{Z_\Sigma} (p) \cap W^u_{Z_\Sigma} (q)  ) \cap \mathfrak{N}$, the sign $\epsilon(r) \in \left \{ 1, -1 \right \}$ is defined by $\epsilon(r) = 1$ if the splitting $T_r  ( W^s_{Z_\Sigma} (p) \cap W^u_{Z_\Sigma} (q)  ) \oplus T_r \mathfrak{N} \cong T_r \Sigma$ is orientation-preserving and by $\epsilon(r) = -1$ otherwise.

\begin{prop}\label{prop:morse_moduli}
    Let $\tilde{p}, \tilde{q} \in \Crit f_Y$ satisfy $\ind_{f_Y} (\tilde{p}) - \ind_{f_Y} (\tilde{q}) = 1$. Recall that $\tilde p \in \{\hat p,\check p\}$, where $\hat{p}$ and $\check{p}$ denote the maximum and the minimum point of $f_Y|_{Y_p}$. 
    If $W^s_{Z_Y}(\tilde{p}) \cap W^u_{Z_Y} (\tilde{q})$ is not empty, then one of the following holds:
    \begin{enumerate}[(a)]
        \item $(\tilde{p}, \tilde{q}) = (\hat{p}, \check{p})$ and $\# \big ( W^s_{Z_Y} (\hat{p}) \cap W^u_{Z_Y} (\check{p})  / \mathbb R \big ) = 0$.
        \item $(\tilde{p}, \tilde{q}) =(\check{p}, \check{q}) $ with $\ind_{f_\Sigma} (p) - \ind_{f_\Sigma} (q) = 1$ and 
        \[
            \# \big ( W^s_{Z_Y} (\check{p}) \cap W^u_{Z_Y} (\check{q})  / \mathbb R \big ) = \# \big ( W^s_{Z_\Sigma} (p) \cap W^u_{Z_\Sigma} (q)  / \mathbb R \big ). 
        \]
        \item $(\tilde{p}, \tilde{q}) = (\hat{p}, \hat{q})$ with $\ind_{f_\Sigma} (p) - \ind_{f_\Sigma} (q) = 1$ and 
        \[
             \# \big ( W^s_{Z_Y} (\hat{p}) \cap W^u_{Z_Y} (\hat{q})  / \mathbb R \big ) = -\# \big ( W^s_{Z_\Sigma} (p) \cap W^u_{Z_\Sigma} (q)  / \mathbb R \big ). 
        \]
        \item $(\tilde{p}, \tilde{q}) = (\check{p}, \hat{q})$ with $\ind_{f_\Sigma} (p) - \ind_{f_\Sigma} (q) = 2$ and 
        \[
            \# \big ( W^s_{Z_Y} (\check{p}) \cap W^u_{Z_Y} (\hat{q})  / \mathbb R \big ) =  \# \big ( ( W^s_{Z_\Sigma} (p) \cap W^u_{Z_\Sigma} (q) ) \cap \mathfrak{N} \big ).
        \]
    \end{enumerate}
    Here, $\#$ denotes the signed cardinality with respect to orientation signs $\epsilon$ defined above. 
\end{prop}

\begin{proof}
The fact that flow lines of $Z_Y$ project to flow lines of $Z_\Sigma$ together with 
a simple index computation shows that only these four cases may occur. 
In case (a), there are exactly two flow lines of $Z_Y$ contained in the fiber $Y_p$ with opposite signs.

Next, we prove (b), and an analogous argument also proves (c). 
Let $p, q \in \Crit f_\Sigma$ with $\ind_{f_\Sigma} (p) - \ind_{f_\Sigma} (q) = 1$.
Let $\tilde{\gamma}:\R\to Y$ be a trajectory of $Z_Y$ converging to $\check{p}$ and $\check{q}$ as $s\in\R$ goes to $+\infty$ and $-\infty$ respectively. Then $\pi_Y \circ \tilde{\gamma}$ is a trajectory of $Z_\Sigma$ connecting $p$ and $q$. 
Conversely, let $\gamma$ be a flow line of $Z_\Sigma$ converging to $p$ and $q$ asymptotically. Then $\gamma$ has its image inside $\Sigma\setminus U_{\mathfrak{N}}$. For the minimum point  $m \in S^1$ of $h_{S^1}$, 
we set $\tilde{\gamma} (s) : = (\gamma(s), m)$ which we view as a curve in $Y$ via the trivialization  \eqref{eq:trivial}. Due to \eqref{eq:trivial1}, it is a unique flow line of $Z_Y$ which connects $\check p$ and $\check q$ and projects to $\gamma$. 
This proves that the projection
\[
\pi_Y:W^s_{Z_Y} (\check p) \cap W^u_{Z_Y} (\check q) \to W^s_{Z_\Sigma} (p) \cap W^u_{Z_\Sigma} (q)
\]
is a diffeomorphism. Moreover, according to our orientation conventions in Section \ref{sec:Morse}, this preserves orientations.

For (d), we recall a related result from \cite{AK23}, where a corresponding identity is proved in the context of the Morse-Bott chain complex of $(\pi_Y^* f_\Sigma, h_{S^1}|_{\Crit \pi_Y^* f_\Sigma})$. We orient  $\pi_Y^{-1}(W^s_{Z_\Sigma}(p))$ and $\pi_Y^{-1}(W^u_{Z_\Sigma}(q))$ so that the isomorphisms 
\begin{align*}\label{eq:ori_preimage}
    \begin{split}
        T_y \big(\pi_Y^{-1}(W^s_{Z_\Sigma}(p))\big) &\cong \mathbb{R}R_y\oplus T_{\pi_Y(y)} W^s_{Z_\Sigma}(p),\\[.5ex]
        T_y \big(\pi_Y^{-1}(W^u_{Z_\Sigma}(q))\big) &\cong \mathbb{R}R_y\oplus T_{\pi_Y(y)} W^u_{Z_\Sigma}(q)
    \end{split}
\end{align*} 
induced by the projection $\pi_Y$ preserve orientations. In particular, $\pi_Y$ gives an orientation-preserving isomorphism
\[
T_y\Big(\pi_Y^{-1}(W^s_{Z_\Sigma}(p))\prescript{}{\iota} \times^{}_{\iota} \pi_Y^{-1}(W^u_{Z_\Sigma}(q))\Big)\cong \R R_y\oplus T_{\pi_Y(y)}\Big(W^s_{Z_\Sigma}(p)\prescript{}{\iota} \times^{}_{\iota} W^u_{Z_\Sigma}(q)\Big),
\]
where $\iota$ denote natural inclusion maps into $Y$ or $\Sigma$. This is coherent with the orientation convention in \cite{AK23}. 
For $\tilde{p} \in \Crit f_Y$, we denote by $W^u_{Z_{S^1}} (\tilde{p})$ and $W^s_{Z_{S^1}}(\tilde{p})$ the unstable and stable manifolds of $Z_{S^1}|_{Y_p}$, respectively.
We orient $W^s_{Z_{S^1}} (\hat{p})$ and $W^u_{Z_{S^1}} (\check{p})$ by the Reeb vector field $R$. Recalling that $Y_p$ is also oriented by $R$, we orient $W^u_{Z_{S^1}} (\hat{p})$ and $W^s_{Z_{S^1}} (\check{p})$, each of which consists of a single point, positively.
For $\check p,\hat q\in\Crit f_Y$ with $\ind_{f_Y} (\check{p}) - \ind_{f_Y} (\hat{q}) = 1$, we consider the fiber product space 
\begin{equation*}\label{eq:MB_fiberproduct}
W^s_{Z_{S^1}}(\check{p}) \prescript{}{\iota} \times^{}_{\ev_{+}} \Big ( \pi_Y^{-1}(W^s_{Z_\Sigma}(p))\prescript{}{\iota} \times^{}_{\iota\,} \pi_Y^{-1}(W^u_{Z_\Sigma}(q)) \Big ) \prescript{}{\ev_{-}} \times^{}_{\iota\,} W^u_{Z_{S^1}}(\hat{q}), 
\end{equation*}
where $\iota$ are natural inclusion maps into $Y_p$, $Y_q$, or $Y$, and 
\[
\ev_\pm : \pi_Y^{-1}(W^s_{Z_\Sigma}(p))\prescript{}{\iota} \times^{}_{\iota\,} \pi_Y^{-1}(W^u_{Z_\Sigma}(q)) \to Y,\qquad \ev_\pm (x, x)=\ev_\pm(x) : = \underset{t \to \pm \infty}{\lim} \varphi^t_{Z^\mathrm{h}_\Sigma} (x).
\]
Then it is proved in \cite[Proposition 2.7.(c)]{AK23} that the signed cardinality of this space equals that of $( W^s_{Z_\Sigma} (p) \cap W^u_{Z_\Sigma} (q) ) \cap \mathfrak{N}$, see Remark \ref{rem:sign_AK23} below.

Therefore, it suffices to verify that the natural map
\begin{align*}
   \chi: W^s_{Z_Y} (\check{p}) \cap W^u_{Z_Y} (\hat{q}) &\to W^s_{Z_{S^1}}(\check{p}) \prescript{}{\iota} \times^{}_{\ev_{+}} \Big ( \pi_Y^{-1}(W^s_{Z_\Sigma}(p))\prescript{}{\iota} \times^{}_{\iota} \pi_Y^{-1}(W^u_{Z_\Sigma}(q)) \Big ) \prescript{}{\ev_{-}} \times^{}_{\iota} W^u_{Z_{S^1}}(\hat{q})\\
    r &\mapsto \big ( \ev_+ (r), (r, r), \ev_- (r) \big )
\end{align*}
is an orientation-preserving diffeomorphism. This map is well-defined, i.e.~$\ev_+(r)=\check p$ and $\ev_-(r)=\hat q$, since
\begin{equation}\label{eq:Z_Y=horizontal}
Z_Y (r) = Z^\mathrm{h}_\Sigma (r)\qquad \forall r \in W^s_{Z_Y} (\check{p}) \cap W^u_{Z_Y} (\hat{q}).	
\end{equation}
Equation \eqref{eq:Z_Y=horizontal} follows from that if $( \beta \circ \pi_Y ) Z_{S^1}(r)$ does not vanish at such $r$, the flow line $\varphi_{Z_Y}^t(z)$ would converge to $\hat p$ and $\check q$ as $t\in \mathbb{R}$ goes to $+\infty$ and $-\infty$ respectively.
For the same reason, the map $\chi$ is bijective. To see that $\chi$ is orientation-preserving, we decompose it to the following chain of orientation-preserving diffeomorphisms.
\begin{align*}
    W^s_{Z_Y}(\check{p}) \cap  W^u_{Z_Y} (\hat{q}) &\cong W^s_{Z_Y}(\check{p}) \prescript{}{\iota} \times^{}_{\iota\,}  W^u_{Z_Y} (\hat{q})\\
                                                   &\cong \Big ( W^s_{Z_{S^1}}(\check{p}) \prescript{}{\iota} \times^{}_{\ev_{+}} \pi_Y^{-1}(W^s_{Z_\Sigma}(p)) \Big ) \prescript{}{\iota} \times^{}_{\iota\,} \Big ( \pi_Y^{-1}(W^u_{Z_\Sigma}(q)) \prescript{}{\ev_{-}} \times^{}_{\iota} W^u_{Z_{S^1}}(\hat{q}) \Big )\\
                                                   &\cong W^s_{Z_{S^1}}(\check{p}) \prescript{}{\iota} \times^{}_{\ev_{+}} \Big ( \pi_Y^{-1}(W^s_{Z_\Sigma}(p))\prescript{}{\iota\,} \times^{}_{\iota} \pi_Y^{-1}(W^u_{Z_\Sigma}(q)) \Big ) \prescript{}{\ev_{-}} \times^{}_{\iota\,} W^u_{Z_{S^1}}(\hat{q}).
\end{align*}
The first and the second lines follow from straightforward computations using our orientation conventions. The last one uses the associativity property of the fibered sum orientation rule.
This finishes the proof of (d).
\end{proof}

\begin{remark}\label{rem:sign_AK23}
In \cite{AK23}, they take $\mathfrak{N}$ to satisfy $[\mathfrak{N}]=\PD(-e_Y)$ and  the $\R$-action given by negative gradient vector fields, as opposed to positive gradient-like vector fields $Z_Y$ and $Z_\Sigma$. These conventions differ from ours and result in two negative signs in Proposition \ref{prop:morse_moduli}.(d) which cancel out.
\end{remark}

Let $(\MC_*(f_Y), \partial_{Z_Y})$ be a Morse chain complex defined exactly in the same way as for $(f_\Sigma, Z_\Sigma)$. The induced homology $\MH_*(f_Y)=\MH_*(f_Y,Z_Y)$ is isomorphic to the singular homology $\H_*(Y;\Z)$. 
We decompose the chain module $\MC_*(f_Y)$ into 
\begin{equation*}
    \MC_*(f_Y) = \widehat{\MC}_*(f_Y) \oplus \widecheck{\MC}_*(f_Y)
\end{equation*}
where $\widehat{\MC}_*(f_Y)$ and $\widecheck{\MC}_*(f_Y)$ are the submodules generated by $\hat p$ and $\check p$ respectively for all $p \in \Crit f_\Sigma$. 
Proposition \ref{prop:morse_moduli} readily implies that with respect to this decomposition the boundary operator $\partial_{Z_Y}$ is of the form 
\begin{equation*}
    \partial_{Z_Y} = 
        \begin{pmatrix}
            \hat{\partial} & \delta \\
            0 & \check{\partial}
        \end{pmatrix}
\end{equation*}
where $\check{\partial}$, $\hat{\partial}$, and $\delta$ correspond to (b), (c), and (d) in Proposition \ref{prop:morse_moduli}, respectively. 
Then $\widehat{\MC}_*(f_Y)$ is a subcomplex with $\partial|_{\widehat{\MC}_*(f_Y)}=\hat\partial$, and the induced quotient complex is isomorphic to $(\widecheck{\MC}_*(f_Y),\check\partial)$.
We denote the homologies of these chain complexes by $\widehat{\MH}_* (f_Y)$ and $\widecheck{\MH}_* (f_Y)$.
The short exact sequence of chain complexes $0 \to \widehat{\MC}_*(f_Y) \to \MC_*(f_Y) \to \widecheck{\MC}_*(f_Y) \to 0$  gives rise to the long exact sequence 
\begin{equation}\label{eq:morse_les0}
    \cdots \longrightarrow \widecheck{\MH}_{*+1}(f_Y) \stackrel{\delta_*} {\longrightarrow} \widehat{\MH}_*(f_Y) \longrightarrow \MH_*(f_Y) \longrightarrow  \widecheck{\MH}_*(f_Y) \stackrel{\delta_*}\longrightarrow \cdots.
\end{equation}
Moreover, (b) and (c) in Proposition \ref{prop:morse_moduli} yield the isomorphisms
\begin{equation*}
    (\widehat{\MC}_*(f_Y), \hat{\partial}) \cong (\MC_{*-1}(f_\Sigma), -\partial_{Z_\Sigma}), \qquad (\widecheck{\MC}_*(f_Y), \check{\partial}) \cong (\MC_*(f_\Sigma), \partial_{Z_\Sigma})
\end{equation*}
induced by $\hat{p} \mapsto p$ and $\check{p} \mapsto p$, respectively. Due to Proposition \ref{prop:morse_moduli}.(d), through the above isomorphisms, the map $\delta$ corresponds to the homomorphism 
\[
\MC_*(f_\Sigma)\to \MC_{*-2}(f_\Sigma),\qquad     p \mapsto \sum_{q} \sum_{r} \epsilon (r) q,
\]
where the sums range over all $q \in \Crit f_\Sigma$ with $\ind_{f_\Sigma} (p) - \ind_{f_\Sigma} (q) = 2$ and $r \in ( W^s_{Z_\Sigma} (p) \cap W^u_{Z_\Sigma} (q) ) \cap \mathfrak{N}$. The induced homomorphism at the homology level is isomorphic to the cap product with $e_Y$ via $\MH_*(f_\Sigma)\cong \H_*(\Sigma;\Z)$.
Hence, the exact sequence in \eqref{eq:morse_les0} recovers the Gysin exact sequence for $\pi_Y:Y\to \Sigma$,
\begin{equation}\label{eq:Morse_LES}
    \cdots\longrightarrow \MH_{*+1}(f_\Sigma) \, \stackrel{\cap e_Y}{\longrightarrow} \, \MH_{*-1}(f_\Sigma) \longrightarrow \MH_* (f_Y)  \longrightarrow \MH_{*} (f_\Sigma) \, \stackrel{\cap e_Y}{\longrightarrow} \, \cdots.
\end{equation}

\subsection{Floer Gysin exact sequence}\label{sec:Gysin_complement}
We decompose the chain module $\SC_*(\partial W)$ defined in \eqref{eq:rabinowitz_chain_cplx} into 
\[
\SC_*(\partial W)=\widehat{\SC}{}_*(\partial W)\oplus \widecheck{\SC}_*(\partial W)
\]
where $\widehat{\SC}{}_*(\partial W)$ and $\widecheck{\SC}_*(\partial W)$ are the submodules generated by $\hat \p_k$ and $\check \p_k$, respectively. In view of Proposition \ref{prop:cpt}, \eqref{continuation_identity}, and Proposition \ref{prop:morse_moduli}, the boundary operator $\partial$ on $\SC_*(\partial W)$ with respect to the above decomposition  is of the form 
\begin{equation}\label{upper_triangular}
    \partial = 
        \begin{pmatrix}
            \hat{\partial} & \Delta \\
            0 & \check{\partial}
        \end{pmatrix}\quad
        \text{with}\quad \Delta = \Delta^Y_a + \Delta^Y_b + \Delta^{W,Y}
\end{equation}
such that 
\begin{enumerate}[(i)]
	\item $\hat{\partial}$, $\check{\partial}$, and $\Delta^Y_a$ count solutions in Proposition \ref{prop:cpt}.(a), 
	\item $\Delta^Y_b$ counts solutions in Proposition \ref{prop:cpt}.(b),
	\item $\Delta^{W,Y}$ counts solutions in Proposition \ref{prop:cpt}.(c).
\end{enumerate}
A direct consequence of \eqref{upper_triangular} is that $(\widehat{\SC}{}_*(\partial W),\hat\partial)$ is a subcomplex.
Moreover, the induced quotient complex is isomorphic to $(\widecheck{\SC}_*(\partial W),\check\partial)$. 
We denote the respective homologies by $\widehat{\SH}{}_*(\partial W)$ and $\widecheck{\SH}_*(\partial W)$. 
The short exact sequence of chain complexes $0\to\widehat{\SC}_*(\partial W)\to \SC_*(\partial W)\to\widecheck{\SC}_*(\partial W)\to 0$ gives rise to the long exact sequence 
\begin{equation}\label{LES_SH}
    \cdots \longrightarrow \widecheck{\SH}_{j+1}(\partial W) \stackrel{(\Delta)_*} {\longrightarrow} \widehat\SH{}_j (\partial W) \longrightarrow \SH_{j} (\partial W) \longrightarrow  \widecheck{\SH}_{j} (\partial W) \stackrel{(\Delta)_*}\longrightarrow \cdots.
\end{equation}

To relate this with the Morse homology of $f_\Sigma$, we recall the coefficient ring from the introduction: 
\[
\Lambda = \begin{cases}
	\Z & \quad \text{if }\; m_W=0,\\[.5ex]
	\mathbb Z [ T, T^{-1} ] & \quad \text{if }\;m_W\neq 0.
\end{cases} 
\]
The formal variable $T$ in the Laurent polynomial ring has degree 
\begin{equation}\label{eq:deg_T}
	\deg T = \frac{2(\tau_X - K)}{K}m_W=\mu_\CZ(\gamma)
\end{equation} 
where $\gamma$ is a Reeb orbit on $Y$ with multiplicity $\m(\gamma)=m_W$, see \eqref{RS-index}.
As computed in Section \ref{sec:index}, in case of $m_W=m_\Sigma$, $\deg T$ equals twice the minimal Chern number of $\Sigma$. Replacing the coefficient ring $\Z$ by $\Lambda$ in \eqref{eq: Morse_chain} and extending the boundary operator in \eqref{eq: Morse_chain_bdry} $\Lambda$-linearly, we obtain the chain complex
$(\MC_{*}(f_\Sigma;\Lambda), \partial_{Z_\Sigma})$ 
with grading $\deg (T^ip):=\ind_{f_\Sigma}(p)+ i \deg T$. Its homology is isomorphic to the singular homology of $\Sigma$ with coefficients in $\Lambda$,
\[
\H_*\big(\MC(f_\Sigma;\Lambda), \partial_{Z_\Sigma}\big)\cong \H_*(\Sigma;\Lambda)\cong\H_*(\Sigma)\otimes_\Z\Lambda.
\]
We observe that the $\mathbb{Z}$-homomorphism 
\[
\begin{split}
	\widehat{\SC}{}_* (\partial W)  &\stackrel{\cong}{\longrightarrow} \MC_{* + \frac{\dim \Sigma}{2} -1} (f_\Sigma; \Lambda) ,\qquad  k\in m_W\Z\\
    \hat{\p}_k &\longmapsto T^{ k / { m_W } } p 	
\end{split}
\]
is an isomorphism, see \eqref{grading} for the degree convention. Furthermore, due to Proposition \ref{prop:cpt}.(a) and Proposition \ref{prop:morse_moduli}, the above two modules are isomorphic as chain complexes:
\begin{equation}\label{eq:hat_SC}
\big(\widehat{\SC}{}_*(\partial W),\hat\partial \big) \cong \big(\MC_{* + \frac{\dim \Sigma}{2} - 1}(f_\Sigma; \Lambda),-\partial_{Z_\Sigma}\big).	
\end{equation}
Similarly, we also have an isomorphism
\begin{equation}\label{eq:check_SC}
\big(\widecheck{\SC}_*(\partial W),\check\partial \big) \cong \big(\MC_{*+\frac{\dim \Sigma}{2}}(f_\Sigma; \Lambda),\partial_{Z_\Sigma}\big).
\end{equation}
Through the isomorphisms in \eqref{eq:hat_SC} and \eqref{eq:check_SC}, the map $\Delta = \Delta^Y_a + \Delta^Y_b + \Delta^{W,Y}$ in \eqref{upper_triangular} corresponds to 
\[
\delta := \delta^{\Sigma}_a + \delta^{\Sigma}_b + \delta^{X, \Sigma} : \MC_*(f_\Sigma; \Lambda)\to \MC_{*-2}(f_\Sigma; \Lambda),
\]
such that
\begin{enumerate}[(i)]
	\item according to Proposition \ref{prop:cpt}.(a) and Proposition \ref{prop:morse_moduli}.(d)
	\begin{equation}\label{eq:delta_a}
 \delta^{\Sigma}_a \big(T^{\frac{k}{m_W}} p\big) =\sum_{q} \#\Big(  \big( W^s_{Z_\Sigma} (p) \cap W^u_{Z_\Sigma} (q) \big) \cap \mathfrak{N} \Big) \,  T^{\frac{k}{m_W}} q \,,
\end{equation}
where the sum ranges over all $q\in\Crit f_\Sigma$ with $\ind_{f_\Sigma} (p) - \ind_{f_\Sigma} (q) = 2$,
	\item according to Proposition \ref{prop:cpt}.(b)
	\begin{equation}\label{eq:delta_b}
    \delta^{\Sigma}_b \big( T^{\frac{k}{m_W}} p \big)  = \sum_{(q,A) }   -K\omega_\Sigma(A) \, \cdot \, \# \Big( \mathcal{N}^*_{N=1, \ell = 0}(q, p; A; J_\Sigma)/ {\mathbb C}^* \Big)\, T^{\frac{k-K\omega_\Sigma(A)}{m_W}} q \, ,
\end{equation}
where the sum is taken over all possible pairs $(q,A)\in\Crit f_\Sigma\times\pi_2(\Sigma)$ satisfying $\ind_{f_\Sigma} (p) - \ind_{f_\Sigma} (q) = 2 - 2c_1^{T\Sigma}(A)$,
	\item according to Proposition \ref{prop:cpt}.(c)
\begin{equation}\label{eq:delta_c}
  \delta^{X, \Sigma} \big( T^{\frac{k}{m_W}} p \big)  = \sum_{ B }     K\omega(B) \, \cdot \, \#\Big( \mathcal{N}^*_{N=1, \ell = 1} (p,p;A=0,B;J_X) / \mathbb{C}^* \Big)\, T^{\frac{k-K\omega(B)}{m_W}} p\,,
\end{equation}
where the sum ranges over all $B\in\pi_2(X)$ satisfying $c_1^{TX}(B)-K\omega(B)=1$ and $K\omega(B)\in m_W\Z$.
The condition on $B$ holds only if $K\omega(B) = m_X=m_W$.
Therefore, $\delta^{X,\Sigma}$ is of the form 
\[
\delta^{X, \Sigma}(T^{\frac{k}{m_W}}p) = c_{X,\Sigma} T^{\frac{k}{m_W}-1} p,
\]
where $c_{X,\Sigma} \in \mathbb Z$ is independent of $p$ due to Remark \ref{rmk: indep_of_point} and divisible by $m_X=m_W$.

\end{enumerate}
These computations show that all of $\delta^{\Sigma}_a$, $\delta^{\Sigma}_b$, and $\delta^{X, \Sigma}$ are $\Lambda$-linear chain maps. From \eqref{upper_triangular}, \eqref{eq:hat_SC}, and \eqref{eq:check_SC} we conclude that the mapping cone of $\delta$, namely
\[
C_*(\delta):= \MC_{*-1}(f_\Sigma; \Lambda)\oplus\MC_*(f_\Sigma; \Lambda),\qquad \partial_{\,C(\delta)}:=\begin{pmatrix}
            -\partial_{Z_\Sigma} & \delta \\
            0 & \partial_{Z_\Sigma}
        \end{pmatrix},
\]
is isomorphic to the chain complex $\SC_*(\partial W)$ up to degree shift, explicitly
\begin{equation*}\label{eq:SC_cone}
\big(\SC_*(\partial W),\partial\big)  \cong \big(C_{* + \frac{\dim \Sigma}{2}}(\delta),\partial_{C(\delta)}\big).
\end{equation*}
Moreover, since $(C_*(\delta),\partial_{C(\delta)})$ admits a $\Lambda$-module structure, so does $(\SC_*(\partial W),\partial)$ given by $T\cdot \tilde{\p}_k=\tilde{\p}_{k+ m_W}$, and in turn 
\[
    \SH_*(\partial W)\cong\SH_{*+\deg T}(\partial W).
\]
Finally, the long exact sequence \eqref{LES_SH} reads 
\begin{equation}\label{eq:gysin}
    \cdots \to \H_{*+2}(\Sigma;\Lambda) \stackrel{\delta_*} {\to} \H_* (\Sigma;\Lambda) \to \SH_{*-\frac{\dim \Sigma}{2}+1} (\partial W) \to  \H_{*+1} (\Sigma;\Lambda) \stackrel{\delta_*} \to \cdots,
\end{equation}
where all maps are $\Lambda$-linear. 
\begin{remark}
    In case that $\delta^{X,\Sigma}$ vanishes, the $\Lambda$-module structure on $(\SC_*(\partial W),\partial)$ coincides with the one introduced in \cite{Ueb19} using a generalization of the Seidel representation.
    \end{remark}

\begin{remark} \label{remark: noncontractible_gysin}
    We can also define the Rabinowitz Floer homology $\SH_* (\partial W;\kappa)$ using periodic Reeb orbits with free homotopy class $\kappa$ in $W$. In this case, 
    statements analogous to Proposition \ref{prop:cpt} and Proposition \ref{continuation_transversality} hold with the same proof. 
    Therefore the chain complex for $\SH_* (\partial W;\kappa)$ is also isomorphic to the mapping cone $C_*(\delta)$, and in turn $\SH_* (\partial W;\kappa)$  is isomorphic to $\SH_* (\partial W)$ up to degree shift.  
\end{remark}

\subsubsection*{Proof of Theorem \ref{thm:gysin}}
The long exact sequence claimed in the theorem is constructed in \eqref{eq:gysin} with $\delta^\Sigma = \delta_a^\Sigma + \delta_b^\Sigma$. The following lemma proves  Theorem \ref{thm:gysin}.(a).

\begin{lemma}\label{lemma:quantum_cap}
   The homomorphisms $(\delta_a^\Sigma)_*$ and $(\delta_a^\Sigma + \delta_b^\Sigma)_*$ on $\H_{*}(\Sigma;\Lambda)$ coincide with the ordinary and the quantum cap product with $e_Y$, respectively.
\end{lemma}
\begin{proof}
The thesis that $(\delta_a^\Sigma)_*$ is the ordinary cap product with $e_Y$ follows from \eqref{eq:Morse_LES} and  \eqref{eq:delta_a}. We consider the evaluation map 
\begin{equation*}
    \ev_1 : \mathcal{N}^*_{N = 1, \ell = 0}(q, p; A; J_\Sigma) \to \Sigma, \quad \ev_1(w) := w(1).
\end{equation*}
We take a triple $(p,q,A)$ as in \eqref{eq:delta_b} so that $\mathcal{N}^*_{N = 1, \ell = 0}(q, p; A; J_\Sigma)$ has dimension 2. We may assume that $\ev_1$ is transverse to $\mathfrak{N}$. Then we have 
\begin{equation*} \label{eq: quantum_cap}
    \#\ev_1^{-1}(\mathfrak{N})= e_Y (A) \, \cdot \, \# \Big( \mathcal{N}^*_{N=1, \ell = 0}(q, p; A; J_\Sigma)/ {\mathbb C}^* \Big),
\end{equation*}
where the orientation of $\ev_1^{-1}(\mathfrak{N})$ is determined by the orientation of $\mathcal{N}^*_{N = 1, \ell = 0}(q, p; A; J_\Sigma)$ and the coorientation of $\mathfrak{N}$. The right-hand side of the above equality multiplied with $T^{\frac{e_Y(A)}{m_W}}$ agrees with the coefficient of the map $\delta^{\Sigma}_b$ in \eqref{eq:delta_b}. The cardinality $\#\ev_1^{-1}(\mathfrak{N})$ multiplied with $T^{\frac{e_Y(A)}{m_W}}$ is exactly the quantum contribution part of the quantum cap product with $e_Y$, see \cite[Remark 12.3.3]{MS17}. 
This proves that $(\delta^{\Sigma}_a+\delta^{\Sigma}_b)_*$ agrees with the quantum cap product with $e_Y$.
\end{proof}

Next, we prove Theorem \ref{thm:gysin}.(b). We recall from \eqref{eq:delta_c} that $\delta^{X, \Sigma}$ counts $J_X$-holomorphic spheres in $X$ representing $B\in\pi_2(X)$ with $c_1^{TX}(B)-K\omega(B)=1$. Therefore if $\delta^{X, \Sigma}$ is nonzero, then the homomorphism $c_1^{TX}-K\omega:\pi_2(X)\to\Z$ is  surjective and accordingly $\deg T=2$. The rest follows from the formula of $\delta^{X,\Sigma}$ in \eqref{eq:delta_c}. 
This completes the proof of Theorem \ref{thm:gysin}. \qed

\medskip

Next we prove the claim in Remark \ref{rem:del^c=0}, which was observed in \cite[Proposition 5.1]{BK} and \cite[Lemma 6.6]{DL2}. We include the proof for completeness.

\begin{lemma}\label{nonexistence}
	Suppose that the map $\pi_2(\Sigma)\to\pi_2(X)$ induced by the inclusion is surjective and the minimal Chern number of $\Sigma$ greater than one. Then $\delta^{X, \Sigma}$ vanishes.
\end{lemma}
\begin{proof}
	By the hypothesis, any $B\in\pi_2(X)$ is the image of some $A\in\pi_2(\Sigma)$, and we have $c_1^{TX}(B)-K\omega(B)=c_1^{T\Sigma}(A)\neq \pm 1$. As observed above, this yields $\delta^{X, \Sigma}=0$.
\end{proof}

Finally, we show if $K$ is sufficiently larger than $\tau_X$, as opposed to our standing assumption $K<\tau_X$, we recover the ordinary Gysin exact sequence, as claimed in Remark \ref{rem:K_large}.
\begin{prop}\label{prop:K_large}
Let $(X,\omega)$ and $(\Sigma,\omega_\Sigma)$ be exactly the same as before except that the condition $K<\tau_X$ is replaced by 
\begin{equation}\label{eq:K_large}
\frac{m_X(\tau_X-K)}{K} \leq -\frac{\dim\Sigma}{2}.	
\end{equation}
Then, $\SH_*(\partial W)\cong \H_{*+\frac{\dim \Sigma}{2}}(Y;\Lambda)$ as $\Lambda$-modules. Moreover, the exact sequence in \eqref{eq:gysin} exists and is isomorphic to the ordinary Gysin exact sequence tensored with $\Lambda$:
\[
    \cdots \longrightarrow \H_{*+2}(\Sigma;\Lambda) \stackrel{\delta_*} {\longrightarrow} \H_* (\Sigma;\Lambda) \longrightarrow \H_{*+1} (Y;\Lambda) \longrightarrow  \H_{*+1} (\Sigma;\Lambda) \stackrel{\delta_*} \longrightarrow \cdots,
\]
where $\delta_*$ is the ordinary cap product with $-[K\omega_{\Sigma}]_{\mathbb Z}$.
\end{prop}

\begin{proof} 
   We show that the split Floer complex is well-defined in this case as well.
   In fact, we show that only trajectories of $Z_Y$ contribute to the boundary operator. More precisely, let $\mu(\tilde{p}_{k_+}) - \mu (\tilde{q}_{k_-}) = 1$. Due to the SFT compactness, through the neck-stretching procedure, a sequence (with respect to a stretching parameter) of chains of ordinary (nontrivial) Floer cylinders with trajectories of $Z_Y$ in $\widehat{W}$ connecting $\tilde{p}_{k_+}$ and $\tilde{q}_{k_-}$ converges to a (nontrivial) Floer-holomorphic building with trajectories of $Z_Y$ whose top component is an element of $\mathcal{M}_{N, \ell} (\tilde{q}_{k_-}, \tilde{p}_{k_+};\mathbf{A};H;J_Y)$.    We claim that, under the assumption \eqref{eq:K_large}, such a limit does not exist.

   Assume for contradiction that such a limit exists. As mentioned above, the top component of this limit is of the form $((\mathbf{v},\mathbf{z}),\mathbf{t})\in  \mathcal{M}_{N, \ell} (\tilde{q}_{k_-}, \tilde{p}_{k_+};\mathbf{A};H;J_Y)$. We first observe that, by the assumption \eqref{eq:K_large}, if $\omega_\Sigma(A)>0$ for $A\in\pi_2(\Sigma)$, then $c_1^{T\Sigma}(A)=(\tau_X-K)\omega_\Sigma(A)\leq -\frac{\dim\Sigma}{2}$ .
   Thus, for a dimension reason, there is no nonconstant $J_\Sigma$-holomorphic sphere. In particular, we have $\mathbf{A} = 0$ and $\ell\geq 1$.

   Using \eqref{RS-index} and \eqref{eq:K_large}, we estimate
   \[
   \begin{split}
   1=\mu (\tilde{p}_{k_+}) - \mu(\tilde{q}_{k_-}) &= \ind_{f_Y} \tilde{p} - \ind_{f_Y} \tilde{q} +\frac{2(\tau_X - K)}{K} (k_+ - k_-) \\
   &\leq \ind_{f_Y} \tilde{p} - \ind_{f_Y} \tilde{q} -\dim\Sigma \\
   &\leq 1.
   \end{split}
   \]
  This yields $\frac{2(\tau_X - K)}{K} (k_+ - k_-) = - \dim \Sigma$, and in turn $k_+ - k_- = m_X$ by  \eqref{eq:K_large}.

  On the other hand, there exists at least one holomorphic plane $u$ in $\widehat{W}$ in the limit Floer-holomorphic building. As pointed out in Remark \ref{rem:hol_plane}, we may view $u$ as an element of $\mathcal{P}_{\ell = 1} (B; J_X)$ for some $B\in\pi_2(X)$. 
  The multiplicity $K\omega(B)$ of the asymptotic periodic Reeb orbit of $u$ is at most $k_+-k_-=m_X$ by \eqref{multiplicity}. Therefore $K\omega(B)=m_X$, and every element of  $\mathcal{P}_{\ell = 1} (B; J_X)$ is simple. But, by Proposition \ref{transversality_chain}.(b),  $\mathcal{P}_{\ell = 1} (B; J_X) = \mathcal{P}^*_{\ell = 1} (B; J_X)$ has dimension
   \begin{align*}
       \dim \mathcal{P}^*_{\ell = 1} (B; J_X) &= \dim X + 2c_1^{TX}(B) - 2K\omega(B) - 4 \\
       &= \dim \Sigma + \frac{2(\tau_X - K)}{K} m_X - 2 \\
       &\leq -2.
   \end{align*}
   This contradiction proves the claim. Therefore the boundary operator for the split Floer complex (also for the ordinary Floer complex with a sufficiently stretched almost complex structure) counts only trajectories of $Z_Y$.
  
  A similar argument for continuation Floer cylinders shows that continuation homomorphisms are essentially the identity maps, see \eqref{continuation_identity}. Therefore, we have $\SH_*(\partial W)\cong \H_{*+\frac{\dim \Sigma}{2}}(Y;\Lambda)$, and the construction in Section \ref{sec:Gysin_complement} gives rise to the ordinary Gysin exact sequence for $Y\to\Sigma$. This completes the proof of Proposition \ref{prop:K_large}.
\end{proof}

\subsection{Index positivity and fillability conditions on $Y$}\label{sec:Gysin_symplectization}

The goal of this section is two-fold. 
First, we observe that the index-positivity condition on $Y$ is equivalent to saying that the minimal Chern number of $\Sigma$ is at least two. 
This is the setting of Corollary \ref{cor:gysin_symplectization} and Corollary \ref{cor:orderability}.
Second, we investigate relations between our standing hypothesis on $(X,\Sigma)$ and certain fillability conditions of $Y$. This will be used in the proof of Corollary \ref{cor:invertibility}.

In this section, let $(\Sigma,\omega_\Sigma)$ be a closed connected symplectic manifold such that $[K\omega_\Sigma]$ is integral for some $K>0$ as before. But we do not assume the presence of an ambient space $(X,\omega)$ where $\Sigma$ embeds into. We take any integral lift $[K\omega_\Sigma]_\Z\in\H^2(\Sigma;\Z)$ of $[K\omega_\Sigma]$. We keep writing $\pi_Y:Y\to\Sigma$ for an oriented circle bundle with Euler class $e_Y=-[K\omega_\Sigma]_\Z$ and $\alpha$ for a contact form on $Y$ such that $d\alpha=\pi_Y^*\omega_\Sigma$. We consider the following conditions on $(Y, \alpha)$.
\begin{enumerate}
    \item[(a1)] For $\xi:=\ker \alpha$, the first Chern class $c_1^\xi$ of $\xi\to Y$ vanishes on $\pi_2(Y)$, and $\mu_\CZ(\gamma)\geq 4$ for every Reeb orbit $\gamma$ of $\alpha$ contractible in $Y$.
    
	\item[(a2)] There exists a Liouville filling $(W,\lambda)$ of $(Y,\alpha)$ such that $c_1^{TW}$ vanishes on $\pi_2(W)$, and  $\mu_\CZ(\gamma)\geq 2$ for every Reeb orbit $\gamma$ of $\alpha$ contractible in $W$.
\end{enumerate}

\begin{remark}
Following \cite[Section 3.3]{Ueb19} and \cite[Section 9.5]{CO1}, the pair $(Y,\alpha)$ is called {\it index-positive} if either (a1) is fulfilled or (a2) with $\mu_\CZ(\gamma)\geq 2$ replaced by $\mu_\CZ(\gamma)>2$.
Due to the periodicity of the Reeb flow of $\alpha$ in our setting, periodic Reeb orbits come in $Y$-family. To break this symmetry, we may consider a perturbation $\alpha_{f_\Sigma}:=((1+f_\Sigma)\circ\pi_Y)\alpha$ of $\alpha$ by a $C^2$-small  Morse function $f_\Sigma:\Sigma\to\mathbb{R}$. Then periodic Reeb orbits  of $\alpha_{f_\Sigma}$ with period less than some threshold (depending on $f_\Sigma$) correspond to periodic Reeb orbits of $\alpha$ over critical points of $f_\Sigma$. If $\gamma_p$ denotes a periodic Reeb orbit of  $\alpha_{f_\Sigma}$ corresponding to a periodic Reeb orbit $\gamma$ of $\alpha$ over $p\in\Crit f_\Sigma$, then the Conley-Zehnder index of $\gamma_p$ equals
\begin{equation}\label{eq:mu_perturb}
\mu_{\mathrm{CZ}}(\gamma_p)=\mu_{\mathrm{CZ}}(\gamma)-\frac{1}{2}\dim\Sigma +\ind_{f_\Sigma}(p).	
\end{equation}
See \cite[Section 2.2]{Bou}. This shows that our index conditions indeed agree with those in \cite{Ueb19,CO1}.
\end{remark}

The geometric setting of our paper is one of the following two conditions on $X$ and $\Sigma$.
\begin{enumerate}
	\item[(b1)] There exists $\tau_\Sigma>0$ such that $c_1^{T\Sigma}=\tau_\Sigma\omega_\Sigma$ on $\pi_2(\Sigma)$, and the minimal Chern number of $c_1^{T\Sigma}$ is at least $2$, i.e. 
	\[
	\inf\{k>0\mid \exists A\in\pi_2(\Sigma) \;\textrm{ such that }\; c_1^{T\Sigma}(A)=k\}\geq 2.
	\]

    \item[(b2)] Our standing hypothesis is satisfied, namely $(\Sigma,\omega_\Sigma)$ embeds as a codimension two symplectic submanifold into a closed symplectic manifold $(X,\omega)$ such that $c_1^{TX}=\tau_X\omega$ on $\pi_2(X)$ with $\tau_X>K$ and $[K\omega]$ admits an integral lift $[K\omega]_\Z$ satisfying $\PD([\Sigma])=[K\omega]_\Z$ and $[K\omega]_\Z|_\Sigma=[K\omega_\Sigma]_\Z$.  
\end{enumerate}
The infimum of the empty set is $+\infty$ in (b1), that is, the minimal Chern number is $+\infty$ when $c_1^{T\Sigma}$ vanishes on $\pi_2(\Sigma)$. 

\begin{lemma}\label{lem:ind_posit1}
	The conditions (a1) and (b1) are equivalent. 
\end{lemma}
\begin{proof}
	We first assume (a1) and prove (b1). By \eqref{homotopy_les}, we have
	\[
	\ker\omega_\Sigma|_{\pi_2(\Sigma)}=\ker e_Y|_{\pi_2(\Sigma)}=(\pi_Y)_*\pi_2(Y).
	\]
	Since the pullback bundle $\pi_Y^*(T\Sigma)$ is isomorphic to $\xi$, using condition (a1) we get
	\[
	0=c_1^\xi|_{\pi_2(Y)}=(\pi_Y)^*c_1^{T\Sigma}|_{\pi_2(Y)}=c_1^{T\Sigma}|_{(\pi_Y)_*\pi_2(Y)} .
	\] 
	Therefore, if $c_1^{T\Sigma}(A)\neq 0$ for some $A\in\pi_2(\Sigma)$, then $A\notin(\pi_Y)_*\pi_2(Y)$ and thus $\omega_\Sigma(A)\neq0$. This implies that there exists $\tau_\Sigma\in\mathbb{R}$ such that 
	\begin{equation}\label{eq:tau_Sigma}
		c_1^{T\Sigma}(A)=\tau_\Sigma\omega_\Sigma(A)\qquad  \forall A\in \pi_2(\Sigma),	
	\end{equation}
	see \cite[proof of Lemma 1.1]{HS95} or \cite[Remark 2.6.(i)]{AK23}. Suppose that $\omega_\Sigma$ vanishes on $\pi_2(\Sigma)$. This corresponds to the case that all periodic Reeb orbits are not contractible in $Y$, see \eqref{homotopy_les}. Then $c_1^{T\Sigma}$ also vanishes on $\pi_2(\Sigma)$ and the minimal Chern number is $+\infty$. In this case, $\tau_\Sigma$ can be any real number. 
	
	Next suppose that $\omega_\Sigma$ does not vanish on $\pi_2(\Sigma)$. We take arbitrary $A\in\pi_2(\Sigma)$. We may assume $K\omega_\Sigma(A)=-e_Y(A)>0$ by changing $A$ to $-A$ if necessary. For any $p\in\Sigma$, let $s:(D^2,\partial D^2)\to (\Sigma,p)$  be a continuous map representing $A$. Then it lifts to $\tilde s:(D^2,\partial D^2)\to (Y,Y_p)$ such that $\m(\tilde s|_{\partial D^2})=-e_Y(A)$. Let $\gamma$ be any periodic Reeb orbit with $\m(\gamma)=-e_Y(A)$. Then  \eqref{RS-index2} and condition (a1) imply $c_1^{T\Sigma}(A)=\frac{1}{2}\mu_\CZ(\gamma)\geq 2$. 
	Moreover, since $K\omega_\Sigma(A)=-e_Y(A)>0$, we have $\omega_\Sigma(A)>0$ and in turn $\tau_\Sigma>0$ by \eqref{eq:tau_Sigma}.

	Now we assume (b1) and prove (a1). Let $\gamma$ be a contractible Reeb orbit on $Y$. By \eqref{RS-index2} and condition (b1), $\mu_\CZ(\gamma)=2c_1^{T\Sigma}(A)\geq 4$ where $A\in\pi_2(\Sigma)$ is the class induced by a capping disk of $\gamma$ contained in $Y$. Furthermore, using condition (b1), we deduce
	\[
	c_1^\xi|_{\pi_2(Y)}=(\pi_Y)^*c_1^{T\Sigma}|_{\pi_2(Y)}=c_1^{T\Sigma}|_{(\pi_Y)_*\pi_2(Y)} =\tau_\Sigma \omega_\Sigma|_{(\pi_Y)_*\pi_2(Y)}=0
	\]
	where the last equality follows from that $\omega_\Sigma|_{(\pi_Y)_*\pi_2(Y)}=-\frac{1}{K}e_Y|_{(\pi_Y)_*\pi_2(Y)}=0$ by \eqref{homotopy_les}. 
	\end{proof}

According to Section 2, the following is true.
\begin{lemma}\label{lem:b2_a2}
    The condition (b2) implies the condition (a2).
\end{lemma}
\begin{proof}
    By Section \ref{sec:setting}, we obtain a Liouville domain $(W, \lambda)$ from $(X, \Sigma)$ satisfying the condition (b2).
    Let $\gamma$ be a periodic Reeb orbit contractible in $W$. Since $\frac{\tau_X}{K}\m(\gamma)=c_1^{TX}(B)$ is an integer, where $B\in\pi_2(X)$ is as in \eqref{RS-index}, the hypothesis $\tau_X>K$ implies $\mu_\CZ(\gamma)=\frac{2(\tau_X-K)}{K}\m(\gamma)\geq2$. 
	Moreover, for any continuous map $w:S^2\to W$ which can be viewed also as a map to $X\setminus\Sigma$, we have $c_1^{TW}([w])=c_1^{TX}([w])=\tau_X\omega([w])=0$ since $\omega$ is exact on $X\setminus\Sigma$. 
\end{proof}
To prove a partial converse, we consider two conditions on $(Y, \alpha)$ stronger than (a2).
\begin{enumerate}
    \item[(a2-1)] A Liouville filling $(W,\lambda)$ in (a2) exists and satisfies $\pi_1 (W, Y) = 0$.
    \item[(a2-2)] A Liouville filling $(W,\lambda)$ in (a2) exists, and  $c_1^{TW} \in \H^2(W;\mathbb Z)$  is torsion.
\end{enumerate}

\begin{remark}\label{rem:cut}
Suppose that there is a Liouville filling $(W,\lambda)$ of $(Y,\alpha)$. Then there is a Hamiltonian $S^1$-action near $\partial W$ which restricts to the $S^1$-action on $\partial W=Y$. The symplectic cut with respect to this $S^1$-action is a closed symplectic manifold $(X,\omega)$ such that $\Sigma$ embeds symplectically into $X$ and $(X\setminus\Sigma,\omega|_{X\setminus\Sigma})$ is symplectomorphic to the interior of $(W,d\lambda)$, see \cite{Ler95}. Moreover, the first Chern class of the symplectic normal bundle of $\Sigma$ in $X$ coincides with $-e_Y=[K\omega_\Sigma]_\Z$, and accordingly there is an integral lift $[K\omega]_\Z\in \H^2(X;\Z)$ such that $\PD([\Sigma])=[K\omega]_\Z$ and $[K\omega]_\Z|_{\Sigma}=[K\omega_\Sigma]_\Z$.
\end{remark}

\begin{lemma} \label{lem:top}
Let $W$ be as in the condition (a2), and let $X$ be the symplectic cut of $W$  as recalled in Remark \ref{rem:cut}. For any  $B \in \pi_2 (X)$ with $\omega(B) = 0$, the following hold.
    \begin{enumerate}[(i)]
        \item There is a continuous map $w: S \to X$ from a closed orientable surface $S$ such that $w$ does not intersect $\Sigma$ and the homology class $w_*([S]	)$ coincides with the image of B in $\H_2 (X; \mathbb Z)$ under the Hurewicz homomorphism.

        \item If furthermore (a2-1) holds, then we can take the surface $S$ in (i) as $S^2$.
    \end{enumerate}
\end{lemma}

\begin{proof}

    Let $v : S^2 \to X$ be a smooth map such that it represents $B \in \pi_2 (X)$ satisfying $\omega (B) = 0$ and  intersects $\Sigma$ transversely at $\{z_1,\dots,z_j\}$. Since $v\cdot \Sigma=K\omega(B)=0$, we know that $j$ is even.      
    For the sake of simplicity, we only give a proof for the case $j=2$.

   We first prove (i). 
   Let  $U$ be a sufficiently small open tubular neighborhood of $\Sigma$ in $X$. Up to a homeomorphism, we may identify $X \setminus U$ with $W$ and $\partial(X \setminus U)$ with $Y$. We may assume that there exist disjoint open disk neighborhoods $D_i \subset S^2$ of $z_i$, for $i=1,2$, such that $v^{-1}(U)=D_1\cup D_2$ and each $v(\partial D_i)$ is contained in the fiber $Y_{v(z_i)}$ over $v(z_i)$. Next we take a path $P$ in $\Sigma$ joining $v(z_1)$ and $v(z_2)$ and denote by $Y_P:=\pi_Y^{-1}(P)$ the cylinder over $P$. Note that $Y_P$ capped with $v(D_1)$ and $v(D_2)$ is null-homologous in $X$. Gluing $v|_{S^2\setminus(D_1\cup D_2)}$ and this cylinder, we obtain a map $w$ from the torus $T^2$ to $X$ satisfying the properties required in statement (i). In the general case, we would obtain a map $w$ defined on the closed orientable surface of genus $j/2$. 

To prove (ii), we assume $\pi_1(W,Y)=0$. Then, for a path $L$ in $S^2\setminus(D_1\cup D_2)$ from a point in $\partial D_1$ to a point in $\partial D_2$, $v(L)$ is homotopic in $W$ to a path $Q$ in $Y$ relative to the endpoints. We take the cylinder $Y_{\pi_Y(Q)}$ containing $Q$, glue this with $v|_{S^2\setminus(D_1\cup D_2)}$, and obtain $u:T^2\to X$ meeting all the properties in statement (i). In addition, due to the homotopy between $v(L)$ and $Q$, we may assume that $u|_{S^1\times\{0\}}$ is contractible in $W$ and thus has a capping disk $D$ in $W$. Let $S^2=\{x^2+y^2+z^2=1\}$ be the unit sphere in $\R^3$, and let $\varphi:S^2\cap \{-\frac{1}{2}<z<\frac{1}{2}\}\to T^2\setminus (S^1\times \{0\})$ be a homeomorphism. Then statement (ii) follows with a continuous map $w:S^2\to W$ satisfying 
\[
w|_{\{-\frac{1}{2}<z<\frac{1}{2}\}} =u\circ\varphi,\qquad  w(\{z\geq 1/2\})=w(\{z\leq -1/2\})= D.
\]
\end{proof}

\begin{lemma} \label{lem:fillability}
 The condition (a2-1) or (a2-2) implies the condition (b2).
\end{lemma}

\begin{proof}
	Let $B\in\pi_2(X)$ satisfy $\omega(B)=0$. 
    We take $w:S\to X$ as in Lemma \ref{lem:top}. As $w$ does not intersect $\Sigma$, we may view $w$ as a map into $W$. In the case of (a2-2), namely $c_1^{TW}$ is torsion, we have 
    \begin{equation*}
        c_1^{TX} (B)  =  c_1^{TX} ([w]) = c_1^{TW} ([w]) = 0.
    \end{equation*}
    The same holds in the case of (a2-1) since $c_1^{TW}|_{\pi_2(W)}=0$ and we can take $S=S^2$. 
    As mentioned in the proof of Lemma \ref{lem:ind_posit1}, this implies that there exists $\tau_X\in\mathbb{R}$ such that
	$c_1^{TX}=\tau_X\omega$ on $\pi_2(X)$ (even in $\H_2(X;\Z)$ in the case of (a2-2)). We next verify $\tau_X>K$. If $\omega$ vanishes on ${\pi_2(X)}$, so does $c_1^{TX}$. Thus, $\tau_X$ can be an arbitrary real number, and we take any $\tau_X>K$. Suppose that $\omega$ is nonzero on ${\pi_2(X)}$. Then, by Proposition \ref{contractible}.(b), there is a periodic Reeb orbit $\gamma$ contractible in $W$ whose index equals $\mu_\CZ(\gamma)=\frac{2(\tau_X-K)}{K}\m(\gamma)$ as computed in \eqref{RS-index}. From the hypothesis $\mu_\CZ(\gamma)\geq 2$, we conclude $\tau_X>K$.
    It remains to prove that $[K\omega]$ admits an integral lift $[K\omega]_\Z$ satisfying $\PD([\Sigma])=[K\omega]_\Z$ and $[K\omega]_\Z|_\Sigma=[K\omega_\Sigma]_\Z$.  
    This follows from Remark \ref{rem:cut}.
\end{proof}

\begin{remark}\label{rem:indep_filling}
Let us assume that (a2) with $\mu_\CZ(\gamma)\geq 2$ replaced by $\mu_\CZ(\gamma)>2$ holds. We call this condition (a2\textsuperscript{+}). For example, this is the case if $\pi_1(Y)\to\pi_1(W)$ induced by the inclusion is injective and the minimal Chern number of $\Sigma$ is greater than one. In this case, the map $\delta^{X,\Sigma}$ in \eqref{eq:delta_c} vanishes. Indeed, if $\delta^{X,\Sigma}\neq0$, then there is $B\in\pi_2(X)$ satisfying $c_1^{TX}(B)-K\omega(B)=1$ which leads to a periodic Reeb orbit $\gamma$ with $\mu_\mathrm{CZ}(\gamma)=2$, see \eqref{RS-index}.
Therefore, the boundary operators of $\SC_*(\partial W)$ count only Floer cylinders in $\R\times Y$, and the Rabinowitz Floer homology is independent of Liouville filling  in the following sense. Let $W_1$ and $W_2$ be  two Liouville fillings of $(Y,\alpha)$ such that both satisfy the condition (a2\textsuperscript{+}) and $\pi_1(Y)\to\pi_1(W_i)$ induced by the inclusion is injective for both $i=1,2$. Then the split chain complex for $W_1$ is identical to that for $W_2$, and accordingly $\SH_*(\partial W_1)=\SH_*(\partial W_2)$.
 If we take all periodic orbits into account, as opposed to contractible ones, we can drop injectivity of $\pi_1(Y)\to\pi_1(W_i)$ at the cost of having independence of Liouville filling up to grading-shift, see also \cite[Remark 3.8]{Ueb19}. 
\end{remark}

\subsection{Proofs of corollaries}\label{sec:pf_cor}

\subsubsection*{Proof of Corollary \ref{cor:invertibility}} 

As in the proof of Lemma \ref{lem:fillability}, our assumption implies $c_1^{TX}  = \tau_X \omega $ on $\pi_2(X)$ for some $\tau_X \in \mathbb R$. Next we use the hypothesis $\SH_*(\partial W)=0$ to deduce $\tau_X>K$. Since $\SH_* (\partial W) = 0$, the unit of $\SH_*(\partial W)$ or of the symplectic homology of $W$, which has degree $\frac{1}{2} \dim X$, is zero, see \cite[Section 10]{CO1}. Thus there is $\tilde{p}_k \in \Crit f_{Y^H_k}$ for some $H\in\mathcal{H}$ with $k>0$ whose index $\mu (\tilde{p}_k) = \ind_{f_Y} (\tilde{p}) - \frac{1}{2} \dim \Sigma + 2\frac{(\tau_X-K)}{K}k$ equals $\frac{1}{2} \dim X + 1$, see \eqref{grading}.  Since $\ind_{f_Y} (\tilde{p}) - \frac{1}{2} \dim \Sigma \leq \frac{1}{2} \dim X$, we conclude $\tau_X - K > 0$.

Therefore we can apply Theorem \ref{thm:gysin}, and the hypothesis $\SH_*(\partial W)=0$ implies that the map $\delta_*:\H_{*+2}(\Sigma;\Lambda)\to \H_{*}(\Sigma;\Lambda)$ is an isomorphism.
In particular, the induced map on the free part $\H_*^{\mathrm{free}}(\Sigma; \mathbb Z)\otimes_\Z\Lambda$, which we also denote by $\delta_*$, is an isomorphism.
To prove (a), we assume for contradiction that $[K\omega]$ is not primitive. 
 Then, there exist an integer $d \geq 2$ with $K\omega(\H_2(X; \mathbb Z)) = d \mathbb Z$ and a class $\mu \in \H^2_{\mathrm{free}}(X; \mathbb Z)$ such that $[K \omega] = d \mu$.
Here we abuse notation and $[K\omega]$ also refers to the unique class in $\H^2_{\mathrm{free}} (X; \mathbb Z)$ corresponding to $[K\omega] \in \H^2 (X; \mathbb R)$.
We claim that the map $\delta_*$ on $\H_*^{\mathrm{free}}(\Sigma; \mathbb Z)\otimes_\Z\Lambda$ 
is divisible by $d$, namely $\delta_*$ is the composition of a $\Lambda$-linear map and the scalar multiplication by $d$. 
By Lemma \ref{lemma:quantum_cap}, $\delta^\Sigma_*$ is the quantum cap product with $-[K\omega_\Sigma]=-d\mu|_{\Sigma}$ on  $\H_*^{\mathrm{free}}(\Sigma; \mathbb Z)\otimes_\Z\Lambda$. Thus it is divisible by $d$. Moreover \eqref{eq:delta_c} yields that $\delta^{X, \Sigma}_*$ is also divisible by $d$. This proves the claim. 
Since $d\geq2$ and $\H^{\mathrm{free}}_*(\Sigma;\Z)$ is nonzero, $\delta_*$ cannot be an isomorphism. This contradiction proves statement (a).

According to the Morse theoretic description of the quantum product, the map $\delta^\Sigma_*$ modulo torsion corresponds to the quantum cup product with $-[K\omega_\Sigma]$ on $\H^*_\mathrm{free}(\Sigma;\Z)\otimes_\Z\Lambda$, see \cite[Chapter 12]{MS}. In particular, the quantum cup product with $-[K\omega_\Sigma]$ on $\QH^*(\Sigma;\Z)=\H^*_\mathrm{free}(\Sigma;\Z)\otimes_\Z\Lambda_\Sigma$ is an isomorphism. 
	This proves statement (b).
	\qed

\subsubsection*{Proof of Corollary \ref{cor:dehn_twist}}

Let $\tau\in \mathrm{Symp}_c(W)$ be a right-handed fibered Dehn twist:  
\begin{equation*}
    \tau := \begin{cases}
        \phi^{\beta(r)}_R &\text{on $(-\infty,0]\times Y$},\\[.5ex]
        \mathrm{Id} &\text{on $W\setminus((-\infty,0]\times Y),$}
    \end{cases}
\end{equation*}
where we identify $(-\infty,0]\times Y$ with its image in $W$ by the Liouville flow, and $\beta:(-\infty,0]\to\R$ is a smooth function which has compact support and is constant $-\frac{1}{K}$ near $0$.

	Following \cite{BG07}, see also \cite[Theorem 6.3]{CDvK14}, we consider the closed contact manifold $\mathrm{OB}(W,\tau)$ obtained by gluing the mapping torus of $\tau$ and $\partial W\times D^2$, which  is contactomorphic to the prequantization bundle over $(X,[K\omega]_\Z)$. Suppose that $\tau$ is symplectically isotopic to the identity $\mathrm{id}_W$ relative to $\partial W$. Then $\mathrm{OB}(W,\tau)$ is contactomorphic to $\mathrm{OB}(W,\mathrm{id}_W)$, and therefore has a subcritical Weinstein filling $W\times D^2$, see \cite[Theorem 2.12]{CDvK14}. Since $c_1^{TW}$ is zero on $\pi_2(W)$ by Lemma \ref{lem:b2_a2}, $c_1^{T(W\times D^2)}$ is zero on $\pi_2(W\times D^2)$. Moreover since $[K\omega]_\Z$ is integral and $\tau_X>K$, the minimal Chern number of $X$ is greater than one. Therefore, Theorem \ref{thm:gysin} yields the long exact sequence 
	\[
    \cdots \longrightarrow \H_{*+2}(X;\Lambda_X) \stackrel{\delta_*} {\longrightarrow} \H_* (X;\Lambda_X) \longrightarrow \SH_{*-\frac{\dim X}{2}+1} (\partial(W\times D^2)) \longrightarrow \cdots,
    \]
    where $\Lambda_X$ is defined as in \eqref{eq:Novikov_Sigma} with $\Sigma$ replaced by $X$ and $\delta_*$ is the quantum cap product with $-[K\omega]_\Z$. 
	Since $W\times D^2$ is subcritical Weinstein, $\SH_*(\partial(W\times D^2))=0$ by \cite{Cie}. One can also deduce this using $\SH_*(D^2)=0$ and the K\"unneth formula in \cite{Oan06}. Hence, due to Corollary \ref{cor:invertibility}.(b), $[K\omega]$ is invertible in $\QH^*(X;\Z)$. 
\qed

\subsubsection*{Proof of Corollary \ref{cor:gysin_symplectization}}

Let us assume the condition (b1), which is equivalent to (a1) as observed in Lemma \ref{lem:ind_posit1}.
It is well known that, under the condition (a1), Rabinowitz Floer homology can be defined without reference to symplectic fillings.
We take type III periodic orbits of $\bigvee$-shaped Hamiltonians $H:\R\times Y\to\R$ contractible in $Y$ as generators and count rigid elements of
\begin{equation}\label{eq:symplectization_moduli}
\mathcal{M}_{N=0}(\tilde{q}_{k_-}, \tilde{p}_{k_+}) / \mathbb{R}, \qquad \mathcal{M}_{N, \ell = 0}(\tilde{q}_{k_-}, \tilde{p}_{k_+}; \mathbf{A};H;J_Y) / \mathbb{R}^N	
\end{equation}
to define boundary operators as in \eqref{differential}. An analogous argument to the proof of Proposition \ref{continuation_transversality}.(b) implies that the latter moduli space in \eqref{eq:symplectization_moduli} consists of simple elements provided $\mu (\tilde{p}_{k_+}) - \mu (\tilde{q}_{k_-}) \leq 2$. Moreover, due to the condition (a1), we can apply \cite[Corollary 3.7]{Ueb19} to ensure that elements of $\mathcal{M}_{N, \ell = 0}(\tilde{q}_{k_-}, \tilde{p}_{k_+}; \mathbf{A};H;J_Y)$ with $\mu (\tilde{p}_{k_+}) - \mu (\tilde{q}_{k_-}) \leq 2$ do not escape to the negative end of the symplectization $\R\times Y$. 

Alternatively one can argue as follows. We assume $\mu (\tilde{p}_{k_+}) - \mu (\tilde{q}_{k_-}) \leq 2$ and consider sequences 
\[
\begin{split}
\big\{((\mathbf{v}^n,\emptyset),\mathbf{t}^n)\big\}_{n\in\N}\subset &\, \mathcal{M}^*_{N, \ell = 0}(\tilde{q}_{k_-}, \tilde{p}_{k_+}; \mathbf{A};H;J_Y),\\[.5ex]
\big\{((\mathbf{w}^n,\emptyset),\mathbf{t}^n)= \Pi((\mathbf{v}^n,\emptyset),\mathbf{t}^n)\big\}_{n\in\N} \subset & \,\mathcal{N}^*_{N, \ell = 0} (q, p; \mathbf{A};J_\Sigma).	
\end{split}
\]
Note that $\ell=0$ implies that all entries of $\mathbf{A}$ are nonzero, see after \eqref{multiplicity}. 
Assume on the contrary that $\mathbf{v}^n$  escapes to the negative end of $\R\times Y$ as $n$ goes to infinity. This leads to bubbling-off of holomorphic spheres in the limit of $((\mathbf{w}^n,\emptyset),\mathbf{t}^n)$. For simplicity, suppose that there is exactly one holomorphic sphere bubble, namely $((\mathbf{w}^n,\emptyset),\mathbf{t}^n)$ converges to $((\mathbf{w}^\infty=(w_1^\infty,\dots,w_N^\infty),\emptyset),\mathbf{t}^\infty)$ and a holomorphic sphere bubble $w'$ attached to $w_N^\infty$ as $n$ goes to infinity. Then $w'$ takes index at least twice the minimal Chern number away from the index of $((\mathbf{w}^n,\emptyset),\mathbf{t}^n)$, which is equal to $\mu (\tilde{p}_{k_+}) - \mu (\tilde{q}_{k_-})-i(\tilde{p})+i(\tilde{q})+N-1\leq N+2$, see Proposition \ref{transversality_chain}.(a) and \eqref{multiplicity}. Therefore, by the condition (b1) which is equivalent to the condition (a1), the index of $((\mathbf{w}^\infty,\emptyset),\mathbf{t}^\infty)$ is at most $N-2$. This is absurd for any $N\in\N$ since $((\mathbf{w}^\infty,\emptyset),\mathbf{t}^\infty)$ carries a free $\R\times S^1$-action on each $w_j^\infty$ for $j=1,\dots,N-1$.

Similarly, one can prove simpleness and compactness for Floer continuation cylinders in $\R\times Y$ connecting $\tilde{p}^H_{k_+}$ and $\tilde{q}^L_{k_-}$ with $\mu(\tilde{p}^H_{k_+}) - \mu(\tilde{q}^L_{k_-})\leq 1$. Hence, we can define  the Rabinowitz Floer homology for $\R\times Y$ and denote it by $\SH_*(Y)$. 
Furthermore, the construction in Section \ref{sec:Gysin_complement} carries over verbatim to this setting, and $\SH_*(Y)$ fits into the long exact sequence
\[
    \cdots \longrightarrow \H_{*+2}(\Sigma;\Lambda_{\Sigma}) \stackrel{\delta^\Sigma_*} {\longrightarrow} \H_* (\Sigma;\Lambda_{\Sigma}) \longrightarrow \SH_{*-\frac{\dim \Sigma}{2}+1} (Y) \longrightarrow  \H_{*+1} (\Sigma;\Lambda_{\Sigma}) \stackrel{\delta^\Sigma_*}\longrightarrow \cdots.
\]
Moreover, the connecting map $\delta_*$ coincides with the quantum cap product with the Euler class $e_Y$ since we only count rigid elements in $\mathbb{R} \times Y$.

If there is no nonconstant $J_\Sigma$-holomorphic sphere in $\Sigma$, we still have the above exact sequence and it recovers the ordinary Gysin sequence \eqref{eq:Morse_LES} with coefficients in $\Lambda_{\Sigma}$. This is the case for example if $\omega_\Sigma$ vanishes on $\pi_2(\Sigma)$ or $c_1^{T\Sigma}(A)\leq 2-\frac{1}{2}\dim \Sigma$ for every $A\in\pi_2(\Sigma)$ with $\omega_\Sigma(A)>0$. This completes the proof of Corollary \ref{cor:gysin_symplectization}. 
\qed

\subsubsection*{Proof of Corollary \ref{cor:orderability}}

	As in the proof of Corollary \ref{cor:invertibility}.(a), since $[K\omega_\Sigma]$ is not a primitive class, $\SH_*(Y)$ is nonzero. Hence, $\widetilde{\Cont}_0(Y,\xi)$ is orderable due to  \cite{AM18}. 
	
	The existence of a translated point follows from $\SH_*(Y)\neq0$ as proved in \cite{MN18}. Although they dealt with hypertight contact manifolds, which corresponds to the case $\omega_\Sigma|_{\pi_2(\Sigma)}=0$ in our setting, their proof  remains valid because  
 our hypothesis on $(\Sigma,\omega_\Sigma)$ ensures necessary compactness properties of Floer cylinders in $\R\times Y$ as shown in the proof of Corollary \ref{cor:gysin_symplectization}. 
	
	Rabinowitz Floer homology used in \cite{AM18,MN18} is defined with the Rabinowitz action functional. Thus, to be exact, we need an isomorphism between the homology defined with the Rabinowitz action functional on $\R\times Y$ and $\SH_*(Y)$. Such an isomorphism is constructed in \cite{CFO} in the Liouville fillable case, and proofs in \cite{CFO} readily carry over to our situation again thanks to our hypothesis on $(\Sigma,\omega_\Sigma)$.  
\qed

\section{Some computations}\label{sec:computations}
To ease the notation, we write 
\[
\SH(\partial(X\setminus \Sigma))= \SH(\partial W), \quad\delta_a=\delta_a^\Sigma,\quad \delta_b=\delta_b^\Sigma,\quad\delta_c=\delta^{X,\Sigma}
\]
where $X, \Sigma,$ and  $W$ are as in Section \ref{sec:setting}. 
All Liouville domains $W$ in this section are Weinstein satisfying $\pi_1 (W, Y) = 0$, in particular $m_W = m_X$ by Proposition \ref{prop:m_W}.(b).

\subsection{Unit cotangent bundles of spheres} \label{sec:spheres}
The goal of this section is to compute the Rabinowitz Floer homology of the unit cotangent bundle of a sphere using the Floer Gysin sequence in Theorem \ref{thm:gysin} and Corollary \ref{cor:gysin_symplectization} and to see that the computation indeed agrees with earlier results in \cite{AS09,CFO}. Let $S^{n+1}\subset \mathbb{R}^{n+2}$ be the Euclidean sphere of radius $\frac{1}{2\pi}$ so that all simple closed geodesics have arc-length 1. Let $\lambda$ denote the tautological 1-form on the cotangent bundle $T^*S^{n+1}$. Then the unit cotangent bundle $S^*S^{n+1}$ is a contact manifold with the 1-form $\alpha:= \lambda|_{S^*S^{n+1}}$. The flow of the Reeb vector field of $\alpha$ on $S^*S^{n+1}$, which is the cogeodesic flow, defines a free $S^1$-action, and the quotient space is the Grassmannian $\widetilde{\mathrm{Gr}}(2,\mathbb{R}^{n+2})$ of oriented 2-planes in $\mathbb{R}^{n+2}$ equipped with the symplectic form induced by $d\alpha$. This Grassmannian is symplectomorphic to the smooth quadric 
\[
Q^n:=\{[z_0:\dots:z_{n+1}]\in\mathbb{CP}^{n+1}\mid z_0^2+\dots +z_{n+1}^2=0\}
\]
with the symplectic form $\omega:=\omega_\mathrm{FS}|_{Q^n}$ given by the restriction of the Fubini-Study form $\omega_\mathrm{FS}$ on $\mathbb{CP}^{n+1}$, which is normalized by $\int_{\mathbb{CP}^{1}}\omega_\mathrm{FS}=1$, see \cite[Exercise III.15]{Aud04} and \cite[Exercise 6.3.7]{MS17}. Thus we obtain the oriented circle bundle 
\[
\wp:S^*S^{n+1}\longrightarrow \widetilde{\mathrm{Gr}}(2,\mathbb{R}^{n+2})\cong Q^n,\qquad d\alpha=\wp^*\omega,
\]
with the Euler class $-[\omega]\in\H^2(Q^n;\mathbb{Z})$. Applying the homotopy exact sequence in \eqref{homotopy_les} and using the fact $\pi_1(S^*S^2)\cong\mathbb{Z}/2$ and $\pi_1(S^*S^{n+1})=0$ for $n\geq2$, we deduce
\begin{equation}\label{cotangent_c1}
\omega(\pi_2(Q^1))=2\mathbb{Z}, \qquad \omega(\pi_2(Q^n))=\mathbb{Z}, \quad n\geq 2.	
\end{equation}
The singular homology of $Q^n\cong\widetilde{\mathrm{Gr}}(2,\mathbb{R}^{n+2})$ is as follows. For $n=2m+1$ odd, 
\[
\H_i(Q^{2m+1};\mathbb{Z})\cong \begin{cases}
	\mathbb{Z}, & i\in\{0,2,\dots,2(2m),2(2m+1)\},\\[0.5ex]
	0 & \textrm{otherwise},
\end{cases}
\]
and, for $n=2m$ even,
\[
\H_i(Q^{2m};\mathbb{Z})\cong \begin{cases}
	\mathbb{Z}, & i\in\{0,2,\dots,2(2m-1),4m\} \setminus \{2m\}, \\[0.5ex]
	\mathbb{Z}\oplus\mathbb{Z}, & i=2m, \\[0.5ex]
	0 & \textrm{otherwise}.
\end{cases}
\]
Since the formal variable $T$ in $\Lambda=\mathbb{Z}[T,T^{-1}]$ has even degree, $\H_*(Q^{n};\Lambda)$ vanishes for every odd degree. Hence the Floer Gysin sequence in Theorem \ref{thm:gysin} (or Corollary \ref{cor:gysin_symplectization} for $n \geq 2$) splits into
\begin{equation}\label{quadric_seq}
	  0\to  \SH_{i+2}(S^*S^{n+1})\rightarrow \H_{i+2+n}(Q^n;\Lambda)\xrightarrow{\delta_*} \H_{i+n}(Q^n;\Lambda) \rightarrow \SH_{i+1}(S^*S^{n+1}) \rightarrow 0
\end{equation} 
for every $i\in 2\mathbb{Z}+n$. For $n\geq 2$, the contact form $\alpha$ on $S^*S^{n+1}$ is index-positive (condition (b1) in Section \ref{sec:Gysin_symplectization}) and $\pi_1(S^*S^{n+1})$ is trivial. Thus $\SH_{*}(S^*S^{n+1})$ makes sense without reference to Liouville fillings, see Corollary \ref{cor:gysin_symplectization} and Remark \ref{rem:indep_filling}.  For $n=1$, we need a Liouville filling of $S^*S^2$, and here we mean $\SH_{*}(S^*S^{2}) := \SH_{*}(\partial D^*S^{2})$, where $D^*S^2$ is the unit disk cotangent bundle over $S^2$. As explained below, $D^*S^2$ can be seen as the complement of the diagonal in $\mathbb{CP}^1 \times \mathbb{CP}^1$, and thus the assumption of Theorem \ref{thm:gysin} is fulfilled. 
To compute $\SH_{*}(S^*S^{n+1})$ using \eqref{quadric_seq}, we will investigate the map $\delta_*$.
We deal with the cases $n=1$ and $n\geq2$ separately. 

\subsubsection*{Case $n=1$: $S^*S^2\to Q^1\cong \mathbb{CP}^1$}
We note that $(Q^1,\omega)$ is symplectomorphic to $(\mathbb{CP}^1,2\omega_\mathrm{FS})$. 
As in \cite[Section 10]{DL1}, we use the fact that the diagonal $\Delta_{\mathbb{CP}^1}$ in $(\mathbb{CP}^1\times \mathbb{CP}^1,\omega_\mathrm{FS}\oplus\omega_\mathrm{FS})$ with the symplectic form $\omega:=(\omega_\mathrm{FS}\oplus\omega_\mathrm{FS})|_{\Delta_{\mathbb{CP}^1}}$ is symplectomorphic to $(\mathbb{CP}^1,2\omega_\mathrm{FS})$, and $\mathbb{CP}^1\times \mathbb{CP}^1\setminus\Delta_{\mathbb{CP}^1}$ is symplectomorphic to the disk cotangent bundle $D^*S^2$ of $S^2$.
We endow $\mathbb{CP}^1\times \mathbb{CP}^1\setminus\Delta_{\mathbb{CP}^1}$ with the pull back of the canonical 1-form on $D^*S^2$ so that $\mathbb{CP}^1\times \mathbb{CP}^1\setminus\Delta_{\mathbb{CP}^1}$ and $D^*S^2$ are exact symplectomorphic.
Thus we will compute the Rabinowitz Floer homology of $\partial(\mathbb{CP}^1\times \mathbb{CP}^1\setminus\Delta_{\mathbb{CP}^1})$. In this case, 
\[
\mathrm{PD}([\Delta_{\mathbb{CP}^1}])=[\omega_\mathrm{FS}\oplus\omega_\mathrm{FS}],\quad c_1^{T(\mathbb{CP}^1\times \mathbb{CP}^1)}=2[\omega_\mathrm{FS}\oplus\omega_\mathrm{FS}],\quad e_{S^*S^2}=-[\omega],
\]
and thus $(\mathbb{CP}^1 \times \mathbb{CP}^1, \Delta_{\mathbb{CP}^1})$ meets the assumption of Theorem \ref{thm:gysin}. 
Moreover, the degree of the formal variable $T$ in $\Lambda=\mathbb{Z}[T,T^{-1}]$ equals $2$. We choose a lacunary Morse function $f:\Delta_{\mathbb{CP}^1}\to\mathbb{R}$ such that $\Crit f=\{p^0,p^1\}$ with $\ind_f(p^i)=2i$. Then the chain module $\mathrm{MC}_*(f;\Lambda)$ is isomorphic to $\mathbb{Z}\oplus\mathbb{Z}$ in every even degree, and hence the same holds for the homology $\H_*(\Delta_{\mathbb{CP}^1};\Lambda)$. More precisely, 
\begin{equation}\label{eq:Morse_basis}
\mathrm{MC}_{2i}(f;\Lambda)=\mathbb{Z}\langle p^0T^{i},p^1T^{i-1}\rangle \qquad \forall i\in\mathbb{Z}.
\end{equation}
The connecting map $\delta$ in \eqref{quadric_seq} at chain level consists of two parts. The contribution given by the quantum cap product with $e_{S^*S^2}=-[\omega]$ is 
\begin{equation*} \label{eq: Q_quantum_cap}
    \delta_a+\delta_b:p^1T^i\mapsto -2p^0T^i,\qquad p^0T^i\mapsto -2p^1T^{i-2} 
\end{equation*}
and the other contribution  given by the number of holomorphic spheres in $\mathbb{CP}^1\times \mathbb{CP}^1$ intersecting $\Delta_{\mathbb{CP}^1}$ and homologous to $\mathbb{CP}^1\times\{\mathrm{pt}\}$ or $\{\mathrm{pt}\}\times\mathbb{CP}^1$ is
\[
\delta_c: p^1T^i\mapsto 2p^1T^{i-1},\qquad p^0T^i\mapsto 2p^0T^{i-1}
\]
as computed in \cite[Section 11]{DL1}. 
Therefore 
\[
\begin{tikzcd}[row sep=1.5em,column sep=5em]
\H_{2i}(\Delta_{\mathbb{CP}^1};\Lambda) \arrow{r}{\delta_*}\arrow{d}{\cong} &  \H_{2i-2}(\Delta_{\mathbb{CP}^1};\Lambda)   \arrow{d}{\cong}   \\
\mathbb{Z}\oplus\mathbb{Z} \arrow{r}{(c,d)\mapsto 2(c-d,-c+d)} & \mathbb{Z}\oplus\mathbb{Z}
\end{tikzcd}
\]
where the vertical isomorphisms are given by \eqref{eq:Morse_basis}.
This together with the exact sequence in \eqref{quadric_seq} implies that 
\[
\SH_i(\partial D^*S^2)=\SH_i(S^*S^2)\cong \begin{cases}
	\mathbb{Z},  & i\in 2\Z+1, \\[0.5ex]
	\mathbb{Z}\oplus\mathbb{Z}_2, \quad & i\in 2\Z\,.
\end{cases}
\]

\subsubsection*{Case $n\geq2$: $S^*S^{n+1}\to Q^n$}
To describe the ring structure, we denote by $h\in \H_{2n-2}(Q^n;\mathbb{Z})$  the hyperplane class and by $\bullet$ the intersection product. For $n=2m+1$ odd, the homology of $Q^n$ is generated by 
\[
\H_{2(2m+1)-2i}(Q^{2m+1};
\mathbb{Z})=\begin{cases}
	\mathbb{Z}\langle h^{\bullet i}\rangle, & 0\leq i\leq m, \\[0.5ex]
	\mathbb{Z}\langle a\bullet h^{\bullet (i-m-1)}\rangle, & m+1\leq i \leq 2m+1, 
\end{cases}
\]
where $a$ is a generator of $\H_{2m}(Q^{2m+1};\mathbb{Z})$ satisfying $h^{\bullet (m+1)}=2a$. For $n=2m$ even, 

\[
\H_{4m-2i}(Q^{2m};\mathbb{Z})=\begin{cases}
	\mathbb{Z}\langle h^{\bullet i}\rangle, & 0\leq i\leq m-1, \\[0.5ex]
	\mathbb{Z}\langle a,b\rangle, & i=m, \\[0.5ex]
	\mathbb{Z}\langle a\bullet h^{\bullet (i-m)}\rangle, & m+1\leq i\leq 2m,
\end{cases}
\]
where generators $a$ and $b$ in degree $2m$ can be chosen to satisfy $h^{\bullet m}=a+b$ and $a\bullet h=b\bullet h$. In both odd and even dimensional cases, $h^{\bullet 0}=[Q^n]$ is the fundamental class and $a\bullet h^{\bullet m}=[\mathrm{pt}]$ is the point class, see \cite{GH94,Van99}.

The quantum product structure on ${\H}_*(Q^n;\Lambda)$ is fully revealed in \cite{Bea97}, see also \cite[Section 6.3]{BC6}. Here we only recall the necessary information, namely the relation involving the Euler class $e_{S^*S^{n+1}}$ of $S^*S^{n+1}\to Q^n$. We denote by $*$ the quantum intersection product on ${\H}_*(Q^n;\Lambda)$. Since $e_{S^*S^{n+1}}=-[\omega]$ is the negative generator in $\H^2(Q^n;\mathbb{Z})$ by \eqref{cotangent_c1}, it equals $-\mathrm{PD}(h)$ and we need to understand the map $*(-h)$, the multiplication by $-h$. The variable $T$ in $\Lambda=\mathbb{Z}[T,T^{-1}]$ has degree $2n$ since  $c_1^{TQ^n}(\pi_2(Q^n))=n\omega(\pi_2(Q^n))=n\mathbb{Z}$. Therefore, $*h$ is equal to $\bullet h$ except degrees $0$ and $2$ in $\H_*(Q^n;\mathbb{Z})$. Moreover, in both odd and even dimensional cases, it holds that
\begin{equation}\label{eq:quadric_product}
(a\bullet h^{\bullet (m-1)})*h=[\mathrm{pt}]+[Q^n]T^{-1},\qquad [\mathrm{pt}]*h=hT^{-1},	
\end{equation}
where $m=\lfloor n/2\rfloor$ as above.

We first consider the case that $n=2m+1$ is odd. Then $h^{*(m + 1)}=h^{\bullet (m + 1)}=2a$ implies that for all $j\in\mathbb{Z}$ 
\[
\begin{tikzcd}[row sep=1.5em,column sep=1.5em]
\H_{2m+2+2jn}(Q^n;\Lambda) \arrow{r}{*(-h)} \arrow{d}{\cong} &  \H_{2m+2jn}(Q^n;\Lambda)   \arrow{d}{\cong}   \\
\mathbb{Z} \arrow{r}{\times (-2)} & \mathbb{Z}
\end{tikzcd}
\]
and thus by \eqref{quadric_seq}
\[
\SH_{2jn}(S^*S^{n+1}) \cong \coker *(-h) \cong\mathbb{Z}_2 \qquad \forall j\in\mathbb{Z}.
\]
Next, \eqref{eq:quadric_product} and $[Q^n]*h=[Q^n]\bullet h=h$ yield

\[
\begin{tikzcd}[row sep=1.5em,column sep=1.5em]
\H_{2+2jn}(Q^n;\Lambda) \arrow{r}{*(-h)} \arrow{d}{\cong} &  \H_{2jn}(Q^n;\Lambda)   \arrow{d}{\cong}  \arrow{r}{*(-h)} & \H_{2jn-2}(Q^n;\Lambda)  \arrow{d}{\cong}    \\
\mathbb{Z} \arrow{r}{1\mapsto(-1,-1)} & \mathbb{Z}\oplus\mathbb{Z} \arrow{r}{(c,d)\mapsto -(c+d)} & \mathbb{Z}
\end{tikzcd}
\]
and hence by \eqref{quadric_seq} again
\[
\SH_{2jn-n+1}(S^*S^{n+1})\cong\mathbb{Z},\qquad \SH_{2jn-n}(S^*S^{n+1})\cong\mathbb{Z}\qquad \forall j\in\mathbb{Z}.
\]
Finally, in all other degrees, $*(-h):\H_i(Q^n;\Lambda)\to \H_{i-2}(Q^n;\Lambda)$ is an isomorphism, and therefore $\SH_*(S^*S^{n+1})$ is nonzero only in the above three cases by \eqref{quadric_seq}. We have computed 
\[
\SH_i(S^*S^{n+1})\cong \begin{cases}
	\mathbb{Z}_2,  & i\in 2n\mathbb{Z},\\[0.5ex]
	\mathbb{Z},  & i\in 2n\mathbb{Z}+\{n,n+1\},\\[0.5ex]
	0 & \textrm{otherwise},\
\end{cases}
\hspace{1.5cm} n+1=\textrm{even}.
\]

The other case that $n=2m$ is even can be treated similarly. Using $h^{*m}=h^{\bullet m}=a+b$ and $a\bullet h =b\bullet h=a*h=b*h$, we compute that for all $j\in\mathbb{Z}$
\[
\begin{tikzcd}[row sep=1.5em,column sep=1.5em]
\H_{2m+2+2jn}(Q^n;\Lambda) \arrow{r}{*(-h)} \arrow{d}{\cong} &  \H_{2m+2jn}(Q^n;\Lambda)   \arrow{d}{\cong} \arrow{r}{*(-h)} & H_{2m-2+2jn}(Q^n;\Lambda) \arrow{d}{\cong}   \\
\mathbb{Z} \arrow{r}{1\mapsto (-1, -1)} & \mathbb{Z}\oplus\mathbb{Z} \arrow{r}{(c,d)\mapsto -(c+d)} & \mathbb{Z}
\end{tikzcd}
\]
and thus by \eqref{quadric_seq}
\[
\SH_{2jn+1}(S^*S^{n+1})\cong\mathbb{Z}, \qquad \SH_{2jn}(S^*S^{n+1})\cong\mathbb{Z} \qquad \forall j\in\mathbb{Z}.\\
\]
As in the previous case, \eqref{eq:quadric_product} and $[Q^n]*h=[Q^n]\bullet h=h$ yield

\[
\begin{tikzcd}[row sep=1.5em,column sep=1.5em]
\H_{2+2jn}(Q^n;\Lambda) \arrow{r}{*(-h)} \arrow{d}{\cong} &  \H_{2jn}(Q^n;\Lambda)   \arrow{d}{\cong}  \arrow{r}{*(-h)} & \H_{2jn-2}(Q^n;\Lambda)  \arrow{d}{\cong}    \\
\mathbb{Z} \arrow{r}{1\mapsto(-1, -1)} & \mathbb{Z}\oplus\mathbb{Z} \arrow{r}{(c,d)\mapsto -(c+d)} & \mathbb{Z}
\end{tikzcd}
\]
and hence by \eqref{quadric_seq} again
\[
\SH_{2jn-n+1}(S^*S^{n+1})\cong\mathbb{Z},\qquad \SH_{2jn-n}(S^*S^{n+1})\cong\mathbb{Z}\qquad \forall j\in\mathbb{Z}.
\]
Finally, in all other degrees, $*(-h):\H_i(Q^n;\Lambda)\to \H_{i-2}(Q^n;\Lambda)$ is an isomorphism, and therefore $\SH_*(S^*S^{n+1})$ is nonzero only in the above four cases by \eqref{quadric_seq}. We have computed 
\[
\SH_i(S^*S^{n+1})\cong \begin{cases}
	\mathbb{Z},  & i\in n\mathbb{Z}+\{0,1\},\\[0.5ex]
	0 & \textrm{otherwise},\
\end{cases}
\qquad n+1=\textrm{odd}.
\]

\begin{remark}
    The Rabinowitz Floer homology of $S^*S^n$ for $n\geq 4$ was computed in \cite{CF09}. More generally, the Rabinowitz Floer homology of the unit cotangent bundle of a closed manifold $N$ is computed  in \cite{AS09,CFO} in terms of the homology and the cohomology of the free loop space of $N$. One can readily check that $\SH_*(S^*S^n)$ computed in this way indeed agrees with our computations.
\end{remark}

\subsection{Complements of projective hypersurfaces}
In this section, we compute the Rabinowitz Floer homology of the complements of smooth projective hypersurfaces $\Sigma^n_d$ of degree $d = 1, \dots, n+1$ in $\mathbb{CP}^{n + 1}$ using Theorem \ref{thm:gysin}. Since $c_1^{T\mathbb{CP}^{n+1}}=(n+2)[\omega_{\mathrm{FS}}]$ and $[\Sigma_d^{n}]=\PD(d[\omega_{\mathrm{FS}}])$, the assumptions of Theorem \ref{thm:gysin} are satisfied. We note that projective hypersurfaces in $\mathbb{CP}^{n+1}$ of the same degree are isomorphic due to Moser's argument.
First, we show that the Rabinowitz Floer homology vanishes when $d = 1$.
Next, we provide computations when hypersurfaces are quadrics $Q^n$, i.e.~$d = 2$.

Then, we provide computations for general hypersurfaces of degree $d = 3, \dots, n$.
Finally, we show that the Rabinowitz Floer homology also vanishes for $d = n + 1$.
We note that the map $\delta_c$ is nonzero in the case $d=1$ and $n=0$ and the case $d=n+1$.

\subsubsection{Case $d = 1$}

Let $\Sigma^n_1 \subset (\mathbb{CP}^{n + 1}, \omega_{\mathrm{FS}})$ denote a  smooth projective hypersurface of degree 1. In this case, $\Sigma_1^n$ and $\mathbb{CP}^{n + 1} \setminus \Sigma^n_1$ are symplectomorphic to $\mathbb{CP}^n$ and an open ball in $\mathbb{C}^{n+1}$. Thus we already know that $\SH_*(\mathbb{CP}^{n + 1} \setminus \Sigma^n_1)$ vanishes. We give an alternative proof using Theorem \ref{thm:gysin}.
We deal with the cases $n \geq 1$ and $n = 0$ separately.

Let $n = 0$. Then, $\Sigma^{0}_1 \subset \mathbb{CP}^1$ is a point, $\H_*(\Sigma^{0}_1;\Lambda)\cong \Lambda=\Z[T,T^{-1}]$ with $\deg T=2$, and both $\delta_a$ and $\delta_b$ vanish.
On the other hand, there is exactly one unparametrized holomorphic sphere passing through $\Sigma^{0}_1$, which contributes to the map $\delta_c$ in the Floer Gysin sequence as multiplication by $T^{-1}$, see \eqref{eq:delta_c}.
Therefore, $\delta_*=(\delta_c)_*$ in the Floer Gysin exact sequence is an isomorphism and 
\[
    \SH_*(\partial(\mathbb{CP}^1 \setminus \Sigma^{0}_1)) = 0.
\]

Let $n \geq 1$. The homology of $\Sigma^n_1\cong\mathbb{CP}^n$ is given by 
\[
    \H_{2n-2i+2(n+1)j} (\Sigma^n_1; \Lambda) = \mathbb{Z} \langle h^{\bullet i} T^j\rangle \quad  \textrm{for } i=0,\dots,n \text{ and } j\in \mathbb{Z}
\]
where $h$ stands for the hyperplane class, $\deg T=2(n+1)$, and the quantum intersection product relations are given by $ h^{*i}=h^{\bullet i}$ for $i=0,\dots,n$ and $h^{*(n+1)}=h^{\bullet 0}T^{-1}$, where $h^{\bullet 0}=[\Sigma^n_1]$, see \cite[Example 11.1.12]{MS}. In this case, $\delta_c$ vanishes by Lemma \ref{nonexistence}. Thus the connecting map $\delta_*$ in the Floer Gysin sequence \eqref{eq:gysin} is the quantum cap product with $-[\omega_{\mathrm{FS}}]$, or equivalently the quantum intersection product with $-h$.
Therefore, $\delta_*=(\delta_a)_*+(\delta_b)_*$ in the Floer Gysin exact sequence is an isomorphism and 
\[
    \SH_* (\partial(\mathbb{CP}^{n + 1} \setminus \Sigma^n_1)) = 0.
\]

\subsubsection{Case $d = 2$ with $n\geq 2$}

We now consider smooth hypersurfaces $\Sigma_2^n\subset \mathbb{CP}^{n+1}$ of degree 2 for $n\geq 2$. The case $n = 1$ will be treated in Section \ref{subsec: d = n +1} below.  We write $Q^n=\Sigma_2^n$ following the notation in Section \ref{sec:spheres}. 
Since $\deg T = 2n$, hence $\H_*(Q^n ; \Lambda)$ vanishes in odd degrees, and 
the Floer Gysin sequence splits into
\begin{equation}\label{quadric_seq2}
\begin{split}
	  0\to  \SH_{i+2}(\partial ( \mathbb{CP}^{n+1}\setminus Q^n ))\rightarrow &\H_{i+2+n}(Q^n;\Lambda) \\[.5ex]
	  &\stackrel{\delta_*}{\rightarrow} \H_{i+n}(Q^n;\Lambda) \rightarrow \SH_{i+1}(\partial ( \mathbb{CP}^{n+1}\setminus Q^n )) \rightarrow 0.
\end{split}
\end{equation}  
By Lemma \ref{nonexistence} or \eqref{eq:delta_c}, the map $\delta_c$ vanishes, and $\delta_*$ equals 
 the quantum cap product with $-2[\omega]$, or equivalently the quantum intersection product with $-2h$, where $h\in\H_{2n-2}(Q^n;\Z)$ denotes the hyperplane class as before. 
 
 Let $n=2m+1$ be odd. Adapting computations made in Section \ref{sec:spheres}, we obtain
\[
\begin{tikzcd}[row sep=1.5em,column sep=1.5em]
\H_{2k+2}(Q^n;\Lambda) \arrow{r}{*(-2h)} \arrow{d}{\cong} &  \H_{2k}(Q^n;\Lambda)   \arrow{d}{\cong}   \\
\mathbb{Z} \arrow{r}{\times (-2)} & \mathbb{Z}
\end{tikzcd}
\qquad 
\begin{tikzcd}[row sep=1.5em,column sep=1.5em]
\H_{2m+2+2jn}(Q^n;\Lambda) \arrow{r}{*(-2h)} \arrow{d}{\cong} &  \H_{2m+2jn}(Q^n;\Lambda)   \arrow{d}{\cong}   \\
\mathbb{Z} \arrow{r}{\times (-4)} & \mathbb{Z}
\end{tikzcd}
\]
for all $k\in\mathbb{Z}\setminus(n\mathbb{Z}+\{-1,0,m\})$ and for all $j\in\mathbb{Z}$. 
Then \eqref{quadric_seq2} yields
\begin{align*}
    \SH_{2k-n+1}(\partial ( \mathbb{CP}^{n+1}\setminus Q^n)) &\cong\mathbb{Z}_2\quad \forall k\in\mathbb{Z}\setminus (n\mathbb{Z}+\{-1,0,m\}), \quad  \\[.5ex]
    \SH_{2jn}(\partial ( \mathbb{CP}^{n+1}\setminus Q^n)) &\cong\mathbb{Z}_4 \quad \forall j\in\mathbb{Z}.        
\end{align*}
Similarly, we have 
\[
\begin{tikzcd}[row sep=1.5em,column sep=1.5em]
\H_{2+2jn}(Q^n;\Lambda) \arrow{r}{*(-2h)} \arrow{d}{\cong} &  \H_{2jn}(Q^n;\Lambda)   \arrow{d}{\cong}  \arrow{r}{*(-2h)} & \H_{2jn-2}(Q^n;\Lambda)  \arrow{d}{\cong}    \\
\mathbb{Z} \arrow{r}{1\mapsto(-2,-2)} & \mathbb{Z}\oplus\mathbb{Z} \arrow{r}{(c,d)\mapsto -2(c+d)} & \mathbb{Z}
\end{tikzcd}
\]
and  \eqref{quadric_seq2} implies that for all $j \in \mathbb Z$,
\begin{align*}
    \SH_{2jn-n+1}(\partial ( \mathbb{CP}^{n+1}\setminus Q^n)) &\cong\mathbb{Z}\oplus\mathbb{Z}_2, \\[.5ex]
    \SH_{2jn-n}(\partial ( \mathbb{CP}^{n+1}\setminus Q^n)) &\cong\mathbb{Z}, \\[.5ex]
    \SH_{2jn-n-1}(\partial ( \mathbb{CP}^{n+1}\setminus Q^n)) &\cong\mathbb{Z}_2.    
\end{align*}
We have computed for odd $n\geq2$
\[
\SH_i (\partial ( \mathbb{CP}^{n+1}\setminus Q^n)) \cong
\begin{cases}
	\mathbb{Z}_4, & i\in 2n\mathbb{Z}, \\[0.5ex]
	\mathbb{Z}_2, & i\in 2\mathbb{Z}\setminus(2n\mathbb{Z}+\{0, n+1\}), \\[0.5ex]
	\mathbb{Z}\oplus\mathbb{Z}_2, & i\in 2n\mathbb{Z}+n+1, \\[0.5ex]
	\mathbb{Z}, & i\in 2n\mathbb{Z}+n, \\[0.5ex]
	0 & \textrm{otherwise}.
\end{cases}
\]
In particular, $\SH_* (\partial ( \mathbb{CP}^{n+1}\setminus Q^n))$ is nonzero for all even degrees.

Next, let $n=2m$ be even. Due to computations in Section \ref{sec:spheres}, we have
\[
\begin{tikzcd}[row sep=1.5em,column sep=1.5em]
    \H_{2k+2}(Q^n;\Lambda) \arrow{r}{*(-2h)} \arrow{d}{\cong} &  \H_{2k}(Q^n;\Lambda)   \arrow{d}{\cong}   \\
    \mathbb{Z} \arrow{r}{\times (-2)} & \mathbb{Z}
\end{tikzcd}
\]    
for all $k \in \mathbb{Z} \setminus (m \mathbb{Z}+\{-1, 0\})$. Then \eqref{quadric_seq2} yields
\[
\SH_{2k-n+1} (\partial ( \mathbb{CP}^{n+1}\setminus Q^n)) \cong\mathbb{Z}_2\qquad \forall k \in  \mathbb{Z}\setminus (m\mathbb{Z}+\{-1,0\})
\]
In addition, the computation
\[
\begin{tikzcd}[row sep=1.5em,column sep=1.5em]
\H_{2+jn}(Q^n;\Lambda) \arrow{r}{*(-2h)} \arrow{d}{\cong} &  \H_{jn}(Q^n;\Lambda)   \arrow{d}{\cong}  \arrow{r}{*(-2h)} & \H_{jn-2}(Q^n;\Lambda)  \arrow{d}{\cong}    \\
\mathbb{Z} \arrow{r}{1\mapsto(-2,-2)} & \mathbb{Z}\oplus\mathbb{Z} \arrow{r}{(c,d)\mapsto -2(c+d)} & \mathbb{Z}
\end{tikzcd}
\]
and \eqref{quadric_seq2} imply that for all $j \in \mathbb Z$,
\begin{align*}
    \SH_{jn-n+1} (\partial ( \mathbb{CP}^{n+1}\setminus Q^n)) &\cong\mathbb{Z}\oplus\mathbb{Z}_2,\\[0.5ex]
    \SH_{jn-n}(\partial ( \mathbb{CP}^{n+1}\setminus Q^n)) &\cong\mathbb{Z},\\[0.5ex]
    \SH_{jn-n-1} (\partial ( \mathbb{CP}^{n+1}\setminus Q^n)) &\cong\mathbb{Z}_2.
\end{align*}
We have computed for even $n\geq2$
\[
\SH_i (\partial ( \mathbb{CP}^{n+1}\setminus Q^n)) \cong
\begin{cases}
	\mathbb{Z}_2, & i\in 2\mathbb{Z} + 1 \setminus(n\mathbb{Z}+1), \\[0.5ex]
	\mathbb{Z}\oplus\mathbb{Z}_2, & i\in n\mathbb{Z}+1, \\[0.5ex]
	\mathbb{Z}, & i\in n\mathbb{Z}, \\[0.5ex]
	0 & \textrm{otherwise}.
\end{cases}
\]
In particular, $\SH_* (\partial ( \mathbb{CP}^{n+1}\setminus Q^n))$ is nonzero for all odd degrees.

\subsubsection{Case $2 \leq d \leq n$ with $n  \geq 2$}

Now we consider projective hypersurfaces $\Sigma^n_d\subset \mathbb{CP}^{n+1}$ of degree $d\in\{2,\dots,n\}$.
To simplify computations, we use $\mathbb{Q}$-coefficients below. Since $\SH(\cdot;\mathbb{Q})\cong\SH(\cdot)\otimes_\Z\mathbb{Q}$ and tensoring $\mathbb{Q}$ is an exact functor, we have the Floer Gysin sequence with $\mathbb{Q}$-coefficients.  Note that the case $d=2$ was carried out with $\Z$-coefficients in the previous section.

The cohomology ring $\H_*(\Sigma^n_d;\mathbb{Q})$ is generated by the hyperplane section class $h\in\H_{2n-2}(\Sigma^n_d;\mathbb{Q})$ and primitive classes in $\H_n(\Sigma^n_d;\mathbb{Q})_0$, i.e.~$\H_n(\Sigma^n_d;\mathbb{Q})_0$ is the kernel of $\bullet h$. 
In particular $\H_n(\Sigma^n_d;\mathbb{Q})=\H_n(\Sigma^n_d;\mathbb{Q})_0 \oplus V$ where $V=0$ if $n$ is odd and $V=\mathbb{Q}\langle h^{n/2}\rangle$ otherwise. 
Since $c_1^{T\Sigma^n_d}=(n-d+2)[\omega_\mathrm{FS}|_{\Sigma^n_d}]$, the formal parameter $T$ in $\Lambda^\mathbb{Q}=\mathbb{Q}[T,T^{-1}]$ has degree $2(n-d+2)$. 
The monotonicity constant $(n-d+2)$ is at least $2$ due to the assumption $2\leq d\leq n$, and thus $\delta_c$ vanishes by Lemma \ref{nonexistence}. 
We abbreviate 
\begin{equation}\label{eq:def_b}
b_i:=\mathrm{rank}_\mathbb{Q}\, \H_{i}(\Sigma^n_d;\Lambda^\mathbb{Q}), \quad b':=\mathrm{rank}_\mathbb{Q}\,\H_n(\Sigma^n_d;\mathbb{Q})_0.
\end{equation}
The rank of $\H_n(\Sigma^n_d;\mathbb{Q})$ can be deduced from the Euler characteristic $\Sigma^n_d$, see \cite{EM76}, and all $b_i$ can be explicitly computed. Due to the periodicity given by multiplying $T^{\pm 1}$, we  have $b_i=b_{i+ 2(n-d+2)}$ for all $i\in\mathbb{Z}$.
The quantum (co)homology ring of $\Sigma^n_d$ with $\mathbb{Q}$-coefficients was computed in \cite{Giv96,CJ99}.
We recall relations that are relevant to our purpose:
\begin{equation}\label{eq:rational_h}
h^{*(n+1)}=d^dh^{*(d-1)}T^{-1},\qquad a*h=0 \quad \forall a\in \H_n(\Sigma^n_d;\mathbb{Q})_0.
\end{equation}

We first consider the case that $n$ is even. Then $\H_*(\Sigma^n_d;\Lambda^\mathbb{Q})$ vanishes on all odd degrees, and thus it suffices to study the map $\delta_* =(\delta_a)_*+(\delta_b)_*= *(-dh)$ in  
\begin{equation}\label{rational}
\begin{split}
	  0 \to\;   \SH_{2i-n}(\partial(\mathbb{CP}^{n+1}\setminus \Sigma^n_d);\mathbb{Q})\rightarrow \H_{2i}(\Sigma^n_d;&\,\Lambda^\mathbb{Q})\stackrel{\delta_*}{\rightarrow} \H_{2i-2}(\Sigma^n_d;\Lambda^\mathbb{Q}) \\[1ex]
	   & \rightarrow \SH_{2i-n-1}(\partial(\mathbb{CP}^{n+1}\setminus \Sigma^n_d);\mathbb{Q}) \rightarrow 0.
\end{split}
\end{equation} 
The exact sequence \eqref{rational} directly implies
\begin{align}
    \begin{split}\label{eq:rank}
    \mathrm{rank}_\mathbb{Q} \SH_{2i-n}(\partial(\mathbb{CP}^{n+1}\setminus \Sigma^n_d);\mathbb{Q})  &= \mathrm{rank}_\mathbb{Q} \ker\left( \delta_* |_{\H_{2i}(\Sigma^n_d; \,\Lambda^\mathbb{Q})}\right)=:c_{2i}, \\
    \mathrm{rank}_\mathbb{Q} \SH_{2i-n-1}(\partial(\mathbb{CP}^{n+1}\setminus \Sigma^n_d);\mathbb{Q})&= b_{2i-2}-b_{2i}+c_{2i}.
    \end{split}
\end{align}
Below $x\equiv y$ means that $x$ equals $y$ modulo $2(n-d+2)$, and similarly for $x\not\equiv y$. 
Using the relations in \eqref{eq:rational_h}, we compute
\begin{align}\label{eq:rank_even}
    (c_{2i}, b_{2i-2}-b_{2i})= \begin{cases}
        (b'+1,-b'-1), & 2i \equiv n \equiv 0, \\ 
        (1,b'-1), & 2i \equiv n+2 \equiv 0, \\
        (1,0), & 2i \equiv 2n+2 \equiv 0, \\
        (b',-b'+1), & 2i \equiv 2n+2 \equiv n, \\
        (0,b'+1), & 2i \equiv 2n+2 \equiv n+2,\\
        (1,-1), & 2i \equiv 0,\; 2i\not\equiv n,\; 2i\not\equiv n+2,\;2i\not\equiv 2n+2, \\
        (b',-b'), & 2i \equiv n,\; 2i\not\equiv 0,\; 2i\not\equiv n+2,\;2i\not\equiv 2n+2,\\
        (0,b'), & 2i \equiv n+2,\; 2i\not\equiv 0,\; 2i\not\equiv n,\;2i\not\equiv 2n+2,\\
        (0,1), & 2i \equiv 2n+2,\; 2i\not\equiv0,\;2i\not\equiv n,\;2i\not\equiv n+2,\\
        (0,0) &\text{otherwise}.
    \end{cases}
\end{align} 
Note that since $n\geq d$, or equivalently $2(n-d+2)\geq 4$, the case $n\equiv n+2$ does not occur, and any three of $\{0,n,n+2,2n+2\}$ cannot be simultaneously the same modulo $2(n-d+2)$.

From \eqref{eq:rank} and \eqref{eq:rank_even}, we conclude for even $n$
\[
    \SH_{j}(\partial(\mathbb{CP}^{n+1}\setminus \Sigma^n_d);\mathbb{Q})\cong\begin{cases}
        \mathbb{Q}^{b'}, & j \in 2(n-d+2)\Z+\{0,1\}, \\[0.5ex]
        \mathbb{Q}, &  j \in 2(n-d+2)\Z +\{-n,n+1\}, \\[0.5ex]
        0 & \text{otherwise},
    \end{cases}  
\]
where, in the case of $n \in 2(n+d-2)\Z$, this should read $\SH_{j}(\partial(\mathbb{CP}^{n+1}\setminus \Sigma^n_d);\mathbb{Q})\cong \mathbb{Q}^{b'+1}$ for $j \in 2(n-d+2)\Z+\{0,1\}$ and $\SH_{j}(\partial(\mathbb{CP}^{n+1}\setminus \Sigma^n_d);\mathbb{Q})=0$ otherwise. 
\medskip

Next we treat the case that $n$ is odd. In this case, $\H_{i}(\Sigma^n_d;\Lambda^\mathbb{Q})=0$ for all odd $i$ except  $i\equiv n$. We look at the following piece of the long exact sequence.
\begin{equation}\label{eq:seq_odd}
\begin{split}
&\H_{2i+1}(\Sigma^n_d;\Lambda^\mathbb{Q})\to \H_{2i-1}(\Sigma^n_d;\Lambda^\mathbb{Q})\to \SH_{2i-n}(\partial(\mathbb{CP}^{n+1}\setminus \Sigma^n_d);\mathbb{Q})\to \H_{2i}(\Sigma^n_d;\Lambda^\mathbb{Q}) \stackrel{\delta_*}{\rightarrow} \\[1ex]
&\quad \H_{2i-2}(\Sigma^n_d;\Lambda^\mathbb{Q}) 
\to \SH_{2i-n-1}(\partial(\mathbb{CP}^{n+1}\setminus \Sigma^n_d);\mathbb{Q}) \to \H_{2i-1}(\Sigma^n_d;\Lambda^\mathbb{Q})\to \H_{2i-3}(\Sigma^n_d;\Lambda^\mathbb{Q}).
\end{split}
\end{equation}
If $2i\not\equiv n+1$, then $\H_{2i-1}(\Sigma^n_d;\Lambda^\mathbb{Q})=0$ and we again have \eqref{rational} and \eqref{eq:rank}. If $2i\equiv n+1$, then $\H_{2i+1}(\Sigma^n_d;\Lambda^\mathbb{Q})=0$ and we get
\begin{align}
    \begin{split}\label{eq:rank2}
    \mathrm{rank}_\mathbb{Q} \SH_{2i-n}(\partial(\mathbb{CP}^{n+1}\setminus \Sigma^n_d);\mathbb{Q})  &= b'+c_{2i}, \\
    \mathrm{rank}_\mathbb{Q} \SH_{2i-n-1}(\partial(\mathbb{CP}^{n+1}\setminus \Sigma^n_d);\mathbb{Q})&= b'+b_{2i-2}-b_{2i}+c_{2i}.
    \end{split}
\end{align}
Using the relations in \eqref{eq:rational_h} again, we compute 
\begin{align}\label{eq:rank_odd}
    (c_{2i}, b_{2i-2}-b_{2i})= \begin{cases}
        (1,0), & 2i \equiv n+1 \equiv 0, \\ 
        (0,0), & 2i \equiv n+1 \not\equiv 0, \\
        (1,-1), & 2i \equiv 0 \not\equiv 2n+2, \\
        (0,1), & 2i \equiv 2n+2 \not\equiv 0, \\
        (1,0), & 2i \equiv 0\equiv 2n+2, \\
        (0,0) & \text{otherwise}.
    \end{cases}
\end{align} 
Hence, using \eqref{eq:rank_odd} together with \eqref{eq:rank} and \eqref{eq:rank2}, we conclude for odd $n$
\[
    \SH_{j}(\partial(\mathbb{CP}^{n+1}\setminus \Sigma^n_d);\mathbb{Q})\cong\begin{cases}
        \mathbb{Q}^{b'}, & j \in 2(n-d+2)\Z+\{0,1\}, \\[0.5ex]
        \mathbb{Q}, &  j \in 2(n-d+2)\Z +\{-n,n+1\}, \\[0.5ex]
        0 & \text{otherwise}.
    \end{cases}  
\]
As before, in the case of $n+1 \in 2(n-d+2)\Z$, this should read $\SH_{j}(\partial(\mathbb{CP}^{n+1}\setminus \Sigma^n_d);\mathbb{Q})\cong \mathbb{Q}^{b'+1}$ for $j \in 2(n-d+2)\Z+\{0,1\}$ and $\SH_{j}(\partial(\mathbb{CP}^{n+1}\setminus \Sigma^n_d);\mathbb{Q})=0$ otherwise.
\subsubsection{Case $d = n + 1$} \label{subsec: d = n +1}

Finally, we consider projective hypersurfaces $\Sigma^n_{n+1}$ of degree $n+1$ in $\mathbb{CP}^{n+1}$ and show that the Rabinowitz Floer homologies of the complements with $\mathbb{Q}$-coefficients vanish.
As above, we denote $\H_*(\Sigma^n_{n+1}; \mathbb Q) \cong \H_*(\Sigma^n_{n+1}; \mathbb Q)_{\mathrm H} \oplus \H_n(\Sigma^n_{n+1}; \mathbb Q)_0$, where the former is generated by the hyperplane section class $h \in \H_{2n - 2}(\Sigma^n_{n+1})$ and the latter by primitive classes.
Since $K = n+1$, $\tau_X = n +2$, we have $\tau_\Sigma = 1$ and the formal parameter $T$ in $\Lambda^{\mathbb Q} = \mathbb Q [T, T^{-1}]$ has degree $2$. Thus all $\H_j(\Sigma^n_{n + 1}; \Lambda^{\mathbb Q})$ for $j\in 2\Z$ have the same rank, and likewise for $j\in 2\Z+1$. 
The quantum relations discovered in \cite{Giv96,CJ99} relevant to our computation are  
\begin{align}\label{eq: quantum_relations_0}
\begin{split}
    &(h + (n+1)!T^{-1}h^{\bullet 0})^{*n} * (h + ((n + 1)! - (n + 1)^{n + 1})T^{-1} h^{\bullet 0}) = 0, \\[0.5ex]
    &h * a = -(n + 1)! T^{-1} a \quad \forall a\in \H_n(\Sigma^n_d;\mathbb{Q})_0.
\end{split}
\end{align}
Here $h^{\bullet 0}=[\Sigma^n_{n+1}]$ is the fundamental class. Following \cite[Cor 7.10]{Sh16}, we set 
\begin{align*}
    x_i &:= (h + (n+1)!T^{-1}h^{\bullet 0})^{*i} * (h + ((n + 1)! - (n + 1)^{n + 1})T^{-1}h^{\bullet 0} ), \quad i = 0, \dots, n - 1, \\[0.5ex]
    y &:= (h + (n+1)!T^{-1}h^{\bullet 0})^{*n},
\end{align*}
where $\deg x_i=2n-2i-2$ and $\deg y=0$. 
These form a basis of $\H_*(\Sigma^n_{n+1}; \Lambda^{\mathbb Q})_{\mathrm H}$ over $\Lambda^{\mathbb Q}$.
Then \eqref{eq: quantum_relations_0} read
\begin{align}\label{eq: quantum_relations}
\begin{split}
    &h * x_i = -(n + 1)! T^{-1} x_i + x_{i + 1} \quad\text{ for } i = 0, \dots, n-2, \\[0.5ex]
    &h * x_{n - 1} = -(n + 1)! T^{-1} x_{n - 1}, \\[0.5ex]
    &h * y = ((n + 1)^{n + 1} - (n + 1)!)T^{-1} y.
\end{split}
\end{align}
For $n = 1$, we ignore the first line of \eqref{eq: quantum_relations}.
Since $e_Y = -(n + 1) [\omega_\mathrm{FS}|_{\Sigma^n_{n+1}}]$, the map  $(\delta_a)_* + (\delta_b)_* $ is the quantum intersection product with $-(n + 1)h $, see Lemma \ref{lemma:quantum_cap}. Therefore, due to \eqref{eq: quantum_relations}, $(\delta_a)_* + (\delta_b)_* $ restricted to ${\H_*(\Sigma^n_{n + 1}; \Lambda^{\mathbb Q})_{\mathrm{H}}}$ is lower-triangular with nonzero diagonal entries with respect to the basis formed by $x_i$ and $y$ multiplied with appropriate powers of $T$:
\[
    ((\delta_a)_* + (\delta_b)_* ) |_{\H_*(\Sigma^n_{n + 1}; \Lambda^{\mathbb Q})_{\mathrm{H}}} = \begin{pmatrix}
        \alpha_1 & & & & &\\
        \alpha_2 & \alpha_1 & &  & &\\
         & \alpha_2 & \ddots & & &\\
         & & \ddots & \alpha_1 & &\\
         & & & \alpha_2 & \alpha_1 & \\
         & & & & &\beta
    \end{pmatrix}, \qquad \begin{array}{l}
        \alpha_1 := (n + 1)(n + 1)!, \\[1ex] \alpha_2 := -(n + 1), \\[1ex] \beta := (n+1)(n+1)! \\[.5ex] \quad\quad\; - (n + 1)^{n + 2}.
    \end{array} 
\]
By the second line of \eqref{eq: quantum_relations_0}, we also have $( (\delta_a)_* + (\delta_b)_* ) |_{\H_*(\Sigma^n_{n + 1}; \Lambda^{\mathbb Q})_0} = (n + 1)(n + 1)! \, T^{-1}\mathrm{id}$.  

\medskip

In contrast to projective hypersurfaces of degree $ d\in\{2,\dots, n\}$, the map  $\delta_c$ is nonzero. Indeed, 
for $B = [\mathbb{CP}^1] \in \H_*(\mathbb{CP}^{n + 1})$, we have
\begin{align}\label{eq:sphere_tangency}
        \#\Big( \mathcal{N}^*_{N=1, \ell = 1} (p,p;A=0,B = [\mathbb{CP}^1]; i \, ) / \mathbb{C}^* \Big) = \# \big ( \mathcal{P}^*_{\ell = 1} ([\mathbb{CP}^1]; i \, ) \big ) 
        = n!
\end{align}
where the first equality is given by a direct calculation using fibered sum orientation rule, and the second equality is computed in \cite[Prop 3.4]{CM18}. The class $B = [\mathbb{CP}^1]$ is the only class satisfying $c_1^{TX}(B)-K\omega(B) = \omega(B) = 1$. 
Due to \eqref{eq:delta_c} and \eqref{eq:sphere_tangency}, we have
\begin{equation*}\label{eq: Q_delta_c}
    \delta_c = (n + 1)! \,T^{-1}\mathrm{id}.
\end{equation*}
Hence, $\delta_*=(\delta_a)_*+(\delta_b)_*+(\delta_c)_*$ restricted to $\H_*(\Sigma^n_{n + 1}; \Lambda^{\mathbb Q})_{\mathrm{H}}$ is also lower-triangular with nonzero diagonal entries:
\[
    \delta_* |_{\H_*(\Sigma^n_{n + 1}; \Lambda^{\mathbb Q})_{\mathrm{H}}} = \begin{pmatrix}
        \alpha'_1 & & & & &\\
        \alpha_2 & \alpha'_1 & & & &\\
         & \alpha_2 & \ddots & & &\\
         & & \ddots & \alpha'_1 & &\\
         & & & \alpha_2 & \alpha'_1 & \\
         & & & & &\beta'
    \end{pmatrix}, \qquad \begin{array}{l}
        \alpha'_1 := (n + 2)!, \\[0.5ex]
        \alpha_2 := -(n + 1), \\[0.5ex]
        \beta' := (n + 2)! - (n + 1)^{n + 2}.
    \end{array} 
\]
Moreover, $\delta_* |_{\H_*(\Sigma^n_{n + 1}; \Lambda^{\mathbb Q})_{0}} = (n + 2)! \, T^{-1}\mathrm{id}$. Therefore $\delta_*$ is an isomorphism in every degree, and the Floer Gysin sequence in Theorem \ref{thm:gysin} readily implies 
\begin{equation}\label{eq:n+1_degree}
 \SH_* (\partial ( \mathbb{CP}^{n + 1} \setminus \Sigma^n_{n + 1} ) ;\mathbb{Q}) = 0.
\end{equation}
Note that this  implies that the symplectic homology of $\mathbb{CP}^{n + 1} \setminus \Sigma^n_{n + 1}$ with $\mathbb{Q}$-coefficients  vanishes as well. We also point out that the above computation shows that the Rabinowitz Floer homology of $\partial ( \mathbb{CP}^{n + 1} \setminus \Sigma^n_{n + 1} )$ and thus also the symplectic homology  of $\mathbb{CP}^{n + 1} \setminus \Sigma^n_{n + 1}$ do not vanish with coefficients in $\Z$ or $\Z_2$.

\subsection{Unit cotangent bundle of $\mathbb{CP}^2$}\label{sec: cotangent}

In the unpublished note \cite{Sei10}, Seidel showed that the symplectic homology of $T^*\mathbb{CP}^2$ with $\mathbb{Q}$-coefficients  vanishes.
In this section, we provide an alternative proof of this fact using the Floer Gysin exact sequence.
Consider $\mathbb{CP}^2 \times \mathbb{CP}^2$ with the symplectic form $\omega = \omega_{\mathrm{FS}} \oplus \omega_{\mathrm{FS}}$ and let 
\[
    F_3 := \left \{ \big([z_0: z_1: z_2], [w_0: w_1: w_2]\big) \in \mathbb{CP}^2 \times \mathbb{CP}^2 \, \mid \, \sum_{i = 0}^{2} z_i w_i = 0 \right \}.
\]
Then, $(F_3, \omega |_{F_3})$ is a symplectic submanifold representing $\PD([\omega])$ in $\H_*(\mathbb{CP}^2 \times \mathbb{CP}^2;\Z)$ and can be viewed as the full flag manifold consisting of all nested sequences of subspaces in $\mathbb{C}^3$, see \cite[Section 11.3.2]{MS}.
The singular homology of $F_3$ with $\Z$-coefficients is given by
\[
    \H_i(F_3;\Z)\cong 
        \begin{cases}
            \mathbb{Z} \langle [\mathrm{pt}] \rangle, & i = 0, \\[0.5ex]
            \mathbb{Z} \langle a^{\bullet 2}, b^{\bullet 2} \rangle, & i = 2, \\[0.5ex]
            \mathbb{Z} \langle a, b \rangle, & i = 4, \\[0.5ex]
            \mathbb{Z} \langle [F_3] \rangle, & i = 6, 
        \end{cases}    
\]
with relations 
\[
a \bullet b = a^{\bullet 2} + b^{\bullet 2},\qquad a^{\bullet 3} = b^{\bullet 3} = 0, \qquad a^{\bullet 2} \bullet b = a \bullet b^{\bullet 2} = [\mathrm{pt}].
\]
Since $c_1^{T(\mathbb{CP}^2 \times \mathbb{CP}^2)} = 3[\omega]$ and $[F_3]=\PD([\omega])$, we have $c_1^{TF_3}  = 2 \, [\omega |_{F_3}]$. Thus the minimal Chern number of $F_3$ is equal to $2$, and the variable $T$ in $\Lambda = \mathbb{Z} [T, T^{-1}]$ has $\deg T = 4$. The quantum (co)homology structures of full flag manifolds are revealed in \cite{GK95,Kim99}, see also \cite[Example 8.2]{Gu05} and \cite[Section 11.3.2]{MS}.
As $[\omega |_{F_3}] = \PD(a + b)$, quantum product relations relevant to our purpose are the following.
\begin{equation}\label{eq:relation_F3}
\begin{split}
	  &[F_3] * a = a, \quad [F_3] * b = b, \\[0.5ex]
    &a * a = a^{\bullet 2} + [F_3] T^{-1}, \quad a * b = a \bullet b, \quad b * b = b^{\bullet 2} + [F_3] T^{-1},\\[0.5ex]
    &a^{\bullet 2} * a = b T^{-1}, \quad a^{\bullet 2} * b = b^{\bullet 2} * a = [\mathrm{pt}], \quad b^{\bullet 2} * b = a T^{-1}, \\[0.5ex]
    &[\mathrm{pt}] * a = b^{\bullet 2} T^{-1} + [F_3] T^{-2}, \quad [\mathrm{pt}] * b = a^{\bullet 2} T^{-1} + [F_3] T^{-2}.
\end{split}
\end{equation}

The complement $\mathbb{CP}^2 \times \mathbb{CP}^2 \setminus F_3$ is symplectomorphic to a disk cotangent bundle $D^*\mathbb{CP}^2$ of $\mathbb{CP}^2$, see \cite{Aud07}.
The Rabinowitz Floer homology $\SH_* (S^* \mathbb{CP}^2)$ of the unit cotangent bundle $S^* \mathbb{CP}^2$ of $\mathbb{CP}^2$ and the associated Floer Gysin exact sequence exist without reference to Liouville fillings due to Corollary  \ref{cor:gysin_symplectization}, see also Remark \ref{rem:indep_filling}. 
 Since $\H_{*}(F_3;\Lambda)$ vanishes in every odd degree, the Floer Gysin exact sequence splits as
 \begin{align*}
    &0\to  \SH_{i-1}(S^*\mathbb{CP}^2) \rightarrow \H_{i+2}(F_3;\Lambda)\xrightarrow{\delta_*} \H_{i}(F_3;\Lambda) \rightarrow \SH_{i-2} (S^*\mathbb{CP}^2) \rightarrow 0. 
\end{align*}
Note that it suffices to consider the above two cases $i=2,4$ since the Floer Gysin sequence is $4$-periodic in degree as $\deg T = 4$. 
By \eqref{eq:relation_F3}, the map $\delta_* = *( - a - b)$ is given by
\begin{equation*}
    \delta_* |_{\H_6 (F_3;\Lambda)} = \begin{pmatrix}
        -1 & -1 & 0 \\
        0 & -1 & -1 \\
        -1 & 0 & -1 
    \end{pmatrix},  \qquad \delta_* |_{\H_4 (F_3;\Lambda)} = \begin{pmatrix}
        -1 & -2 & -1 \\
        -1 & -1 & -2 \\
        -2 & -1 & -1
    \end{pmatrix},
\end{equation*}
with respect to the base $\left \{ a^{\bullet 2} T, b^{\bullet 2} T, [F_3] \right \}$ and $\left \{ [\mathrm{pt}]T, a, b \right \}$ for the former one and $\left \{ [\mathrm{pt}]T, a, b \right \}$ and $\left \{ a^{\bullet 2} , b^{\bullet 2} , [F_3] T^{-1} \right \}$ for the latter one.
Both are isomorphisms in $\mathbb{Q}$-coefficients. Hence,
\[
\SH_*(S^*\mathbb{CP}^2;\mathbb{Q})=0,
\] 
and the symplectic homology of $T^* \mathbb{CP}^2$ with $\mathbb{Q}$-coefficients vanishes as well.

\begin{remark}
 The case of $T^*\mathbb{CP}^2$ is one instance that the isomorphism between the symplectic homology of a cotangent bundle $T^*N$ and the homology of the free loop space of $N$ constructed in \cite{Vit96,AS06,SW06}  holds after twisting one side by local coefficients (except for $\mathbb{Z}_2$-coefficients), see \cite{Abo15}.

    Similar phenomenon also happens for $\mathbb{RP}^2$. The complement $\mathbb{CP}^{n+1}\setminus\Sigma^n_{2}$ is symplectomorphic to a disk cotangent bundle over $\mathbb{RP}^{n+1}$, see e.g.~\cite{Aud07}. Therefore \eqref{eq:n+1_degree} implies 
    \[
	\SH_*(S^*\mathbb{RP}^2;\mathbb{Q})=0.
	\] 
	and the symplectic homology of $T^*\mathbb{RP}^2$ with $\mathbb{Q}$-coefficients vanishes.
\end{remark}

\paragraph{Acknowledgments} 
This work is supported by National Research Foundation of Korea grant NRF-2020R1A5A1016126 and RS-2023-00211186. The research of Sungho Kim is supported by the POSCO Science Fellowship of POSCO TJ Park Foundation. We are grateful to Peter Albers, Urs Frauenfelder, and Otto van Koert for stimulating discussions. We also thank Lu\'is Diogo and Egor Shelukhin for their feedback.

\paragraph{Data Availability} Data sharing not applicable to this article as no data sets were generated or analyzed during the current study.
\paragraph{Conflict of interest} On behalf of all authors, the corresponding author states that there is no conflict of interest.

\bibliographystyle{amsalpha}
\bibliography{./ref.bib}

\end{document}